\documentclass{amsart}
\usepackage{amssymb}
\usepackage{amsfonts}
\usepackage{geometry}

\newtheorem{theorem}{Theorem}
\theoremstyle{plain}

\newtheorem{conjecture}[theorem]{Conjecture}

\newtheorem{definition}[theorem]{Definition}

\newtheorem{lemma}[theorem]{Lemma}

\newtheorem{proposition}[theorem]{Proposition}

\numberwithin{equation}{section}
\input{tcilatex}
\begin{document}
\title{A counterexample in the two weight theory for Calder\'{o}n-Zygmund
operators}
\author[E.T. Sawyer]{Eric T. Sawyer}
\address{ Department of Mathematics \& Statistics, McMaster University, 1280
Main Street West, Hamilton, Ontario, Canada L8S 4K1 }
\email{sawyer@mcmaster.ca}
\thanks{Research supported in part by NSERC}
\author[C.-Y. Shen]{Chun-Yen Shen}
\address{ Department of Mathematics \\
National Taiwan University \\
10617, Taipei}
\email{cyshen@math.ntu.edu.tw}
\thanks{C.-Y. Shen supported in part by MOST, through grant 104-2628-M-002
-015 -MY4}
\author[I. Uriarte-Tuero]{Ignacio Uriarte-Tuero}
\address{ Department of Mathematics \\
Michigan State University \\
East Lansing MI }
\email{ignacio@math.msu.edu}
\thanks{ I. Uriarte-Tuero has been partially supported by grant
MTM2015-65792-P (MINECO, Spain).}
\date{July 15, 2016}

\begin{abstract}
We give an example of a pair of weights $\left( \widehat{\sigma },\widehat{%
\omega }\right) $ on the line, and an elliptic singular integral operator $%
H_{\flat }$ on the line, such that $H_{\flat ,\sigma }$ is bounded from $%
L^{2}\left( \widehat{\sigma }\right) $ to $L^{2}\left( \widehat{\omega }%
\right) $, yet the measure pair $\left( \widehat{\sigma },\widehat{\omega }%
\right) $ fails to satisfy one of the energy conditions. The convolution
kernel $K_{\flat }\left( x\right) $ of the operator $H_{\flat }$ is a smooth
flattened version of the Hilbert transform kernel $K\left( x\right) =\frac{1%
}{x}$ that satisfies ellipticity $\left\vert K_{\flat }\left( x\right)
\right\vert \gtrsim \frac{1}{\left\vert x\right\vert }$, but not gradient
ellipticity $\left\vert K_{\flat }^{\prime }\left( x\right) \right\vert
\gtrsim \frac{1}{\left\vert x\right\vert ^{2}}$. Indeed the kernel has flat
spots where $K_{\flat }^{\prime }\left( x\right) =0$ on a family of
intervals, but $K_{\flat }^{\prime }\left( x\right) $ is otherwise negative
on $\mathbb{R}\setminus \left\{ 0\right\} $. On the other hand, if a
one-dimensional kernel $K\left( x,y\right) $ is both elliptic and gradient
elliptic, then the energy conditions are necessary, and so by our theorem in 
\cite{SaShUr10}, the $T1$ theorem holds for such kernels on the line.
\end{abstract}

\maketitle
\tableofcontents

\section{Introduction}

This paper addresses the main obstacle, namely the energy condition, arising
in the theory of two weight norm inequalities for operators with
cancellation (singular integrals) in the aftermath of the solution to the $%
T1 $ conjecture for the Hilbert transform in the two part paper \cite%
{LaSaShUr3}, \cite{Lac} by the authors and M. Lacey (see also \cite{Hyt2}).
But before putting matters into perspective, it will be useful to briefly
review the history of weighted norm inequalities for the Hilbert transform.
For a signed measure $\nu $ on $\mathbb{R}$ define 
\begin{equation}
H\nu \left( x\right) \equiv \QTR{up}{p.v.}\int \frac{1}{x-y}\;\nu (dy)\,,
\label{e.H}
\end{equation}%
for an appropriate truncation of the kernel $\frac{1}{x-y}$ - see below. A 
\emph{weight} $\omega $ is a non-negative locally finite Borel measure.

The one weight inequality for the Hilbert transform is then%
\begin{equation*}
\left\Vert Hf\right\Vert _{L^{2}(\omega )}\lesssim \left\Vert f\right\Vert
_{L^{2}(\omega )},
\end{equation*}%
and was shown by Hunt, Muckenhoupt and Wheeden in \cite{HuMuWh} to be
equivalent to finiteness of the remarkable $A_{2}$ condition of Muckenhoupt:
namely that $d\omega =w\left( x\right) dx$ is absolutely continuous with
respect to Lebesgue measure and%
\begin{equation*}
\sup_{I}\frac{1}{\left\vert I\right\vert }\int_{I}w(x)dx\cdot \frac{1}{%
\left\vert I\right\vert }\int_{I}\frac{1}{w(x)}dx<\infty \,,
\end{equation*}%
which says that the Cauchy-Schwarz inequality%
\begin{equation*}
1=\left( \frac{1}{\left\vert I\right\vert }\int_{I}\sqrt{w(x)\frac{1}{w(x)}}%
dx\right) ^{2}\leq \frac{1}{\left\vert I\right\vert }\int_{I}w(x)dx\cdot 
\frac{1}{\left\vert I\right\vert }\int_{I}\frac{1}{w(x)}dx
\end{equation*}%
can be reversed up to a constant uniformly over intervals $I$.

For two weights $\omega ,\sigma $, we consider the two weight norm
inequality 
\begin{equation}
\left\Vert H(f\sigma )\right\Vert _{L^{2}(\omega )}\leq \mathfrak{N}%
\left\Vert f\right\Vert _{L^{2}(\sigma )}.\ \ \ \ \ \ \ \ \ \ \ \ \ \ \ \ \
\ \ \ \left( \mathfrak{N}\right)  \label{e.H<}
\end{equation}%
Note that when $\omega $ is absolutely continuous with density $w$, and both 
$w$ and $\frac{1}{w}$ are locally integrable, then the case $\sigma =\frac{1%
}{w}dx$ reduces to the one weight inequality above. A simple necessary
condition for (\ref{e.H<}) to hold is the analogous two weight $A_{2}$
condition%
\begin{equation*}
\sup_{I}\frac{1}{\left\vert I\right\vert }\int_{I}d\omega (x)\cdot \frac{1}{%
\left\vert I\right\vert }\int_{I}d\sigma (x)<\infty \,,
\end{equation*}%
taken uniformly over intervals $I$, where of course the weights need no
longer be absolutely continuous. As it turns out though, this two weight $%
A_{2}$ condition is no longer sufficient for the norm inequality (\ref{e.H<}%
), and F. Nazarov has shown that even the following necessary $\mathcal{A}%
_{2}$ condition `on steroids' of Nazarov, Treil and Volberg is not
sufficient:%
\begin{equation}
\sup_{I}\mathrm{P}(I,\omega )\cdot \mathrm{P}(I,\sigma )=\mathcal{A}%
_{2}<\infty ,\ \ \ \ \ \ \ \ \ \ \ \ \ \ \ \ \ \ \ \ \left( \mathcal{A}%
_{2}\right)  \label{A2}
\end{equation}%
where for an interval $I$ and measure $\omega $, the Poisson integral $%
\mathrm{P}(I,\omega )$ at $I$ is given by%
\begin{equation}
\mathrm{P}(I,\omega )\equiv \int_{\mathbb{R}}\frac{\lvert I\rvert }{(\lvert
I\rvert +\QTR{up}{dist}(x,I))^{2}}\;d\omega \left( x\right) .  \label{e.P}
\end{equation}%
See e.g. Theorem 2.1 in \cite{NiTr}.

We require in addition that the following testing conditions, also necessary
for the two weight inequality (\ref{e.H<}),%
\begin{equation}
\int_{I}\lvert H(\mathbf{1}_{I}\sigma )\rvert ^{2}\;\omega (dx)\leq 
\mathfrak{T}^{2}\left\vert I\right\vert _{\sigma }\,,\ \ \ \ \ \ \ \ \ \ \ \
\ \ \ \ \ \ \ \ \left( \mathfrak{T}\right)  \label{e.H1}
\end{equation}%
\begin{equation}
\int_{I}\lvert H(\mathbf{1}_{I}\omega )\rvert ^{2}\;\sigma (dx)\leq (%
\mathfrak{T}^{\ast })^{2}\left\vert I\right\vert _{\omega }\,,\ \ \ \ \ \ \
\ \ \ \ \ \ \ \ \ \ \ \ \ \left( \mathfrak{T}^{\ast }\right)  \label{e.H2}
\end{equation}%
hold uniformly over intervals $I$. Here, we are letting $\mathfrak{T}$ and $%
\mathfrak{T}^{\ast }$ denote the smallest constants for which these
inequalities are true uniformly over all intervals $I$, and we write $\sigma
(I)\equiv \int_{I}\sigma (dx)\equiv \left\vert I\right\vert _{\sigma }$. The
`NTV conjecture' of Nazarov, Treil and Volberg was that (\ref{A2}) and the
testing conditions (\ref{e.H1},\ref{e.H2}) are also sufficient for the norm
inequality (\ref{e.H<}). The following diagram illustrates the main
connections between the conditions considered here, but with arbitrary
singular integrals in place of the Hilbert transform.

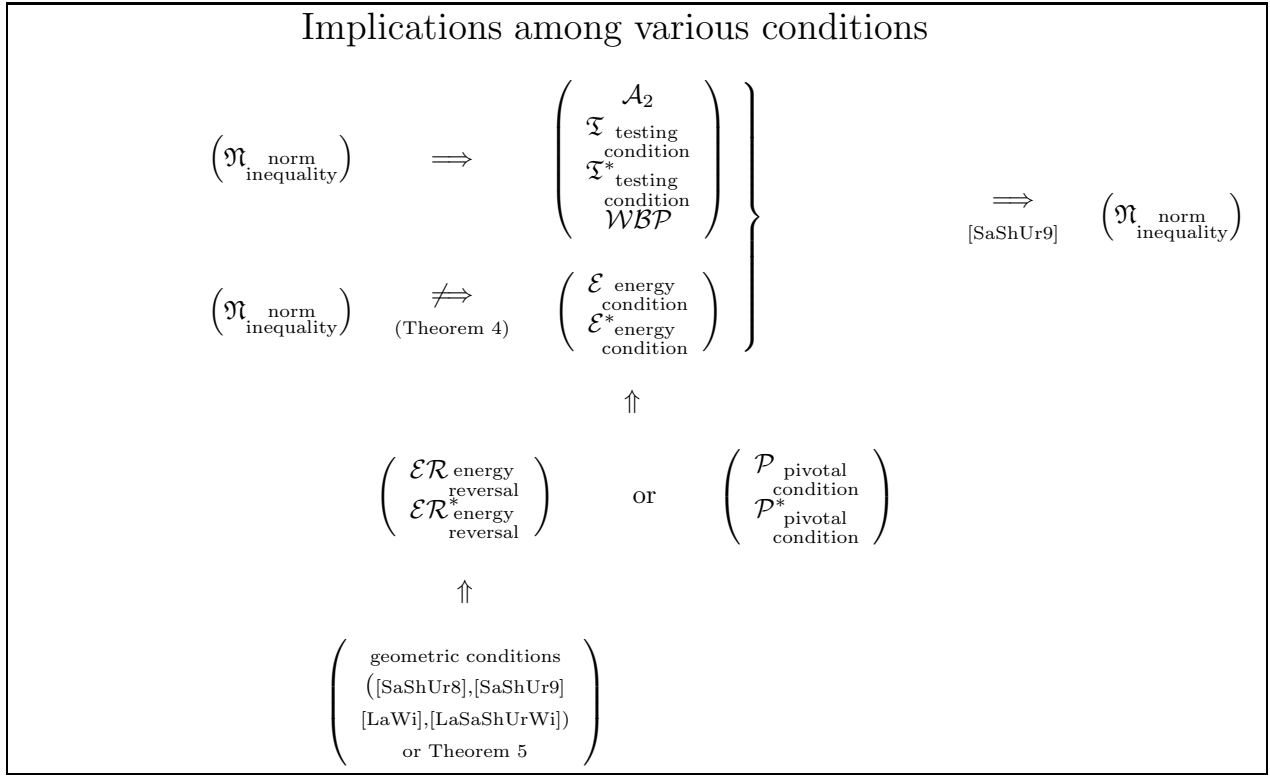
\begin{equation*}
\fbox{$%
\begin{array}{ccc}
\ \ \ \ \ \ \ \ \ \ \ \ \ \ \ \ \ \ \ \ \ \ \ \ \ \ \ \ \ \ \ \text{{\LARGE %
Implications among various conditions}} &  &  \\ 
&  &  \\ 
\left. 
\begin{array}{ccc}
\left( \mathfrak{N}_{\substack{ \limfunc{norm}  \\ \limfunc{inequality}}}%
\right) & \Longrightarrow & \left( 
\begin{array}{c}
\mathcal{A}_{2} \\ 
\mathfrak{T}_{\substack{ \limfunc{testing}  \\ \limfunc{condition}}} \\ 
\mathfrak{T}_{\substack{ \limfunc{testing}  \\ \limfunc{condition}}}^{\ast }
\\ 
\mathcal{WBP}%
\end{array}%
\right) \\ 
&  &  \\ 
\left( \mathfrak{N}_{\substack{ \limfunc{norm}  \\ \limfunc{inequality}}}%
\right) & 
\begin{array}{c}
\not\Longrightarrow \\ 
\text{{\scriptsize (Theorem\ \ref{energy condition fails})}}%
\end{array}
& \left( 
\begin{array}{c}
\mathcal{E}_{\substack{ \limfunc{energy}  \\ \limfunc{condition}}} \\ 
\mathcal{E}_{\substack{ \limfunc{energy}  \\ \limfunc{condition}}}^{\ast }%
\end{array}%
\right)%
\end{array}%
\right\} & 
\begin{array}{c}
\Longrightarrow \\ 
\text{{\scriptsize \cite{SaShUr9}}}%
\end{array}
& \left( \mathfrak{N}_{\substack{ \limfunc{norm}  \\ \limfunc{inequality}}}%
\right) \\ 
\ \ \ \ \ \ \ \ \ \ \ \ \ \ \ \ \ \ \ \ \ \ \ \ \ \ \ \ \ \ 
\begin{array}{c}
\\ 
\ \ \ \ \  \Uparrow \\ 
\\ 
\begin{array}{ccc}
\left( 
\begin{array}{c}
\mathcal{ER}_{\substack{ \limfunc{energy}  \\ \limfunc{reversal}}} \\ 
\mathcal{ER}_{\substack{ \limfunc{energy}  \\ \limfunc{reversal}}}^{\ast }%
\end{array}%
\right) & \text{or} & \ \ \ \ \left( 
\begin{array}{c}
\mathcal{P}_{\substack{ \limfunc{pivotal}  \\ \limfunc{condition}}} \\ 
\mathcal{P}_{\substack{ \limfunc{pivotal}  \\ \limfunc{condition}}}^{\ast }%
\end{array}%
\right) \\ 
&  &  \\ 
\Uparrow &  &  \\ 
&  &  \\ 
\left( 
\begin{array}{c}
\text{{\scriptsize geometric\ conditions}} \\ 
(\text{{\scriptsize \cite{SaShUr8},\cite{SaShUr9}}} \\ 
\text{{\scriptsize \cite{LaWi},\cite{LaSaShUrWi})}} \\ 
\text{{\scriptsize or Theorem\ \ref{gradient elliptic}}}%
\end{array}%
\right) &  & 
\end{array}%
\end{array}
&  & 
\end{array}%
$}
\end{equation*}

The approach of Nazarov, Treil and Volberg to prove (\ref{e.H<}) in \cite%
{NTV3} involves the assumption of two additional side conditions, the \emph{%
Pivotal Conditions} given by 
\begin{equation}
\sum_{r=1}^{\infty }\left\vert I_{r}\right\vert _{\omega }\mathsf{P}\left(
I_{r},\mathbf{1}_{I_{0}}\sigma \right) ^{2}\leq \mathcal{P}^{2}\left\vert
I_{0}\right\vert _{\sigma },\ \ \ \ \ \ \ \ \ \ \ \ \ \ \ \ \ \ \ \ \left( 
\mathcal{P}\right)  \label{pivotalcondition}
\end{equation}%
and its dual in which the measures $\sigma $ and $\omega $ are interchanged,
and where the inequality is required to hold for all intervals $I_{0}$, and
decompositions $\{I_{r}\;:\;r\geq 1\}$ of $I_{0}$ into disjoint intervals $%
I_{r}\subsetneq I_{0}$. Here $\mathcal{P}$ and $\mathcal{P}^{\ast }$ denote
the best constants in (\ref{pivotalcondition}) and its dual respectively.

In the approach initiated in \cite{LaSaUr2}, the Pivotal Condition (\ref%
{pivotalcondition}) was replaced by certain weaker conditions of \emph{%
energy type}, as given in (\ref{en}) above.

\begin{definition}
\label{d.energy} For a weight $\omega $, and interval $I$, we set 
\begin{equation*}
\mathsf{E}\left( I,\omega \right) \equiv \left[ \mathbb{E}_{I}^{\omega (dx)}%
\left[ \mathbb{E}_{I}^{\omega (dx^{\prime })}\frac{x-x^{\prime }}{\lvert
I\rvert }\right] ^{2}\right] ^{1/2}.
\end{equation*}
\end{definition}

It is important to note that $\mathsf{E}(I,\omega )\leq 1$, and can be quite
small, if $\omega $ is highly concentrated inside the interval $I$; in
particular if $\omega \mathbf{1}_{I}$ is a point mass, then $\mathsf{E}%
(I,\omega )=0$. Note also that $\omega (I)\left\vert I\right\vert ^{2}%
\mathsf{E}(I,\omega )^{2}$ is the variance of the variable $x$, and that we
have the identity%
\begin{equation*}
\mathsf{E}\left( I,\omega \right) ^{2}=\frac{1}{2}\mathbb{E}_{I}^{\omega
(dx)}\mathbb{E}_{I}^{\omega (dx^{\prime })}\frac{\left( x-x^{\prime }\right)
^{2}}{\left\vert I\right\vert ^{2}}\,.
\end{equation*}%
The following \emph{Energy Condition} and its dual condition were shown in 
\cite{LaSaUr2} to be necessary for the two weight norm inequality for the
Hilbert transform:%
\begin{equation}
\sum_{r\geq 1}\left\vert I_{r}\right\vert _{\omega }\mathsf{E}(I_{r},\omega
)^{2}\mathrm{P}(I_{r},\sigma \mathbf{1}_{I_{0}})^{2}\leq \mathcal{E}%
^{2}\left\vert I_{0}\right\vert _{\sigma },\ \ \ \ \ \ \ \ \ \ \ \ \ \ \ \ \
\ \ \ \left( \mathcal{E}\right)  \label{energy condition}
\end{equation}%
where the sum is taken over all decompositions $I_{0}=\bigcup_{r=1}^{\infty
}I_{r}$ of the interval $I_{0}$ into pairwise disjoint intervals $\left\{
I_{r}\right\} _{r\geq 1}$. Here $\mathcal{E}$ and $\mathcal{E}^{\ast }$
denote the best constants in (\ref{energy condition}) and its dual
respectively. As $\mathsf{E}(I,\omega )\leq 1$, the Energy Condition is
weaker than the Pivotal Condition. Indeed, in \cite{LaSaUr2} it was proved
that both the pivotal condition and some hybrid conditions, intermediate
between the pivotal and the energy condition, were sufficient for the norm
inequality for the Hilbert transform, but not necessary. The energy
condition, weakest of all of these conditions, was on the other hand proved
necessary. As mentioned earlier, the maneuvering room present in the pivotal
and hybrid conditions disappears at the level of the energy condition.

Later, in the two part paper \cite{LaSaShUr3}-\cite{Lac} by the authors and
Lacey, the NTV conjecture was proved, namely that the two weight inequality
for the Hilbert transform holds if and only if the $\mathcal{A}_{2}$ and
testing conditions hold - a restriction that the measures have no common
point masses was subsequently removed by Hyt\"{o}nen \cite{Hyt2}. A key role
was played in these results by the energy conditions, and in fact the energy
conditions have continued to play a crucial role in higher dimensions.

For example, they arise as side conditions for the $T1$ theorem in our
papers \cite{SaShUr6} and \cite{SaShUr7}, where we raised the question of
whether or not the energy conditions are necessary for \emph{any} elliptic
operator in higher dimensions. They also play a critical role in the side
conditions of uniformly full dimension of Lacey and Wick \cite{LaWi}\ since
those conditions, like the doubling conditions in \cite{NTV1}, and the
Energy Hypothesis in \cite{LaSaUr2}, imply the energy conditions. The energy
conditions were also shown to be necessary, and used crucially, in the cases
of the Hilbert transform \cite{LaSaShUr3}-\cite{Lac} already mentioned, in
the Cauchy operator with one measure on a line or circle \cite{LaSaShUrWi},
and in the generalization to compact $C^{1,\delta }$ curves in higher
dimension \cite{SaShUr8}.

In summary, the necessity of the energy condition was a crucial element in
each of these proofs of the $T1$ theorem for two weights, and in fact in all
proofs of a two weight $T1$ theorem to date. However, in every one of these
cases, the energy conditions were derived from the very strong property of 
\emph{energy reversal}. Recall from \cite{SaShUr4} that a vector $\mathbf{T}%
^{\alpha }=\left\{ T_{\ell }^{\alpha }\right\} _{\ell =1}^{N}$ of $\alpha $%
-fractional transforms in Euclidean space $\mathbb{R}^{n}$ satisfies a \emph{%
strong} reversal of $\omega $-energy on a cube $J$ if there is a positive
constant $C_{0}$ such that for all $\gamma \geq 2$ sufficiently large and
for all positive measures $\mu $ supported outside $\gamma J$, we have the
inequality%
\begin{equation}
\mathsf{E}\left( J,\omega \right) ^{2}\mathrm{P}^{\alpha }\left( J,\mu
\right) ^{2}\leq C_{0}\ \mathbb{E}_{J}^{d\omega \left( x\right) }\mathbb{E}%
_{J}^{d\omega \left( z\right) }\left\vert \mathbf{T}^{\alpha }\mu \left(
x\right) -\mathbf{T}^{\alpha }\mu \left( z\right) \right\vert ^{2}.
\label{will fail}
\end{equation}%
For the Hilbert transform in dimension $n=1$, (\ref{will fail}) is an
immediate consequence of of the equivalence $\left\vert H\mu \left( x\right)
-H\mu \left( x^{\prime }\right) \right\vert ^{2}\approx \left\vert \frac{%
x-x^{\prime }}{\left\vert J\right\vert }\right\vert ^{2}\mathrm{P}\left(
J,\mu \right) ^{2}$ - simply take $\omega $-expectations over $J$ in both $x$
and $x^{\prime }$. This property is tied in an essential way to the
ellipticity of the gradient of the kernel of the operator, and\ to either
the `fullness' of doubling and uniformly full dimension measures, or to the
`one-dimensional nature' of one of the measures. Subsequently, Lacey
observed the failure of energy reversal for the Cauchy operator in the
plane, and shortly thereafter in \cite{SaShUr4}, we established that energy
reversal fails spectacularly for classical fractional singular integrals in
higher dimension. However, this left open the crucial question of whether or
not the energy conditions themselves were necessary for (\ref{e.H<}) to
hold. Prior to this paper, \emph{no counterexample to the necessity of the
crucial \textbf{energy conditions} has been found for any elliptic operator
in any dimension}.

Now we can proceed to put matters into perspective. Recall that after
Nazarov's proof that the two-tailed Muckenhoupt $\mathcal{A}_{2}$ condition
alone was not sufficient to characterize the two weight norm inequality for
the Hilbert transform, all subsequent attempts to characterize a two weight
norm inequality for a \emph{general} class of singular integrals have failed
in that they all required a side condition. These side conditions have been
proposed by Nazarov, Treil, Volberg, Lacey, and the authors in various
papers since 2005, becoming weaker as time went on. But all of them have
failed to be necessary for a general two weight norm inequality, until the
weakest condition of them all, the one which arises implicitly in all proofs
of positive results to date, was proposed initially in \cite{LaSaUr2}, and
implemented systematically in higher dimensions beginning in \cite{SaShUr}
and \cite{SaShUr7}: namely the \emph{energy condition }(\ref{energy
condition}) on the pair of weights $\sigma $ and $\omega $, which is
independent of the singular integral operator $T$. This energy condition
turned out to be necessary for the case of the Hilbert transform, not only
leading to the solution of the $T1$ theorem for that operator, but also
raising expectations (which crystalized into conjectures) that it might be
necessary for \emph{all} elliptic singular integral operators. Our main
result, Theorem 1 below, culminates this thread of investigation by dashing
such expectations and proving that the energy condition is indeed not
necessary for elliptic operators in general.

Consequently, some other hypothesis is needed in order to use the only known
method of proof for two weight inequalities for general singular integrals
(our results here point to gradient ellipticity), without which either the $%
T1$ theorem fails, or its proof requires a completely new idea. The
counterexample in this paper should help inform any subsequent
investigations in these directions. A further comment is perhaps in order
here. Borderline results, such as this one, require a delicate understanding
of the problem involved. Indeed, the maneuvering room present in partial
results disappears in these problems posed at a level of critical behavior,
and overcoming the technical hurdles required to push the partial results to
their limit constitutes a significant tour de force. In the unstable
equilibrium of the current paper, where certain conditions barely hold and
others barely fail, delicate and substantial modifications are needed of the
example in \cite{LaSaUr2} (which started the series of counterexamples
mentioned above), and we discuss these thoroughly below.

Let us now briefly describe in a nutshell for experts the content of this
manuscript. In our main result, Theorem \ref{energy condition fails} below,
we provide a counterexample to the necessity of the energy condition for the
boundedness of a small, but elliptic, perturbation of the Hilbert transform.
This answers in the negative Problem 7 in \cite{SaShUr7} and Conjecture 3 in
both \cite{SaShUr8} and \cite{SaShUr9}, which if true would have established
the celebrated $T1$ theorem for such operators - namely that boundedness of $%
T$ is equivalent to a Muckenhoupt condition and testing the operator and its
dual on indicators of intervals. In fact, the only known proofs of two
weight $T1$ theorems to date all rely crucially on the necessity of the
energy condition. On the other hand, in a relatively simple adaptation of an
existing argument in \cite{LaSaUr2}, together with our main theorem in \cite%
{SaShUr10} (see also \cite{SaShUr9} or \cite{SaShUr6} for a more leisurely
exposition), we show in Theorem \ref{gradient elliptic} below, that for a
class of singular integral operators on the line that narrowly avoid our
counterexample (i.e. their kernels are both elliptic and gradient elliptic),
the $T1$ theorem does indeed hold for these operators.

\begin{conjecture}
The energy conditions (see \cite{SaShUr7}) are necessary for boundedness of
a vector of standard singular integrals in higher dimensions provided the
vector singular integral is both strongly elliptic (see \cite{SaShUr7}) and
strongly gradient elliptic. If this conjecture is true, the $T1$ theorem
would then follow for such operators by the main theorem in \cite{SaShUr9}
or \cite{SaShUr10}.
\end{conjecture}

\subsection{Statements of theorems}

The main purpose of this paper then is to give such a counterexample. For
this we first we recall the precise meaning of the two weight norm
inequality for a standard singular integral $T$ on the real line. Define a
standard CZ kernel $K(x,y)$ to be a real-valued function defined on $\mathbb{%
R}\times \mathbb{R}$ satisfying the following fractional size and smoothness
conditions of order $1+\delta $ for some $\delta >0$: For $x\neq y$,%
\begin{eqnarray}
\left\vert K\left( x,y\right) \right\vert &\leq &C_{CZ}\left\vert
x-y\right\vert ^{-1}\text{ and }\left\vert \nabla K\left( x,y\right)
\right\vert \leq C_{CZ}\left\vert x-y\right\vert ^{-2},
\label{sizeandsmoothness'} \\
\left\vert \nabla K\left( x,y\right) -\nabla K\left( x^{\prime },y\right)
\right\vert &\leq &C_{CZ}\left( \frac{\left\vert x-x^{\prime }\right\vert }{%
\left\vert x-y\right\vert }\right) ^{\delta }\left\vert x-y\right\vert
^{-2},\ \ \ \ \ \frac{\left\vert x-x^{\prime }\right\vert }{\left\vert
x-y\right\vert }\leq \frac{1}{2},  \notag
\end{eqnarray}%
and the last inequality also holds for the adjoint kernel in which $x$ and $%
y $ are interchanged. We note that a more general definition of kernel has
only order of smoothness $\delta >0$, rather than $1+\delta $, but the use
of the Monotonicity and Energy Lemmas in arguments below involves first
order Taylor approximations to the kernel functions $K\left( \cdot ,y\right) 
$. In order to give a precise definition of the two weight norm inequality%
\begin{equation}
\left\Vert T_{\sigma }f\right\Vert _{L^{2}\left( \omega \right) }\leq 
\mathfrak{N}_{T_{\sigma }}\left\Vert f\right\Vert _{L^{2}\left( \sigma
\right) },\ \ \ \ \ f\in L^{2}\left( \sigma \right) ,  \label{two weight'}
\end{equation}%
we introduce a family $\left\{ \eta _{\delta ,R}\right\} _{0<\delta
<R<\infty }$ of nonnegative functions on $\left[ 0,\infty \right) $ so that
the truncated kernels $K_{\delta ,R}\left( x,y\right) =\eta _{\delta
,R}\left( \left\vert x-y\right\vert \right) K\left( x,y\right) $ are bounded
with compact support for fixed $x$ or $y$. Then the truncated operators 
\begin{equation*}
T_{\sigma ,\delta ,R}f\left( x\right) \equiv \int_{\mathbb{R}}K_{\delta
,R}\left( x,y\right) f\left( y\right) d\sigma \left( y\right) ,\ \ \ \ \
x\in \mathbb{R},
\end{equation*}%
are pointwise well-defined, and we will refer to the pair $\left( K,\left\{
\eta _{\delta ,R}\right\} _{0<\delta <R<\infty }\right) $ as a singular
integral operator, which we typically denote by $T$, suppressing the
dependence on the truncations.

\begin{definition}
We say that a singular integral operator $T=\left( K,\left\{ \eta _{\delta
,R}\right\} _{0<\delta <R<\infty }\right) $ satisfies the norm inequality (%
\ref{two weight'}) provided%
\begin{equation*}
\left\Vert T_{\sigma ,\delta ,R}f\right\Vert _{L^{2}\left( \omega \right)
}\leq \mathfrak{N}_{T_{\sigma }}\left\Vert f\right\Vert _{L^{2}\left( \sigma
\right) },\ \ \ \ \ f\in L^{2}\left( \sigma \right) ,0<\delta <R<\infty .
\end{equation*}
\end{definition}

It turns out that, in the presence of Muckenhoupt conditions, the norm
inequality (\ref{two weight'}) is essentially independent of the choice of
truncations used, and this is explained in some detail in \cite{SaShUr10},
see also \cite{LaSaShUr3}. Thus, as in \cite{SaShUr10}, we are free to use
the tangent line truncations described there throughout the proofs of our
results.

\begin{theorem}
\label{energy condition fails}There exists a weight pair $\left( \widehat{%
\sigma },\widehat{\omega }\right) $ and an elliptic singular integral $%
H_{\flat }$ on the real line $\mathbb{R}$ such that $H_{\flat }$ satisfies
the two weight norm inequality 
\begin{equation*}
\left\Vert H_{\flat }(f\widehat{\sigma })\right\Vert _{L^{2}(\widehat{\omega 
})}\lesssim \left\Vert f\right\Vert _{L^{2}(\widehat{\sigma })}\ ,
\end{equation*}%
yet the weight pair $\left( \widehat{\sigma },\widehat{\omega }\right) $ 
\textbf{fails} to satisfy the backward Energy Condition%
\begin{equation}
\sum_{r\geq 1}\left\vert I_{r}\right\vert _{\widehat{\sigma }}\mathsf{E}%
(I_{r},\widehat{\sigma })^{2}\mathrm{P}(I_{r},\widehat{\omega }\mathbf{1}%
_{I_{0}})^{2}\leq \left( \mathcal{E}^{\ast }\right) ^{2}\left\vert
I_{0}\right\vert _{\widehat{\omega }},\ \ \ \ \ \ \ \ \ \ \ \ \ \ \ \ \ \ \
\ \left( \mathcal{E}^{\ast }\right)  \label{dual energy condition}
\end{equation}%
for any $\mathcal{E}^{\ast }<\infty $, where the sum is taken over all
decompositions $I_{0}=\bigcup_{r=1}^{\infty }I_{r}$ of the interval $I_{0}$
into pairwise disjoint intervals $\left\{ I_{r}\right\} _{r\geq 1}$.
\end{theorem}

To prove this theorem, we build on the example from \cite{LaSaUr2} that was
used to show the failure of necessity of the pivotal condition for the
Hilbert transform. However, there is a strong connection between testing
conditions and their corresponding energy conditions, and we must work hard
to disengage this connection while simultaneously retaining the norm
inequality. All of this requires substantial modification of the example in 
\cite{LaSaUr2}, and involves delicate symmetries of redistributed Cantor
measures. Moreover, it is essential to perturb the Hilbert transform so that
it is no longer gradient elliptic. We now describe our strategy in more
detail:

\begin{itemize}
\item As in \cite{LaSaUr2}, we start with one of the usual Cantor measures $%
\omega $ on $\left[ 0,1\right] $, and a measure $\dot{\sigma}$ that consists
of an infinite number of point masses centered in the gaps of the Cantor
measure $\omega $, and that is chosen so that the $A_{2}$ condition holds
for the pair $\left( \dot{\sigma},\omega \right) $.

\item Now we replace the point masses in $\dot{\sigma}$ with averages over
small intervals $L$ to get a measure $\sigma $ so that the $A_{2}$ condition
holds for $\left( \sigma ,\omega \right) $, but the backward energy
condition \textbf{fails} since it becomes the backward pivotal condition,
which was essentially shown to fail in \cite{LaSaUr2}. This depends
crucially on the fact that the energy of a point mass vanishes, but not that
of an average over an interval.

\item However, the backward testing condition for the pair $\left( \sigma
,\omega \right) $ then fails due to the strong connection between testing
and energy conditions. In order to obtain the backward testing condition for
the pair $\left( \sigma ,\omega \right) $, it would suffice to have our
singular integral $H_{\flat }$ applied to $\omega $ \textbf{vanish} on the
intervals $L$, and in particular this requires the kernel $K_{\flat }$ of
our singular integral to have appropriately located flat spots.

\item In order to achieve $H_{\flat }\omega =0$ on the intervals $L$, it is
thus necessary to inductively redistribute the Cantor measure $\omega $ into
a new measure $\widehat{\omega }$.

\item To preserve the $A_{2}$ condition we must reweight the interval masses
in $\sigma $ to get a new measure $\widehat{\sigma }$. It turns out that the
backward testing condition continues to hold for the new pair $\left( 
\widehat{\sigma },\widehat{\omega }\right) $ and that the backward energy
condition continues to fail for the new pair $\left( \widehat{\sigma },%
\widehat{\omega }\right) $.

\item However, the forward testing condition is now in doubt because the
argument in \cite{LaSaUr2} for the analogous inequality made strong use of
the self-similarity of both measures involved. The redistributed measures $%
\widehat{\omega }$ and $\widehat{\sigma }$ seem at first sight to have lost
all trace of self-similarity. Surprisingly, there are hidden symmetries in
the construction of $\widehat{\omega }$ that lead to suitable replication
formulas for both $\widehat{\omega }$ and $\widehat{\sigma }$, and then a
delicate argument shows that the forward testing condition does indeed hold
for the pair $\left( \widehat{\sigma },\widehat{\omega }\right) $. This step
represents the main challenge overcome in this paper.

\item Finally, we wish to prove that the norm inequality holds for the
weight pair $\left( \widehat{\sigma },\widehat{\omega }\right) $, and the
only known methods for this to date involve having both energy conditions
for the weight pair under consideration. Since the backward energy condition
fails for this weight pair, we must resort to a trick that exploits the flat
spots of our kernel. The first half of the trick is to notice that
everything we have done for the weight pair $\left( \widehat{\sigma },%
\widehat{\omega }\right) $ can also be done for the weight pair $\left( 
\widehat{\dot{\sigma}},\widehat{\omega }\right) $ in which $\widehat{\dot{%
\sigma}}$ is the corresponding reweighting of the original measure $\dot{%
\sigma}$, but with one exception: the backward energy condition now \textbf{%
holds} because $\widehat{\dot{\sigma}}$ consists of point masses instead of
intervals, and because the energy of individual point masses vanishes. The
second half of the trick is then to observe that the dual norm inequalities
for the weight pairs $\left( \widehat{\sigma },\widehat{\omega }\right) $
and $\left( \widehat{\dot{\sigma}},\widehat{\omega }\right) $ are \textbf{%
equivalent} due to the flat spots in the kernel! Thus the norm inequality
for $H_{\flat }$ holds for the weight pair $\left( \widehat{\sigma },%
\widehat{\omega }\right) $, but the backward energy condition fails for $%
\left( \widehat{\sigma },\widehat{\omega }\right) $.
\end{itemize}

As mentioned earlier, provided we narrowly avoid the operator $H_{\flat }$
constructed in the proof of Theorem \ref{energy condition fails}, the $T1$
theorem will hold. To state this result for more general measures $\omega $
and $\sigma $ we need the more refined Muckenhoupt conditions that are
adapted to the case of locally finite positive Borel measures $\omega $ and $%
\sigma $ that may have common point masses. Define \emph{fraktur} $A_{2}$ to
be the sum of the four $A_{2}^{\alpha }$ conditions,%
\begin{equation*}
\mathfrak{A}_{2}=\mathcal{A}_{2}+\mathcal{A}_{2}^{\ast }+A_{2}^{\limfunc{%
punct}}+A_{2}^{\ast ,\limfunc{punct}},
\end{equation*}%
where $\mathcal{A}_{2}$ and $\mathcal{A}_{2}^{\ast }$ are the one-tailed
Muckenhoupt conditions with \emph{holes}, and $A_{2}^{\limfunc{punct}}$ and $%
A_{2}^{\ast ,\limfunc{punct}}$ are the \emph{punctured} Muckenhoupt
conditions. All of these Muckenhoupt conditions avoid taking products of
integrals of $\omega $ and $\sigma $ over common point masses, and we refer
the reader to \cite{SaShUr9} or \cite{SaShUr10} for the somewhat technical
definitions of these Muckenhoupt conditions adapted to measures having
common point masses.

We say that the kernel $K$ of a standard singular integral $T$ on the real
line is \emph{gradient elliptic} if%
\begin{equation*}
\frac{\partial }{\partial x}K\left( x,y\right) ,-\frac{\partial }{\partial y}%
K\left( x,y\right) \geq \frac{c}{\left( x-y\right) ^{2}},\ \ \ \ \ x,y\in 
\mathbb{R}.
\end{equation*}%
Note that the Hilbert transform kernel $K\left( x,y\right) =\frac{1}{y-x}$
satisfies $\frac{\partial }{\partial x}K\left( x,y\right) =-\frac{\partial }{%
\partial y}K\left( x,y\right) =\frac{1}{\left( x-y\right) ^{2}}$.

\begin{theorem}
\label{gradient elliptic}Suppose that the kernel $K$ of a standard singular
integral $T$ on the real line is elliptic and gradient elliptic. Then $%
T:L^{2}\left( \sigma \right) \longrightarrow L^{2}\left( \omega \right) $ 
\emph{if and only if} the four Muckenhoupt conditions hold, i.e. $\mathfrak{A%
}_{2}<\infty $, and both testing conditions hold,%
\begin{equation}
\int_{I}\lvert T(\mathbf{1}_{I}\sigma )\rvert ^{2}\;\omega (dx)\leq 
\mathfrak{T}^{2}\left\vert I\right\vert _{\sigma }\,,
\label{forward testing}
\end{equation}%
\begin{equation}
\int_{I}\lvert T^{\ast }(\mathbf{1}_{I}\omega )\rvert ^{2}\;\sigma (dx)\leq (%
\mathfrak{T}^{\ast })^{2}\left\vert I\right\vert _{\omega }\,.
\label{baclward testimg}
\end{equation}
\end{theorem}

Theorems and \ref{energy condition fails} and \ref{gradient elliptic} point
to a new phenomenon that is needed in order obtain the energy conditions,
namely the need for gradient ellipticity of the kernels of the singular
integrals. This paper takes a first step toward the pursuit of $T1$ theorems
via energy conditions.

We end this introduction by pointing to some applications of the two weight $%
T1$ theorem in operator theory, such as in \cite{LaSaShUrWi}, where
embedding measures are characterized for model spaces $K_{\theta }$, where $%
\theta $ is an inner function on the disk, and where norms of composition
operators are characterized that map $K_{\theta }$ into Hardy and Bergman
spaces. A $T1$ theorem could also have implications for a number of problems
that are higher dimensional analogues of those connected to the Hilbert
transform\footnote{%
We thank Professor N. Nikolski for providing information on many of these
topics.},

\begin{enumerate}
\item when a rank one perturbation of a unitary operator is similar to a
unitary operator (see e.g. \cite{Vol} and \cite{NiTr}): this could extend to
an analogous question for a rank one perturbation of a normal operator $T$
and lead to a two weight inequality for the Cauchy transform with one
measure being the spectral measure of $T$,

\item when a product of two densely defined Toeplitz operators $T_{a}T_{b}$
is a bounded operator, which is equivalent to the Birkhoff-Wiener-Hopf
factorization for a given function $c=ab$; the same questions for the
Bergman space could lead to a two weight problem for the Beurling transform,

\item questions regarding subspaces of the Hardy space invariant under the
inverse shift operator (see e.g. \cite{Vol} and \cite{NaVo}),

\item questions concerning orthogonal polynomials (see e.g. \cite{VoYu}, 
\cite{PeVoYu} and \cite{PeVoYu1}),
\end{enumerate}

and also to a variety of questions in quasiconformal theory (due to the
relevance of the Beurling transform in that context) such as,

\begin{enumerate}
\item the conjecture of Iwaniec and Martin, at the level of Hausdorff
dimension distortion (see \cite{IwMa}) and, at the level of Hausdorff
measures, higher dimensional analogues of the Astala conjecture (see e.g. 
\cite{LaSaUr}), for which proof an (essentially) two weight inequality for
the Beurling transform was crucial, and in general, similar questions
pertaining to the higher dimensional analogues of the Beurling transform,

\item the problem of characterizing which Beltrami coefficients give rise to
biLipschitz maps (see e.g. \cite{AsGo} where a two weight theorem is proved
for very specific weights),

\item and to the well known problem of the connectivity of the manifold of
planar chord-arc curves (see e.g. \cite{AsGo} and \cite{AsZi}).
\end{enumerate}

The proof of Theorem \ref{energy condition fails} will be given in the next
four sections, and the proof of Theorem \ref{gradient elliptic} will be
given in the final section.

\section{Families of Cantor measures}

Let $N\geq 3$. We construct a standard variant $\mathsf{E}^{\left( N\right)
} $ of the middle thirds Cantor set $\mathsf{E}$, namely we remove the
middle `$\frac{N-2}{N}$' at each generation, and we then construct an
associated weight pair $\left( \widehat{\sigma },\widehat{\omega }\right) $,
where $\widehat{\omega }$ is a `redistribution' of the Cantor measure $%
\omega ^{\left( N\right) }$ associated with $\mathsf{E}^{\left( N\right) }$,
and $\widehat{\sigma }$ is mutually singular with respect to $\widehat{%
\omega }$. We will later show that for $N$ sufficiently large, there is an
elliptic perturbation $H_{\flat }$ of the Hilbert transform $H$, such that $%
H_{\flat } $ is bounded from $L^{2}\left( \widehat{\sigma }\right) $ to $%
L^{2}\left( \widehat{\omega }\right) $, but the backward Energy Condition 
\textbf{fails} for the weight pair $\left( \widehat{\sigma },\widehat{\omega 
}\right) $.

\subsection{The Cantor construction}

We modify the construction of the middle-third Cantor set $\mathsf{E}$ and
Cantor measure $\omega $ on the closed unit interval $I_{1}^{0}=\left[ 0,1%
\right] $. Fix a real number%
\begin{equation*}
N>2,
\end{equation*}%
and at each stage, remove the central interval of proportion%
\begin{equation*}
\alpha \equiv \frac{N-2}{N}
\end{equation*}%
(the usual Cantor construction is the case $N=3$ and $\alpha =\frac{1}{3}$).
At the $k^{th}$ generation in the new construction, there is a collection $%
\left\{ I_{j}^{k}\right\} _{j=1}^{2^{k}}$ of $2^{k}$ pairwise disjoint
closed intervals of length $\left\vert I_{j}^{k}\right\vert =\frac{1}{N^{k}}$%
. With $K_{k}=\bigcup_{j=1}^{2^{k}}I_{j}^{k}$, the $N$-Cantor set is defined
by 
\begin{equation*}
\mathsf{E}^{\left( N\right) }=\bigcap_{k=1}^{\infty
}K_{k}=\bigcap_{k=1}^{\infty }\left( \bigcup_{j=1}^{2^{k}}I_{j}^{k}\right) .
\end{equation*}%
The $N$-Cantor measure $\omega ^{\left( N\right) }$ is the unique
probability measure supported in $\mathsf{E}^{\left( N\right) }$ with the
property that it is equidistributed among the intervals $\left\{
I_{j}^{k}\right\} _{j=1}^{2^{k}}$ at each scale $k$, i.e.%
\begin{equation*}
\omega ^{\left( N\right) }(I_{j}^{k})=2^{-k},\ \ \ \ \ k\geq 0,1\leq j\leq
2^{k}.
\end{equation*}

Now we fix $N$ large for the moment and suppress the dependence on $N$ in
our notation, e.g. we denote $\omega ^{\left( N\right) }$ by simply $\omega $%
. We denote the removed open middle $\alpha ^{th}$ of $I_{j}^{k}$ by $%
G_{j}^{k}$.

\section{The flattened Hilbert transform}

We will now take an integer $N\geq 3$ large and a positive number $\rho \in
\left( 0,1\right) $ close to $1$, and flatten the convolution kernel $%
K\left( x\right) =\frac{1}{x}$ of the Hilbert transform in sufficiently
small neighbourhoods of the points $\left\{ \pm N^{k}\right\} _{k\in \mathbb{%
Z}}$ to obtain a flattened Hilbert transform $H_{\flat }$. In order to help
motivate this definition, we divert our attention to a brief explanation of
what we will subsequently do with the flattened Hilbert transform $H_{\flat
} $. Let $\dot{z}_{j}^{k}\in G_{j}^{k}$ be the center of the interval $%
G_{j}^{k}$, which is also the center of the interval $I_{j}^{k}$.

\begin{description}
\item[Motivation] We will later redistribute the measure $\omega $
constructed above into a new measure $\widehat{\omega }$ supported on $%
E^{\left( N\right) }$ with the property that $H_{\flat }\widehat{\omega }%
\left( \dot{z}_{j}^{k}\right) =0$ for all $\left( k,j\right) $. Then we will
define the weights $\widehat{s}_{j}^{k}$ so that the measure 
\begin{equation*}
\widehat{\dot{\sigma}}=\sum_{k,j}\widehat{s}_{j}^{k}\delta _{{\dot{z}}%
_{j}^{k}}\ ,
\end{equation*}%
satisfies a certain `local' $A_{2}$ condition with respect to $\widehat{%
\omega }$, and define%
\begin{equation*}
\widehat{{\sigma }}\equiv \sum_{k,j}\widehat{s}_{j}^{k}\frac{1}{\left\vert
L_{j}^{k}\right\vert }\mathbf{1}_{L_{j}^{k}}\ ,
\end{equation*}%
where $L_{j}^{k}$ is a small interval centered at ${\dot{z}}_{j}^{k}$. We
will then establish in later sections that the weight pair $\left( \widehat{%
\dot{\sigma}},\widehat{\omega }\right) $ satisfies the $\mathcal{A}_{2}$
condition, both the forward and backward testing conditions with respect to $%
H_{\flat }$, and the forward and backward energy conditions. Thus we
conclude from our theorem in \cite{SaShUr7} (note that $\widehat{\omega }$
and $\widehat{\dot{\sigma}}$ have no common point masses) that $\left( 
\widehat{\dot{\sigma}},\widehat{\omega }\right) $ satisfies the norm
inequality with respect to $H_{\flat }$, and from this we then deduce the
norm inequality for the weight pair $\left( \widehat{\sigma },\widehat{%
\omega }\right) $ with respect to $H_{\flat }$. Finally we show that the
weight pair $\left( \widehat{\sigma },\widehat{\omega }\right) $ \textbf{%
fails} to satisfy the backward energy condition.
\end{description}

\subsection{The flattened Hilbert kernel}

We define $K_{\flat }:\left[ \frac{1}{\sqrt{N}},\sqrt{N}\right] \rightarrow %
\left[ \frac{1}{\sqrt{N}},\sqrt{N}\right] $ to be smooth and satisfy 
\begin{equation*}
K_{\flat }\left( x\right) =\left\{ 
\begin{array}{ccc}
\frac{1}{x} & \text{ if } & \frac{1}{\sqrt{N}}\leq x\leq \frac{1}{\sqrt{\rho
N}} \\ 
1 & \text{ if } & \frac{1}{\sqrt{\rho ^{2}N}}\leq x\leq \sqrt{\rho ^{2}N} \\ 
\frac{1}{x} & \text{ if } & \sqrt{\rho N}\leq x\leq \sqrt{N}%
\end{array}%
\right. .
\end{equation*}%
Then for $k\in \mathbb{Z}$, we extend the definition of $K_{\flat }$ to the
intervals $N^{k}\left[ \frac{1}{\sqrt{N}},\sqrt{N}\right] =\left[ N^{k-\frac{%
1}{2}},N^{k+\frac{1}{2}}\right] $ by $K_{\flat }\left( x\right)
=N^{-k}K_{\flat }\left( N^{-k}x\right) $ for $x\in \left[ N^{k-\frac{1}{2}%
},N^{k+\frac{1}{2}}\right] $, and finally extend the definition of $K_{\flat
}$ to the entire real line $\mathbb{R}$ by requiring it to be an odd
function on $\mathbb{R}$. The resulting kernel $K_{\flat }$ is smooth on $%
\mathbb{R}\setminus \left\{ 0\right\} $ and satisfies the standard CZ
estimates, and most importantly, for each $k$ the kernel $K_{\flat }$ is
flat on the interval%
\begin{equation*}
F_{k}\equiv N^{-k}\left[ \frac{1}{\sqrt{\rho ^{2}N}},\sqrt{\rho ^{2}N}\right]
=\left[ \frac{1}{\rho }N^{-k-\frac{1}{2}},\rho N^{-k+\frac{1}{2}}\right]
\end{equation*}%
containing the point $N^{-k}$.

Equally important is the connection with the measure $\omega $ constructed
with parameter $N$. If $x\in I_{r}^{\ell }$ and $y\in I_{j}^{k}$ with $%
I_{r}^{\ell }\cap I_{j}^{k}=\emptyset $, and if $m\in \mathbb{N}$ is the
largest positive integer such that both $x$ and $y$ belong to some $%
I_{i}^{m} $ for $1\leq i\leq 2^{m}$, then%
\begin{equation}
\left\vert \frac{1}{N^{m}}-\left\vert x-y\right\vert \right\vert \leq \frac{2%
}{N^{m+1}}.  \label{unique m}
\end{equation}%
Define $I_{r,\limfunc{left}}^{\ell }$ and $I_{r,\limfunc{right}}^{\ell }$ to
be the closest intervals $I_{j}^{\ell +1}$ on each side of $\dot{z}%
_{r}^{\ell }$ that belong to the next generation. We claim that if 
\begin{equation*}
x\in I_{i,\limfunc{left}}^{m}\cup \left[ \dot{z}_{i}^{m}-\frac{1}{N^{m+1}},%
\dot{z}_{i}^{m}+\frac{1}{N^{m+1}}\right] ,\ \ \ y\in I_{i,\limfunc{right}%
}^{m},
\end{equation*}%
then%
\begin{equation}
\left\vert x-y\right\vert \in \left[ \frac{1}{N^{m}}-\frac{2}{N^{m+1}},\frac{%
1}{N^{m}}\right] \cup \left[ \frac{1}{2N^{m}}-\frac{2}{N^{m+1}},\frac{1}{%
2N^{m}}+\frac{1}{N^{m+1}}\right] \subset F_{m}\ ,  \label{contain}
\end{equation}%
provided 
\begin{equation*}
\frac{1}{\rho }N^{-m-\frac{1}{2}}\leq \frac{1}{2N^{m}}-\frac{2}{N^{m+1}}%
\text{ and }\frac{1}{N^{m}}\leq \rho N^{-m+\frac{1}{2}}\ ,
\end{equation*}%
equivalently 
\begin{equation*}
\frac{1}{\rho }\frac{1}{\sqrt{N}}\leq \frac{1}{2}-\frac{2}{N}\text{ and }%
1\leq \rho \sqrt{N}\ ,
\end{equation*}%
which holds for example if $N\geq 16$ and $\rho \geq \frac{2}{3}$.

\section{Self-similar measures}

\subsection{The redistributed measure $\protect\widehat{\protect\omega } 
\label{SubSec redist}$}

We now construct, by induction on the level $\ell $, new measures $\omega
_{\ell }$ by adjusting at each stage the relative weighting of the measure $%
\omega _{\ell }$ on these two intervals $I_{r,\limfunc{left}}^{\ell }$ and $%
I_{r,\limfunc{right}}^{\ell }$. We begin with $\omega _{1}\equiv \omega $,
and having defined $\omega _{\ell }$ inductively we define $\omega _{\ell
+1} $ as follows. Fix the measure $\omega _{\ell }$ that was constructed
inductively, and note that it is supported on the Cantor set $\mathsf{E}%
^{\left( N\right) }$ constructed in the previous section. At the point $y=%
\dot{z}_{r}^{\ell }$, and with $x_{r,\limfunc{left}}^{\ell }$ and $x_{r,%
\limfunc{right}}^{\ell }$ denoting any points in $I_{r,\limfunc{left}}^{\ell
}$ and $I_{r,\limfunc{right}}^{\ell }$ respectively, we have from (\ref%
{contain}) that 
\begin{eqnarray}
&&H_{\flat }\omega _{\ell }\left( \dot{z}_{r}^{\ell }\right) =\int_{\left(
G_{r}^{\ell }\right) ^{c}}K_{\flat }\left( x-\dot{z}_{r}^{\ell }\right)
d\omega _{\ell }\left( x\right) =\sum_{i}\int_{I_{i}^{\ell +1}}K_{\flat
}\left( x-\dot{z}_{r}^{\ell }\right) d\omega _{\ell }\left( x\right)
\label{H flat} \\
&=&\int_{I_{r,\limfunc{left}}^{\ell }}K_{\flat }\left( x-\dot{z}_{r}^{\ell
}\right) d\omega _{\ell }\left( x\right) +\int_{I_{r,\limfunc{right}}^{\ell
}}K_{\flat }\left( x-\dot{z}_{r}^{\ell }\right) d\omega _{\ell }\left(
x\right) +\sum_{i:I_{i}^{\ell +1}\notin \left\{ I_{r,\limfunc{left}}^{\ell
},I_{r,\limfunc{right}}^{\ell }\right\} }\int_{I_{i}^{\ell +1}}K_{\flat
}\left( x-\dot{z}_{r}^{\ell }\right) d\omega _{\ell }\left( x\right)  \notag
\\
&=&K_{\flat }\left( x_{r,\limfunc{left}}^{\ell }-\dot{z}_{r}^{\ell }\right)
\left\vert I_{r,\limfunc{left}}^{\ell }\right\vert _{\omega _{\ell
}}+K_{\flat }\left( x_{r,\limfunc{right}}^{\ell }-\dot{z}_{r}^{\ell }\right)
\left\vert I_{r,\limfunc{right}}^{\ell }\right\vert _{\omega _{\ell
}}+\sum_{i:I_{i}^{\ell +1}\notin \left\{ I_{r,\limfunc{left}}^{\ell },I_{r,%
\limfunc{right}}^{\ell }\right\} }\int_{I_{i}^{\ell +1}}K_{\flat }\left( x-%
\dot{z}_{r}^{\ell }\right) d\omega _{\ell }\left( x\right)  \notag \\
&=&K_{\flat }\left( x_{r,\limfunc{left}}^{\ell }-\dot{z}_{r}^{\ell }\right)
\left( \left\vert I_{r,\limfunc{left}}^{\ell }\right\vert _{\omega _{\ell
}}-\left\vert I_{r,\limfunc{right}}^{\ell }\right\vert _{\omega _{\ell
}}\right) +\sum_{i:I_{i}^{\ell +1}\notin \left\{ I_{r,\limfunc{left}}^{\ell
},I_{r,\limfunc{right}}^{\ell }\right\} }\int_{I_{i}^{\ell +1}}K_{\flat
}\left( x-\dot{z}_{r}^{\ell }\right) d\omega _{\ell }\left( x\right) , 
\notag
\end{eqnarray}%
where we recall that $I_{r,\limfunc{left}}^{\ell }$ and $I_{r,\limfunc{right}%
}^{\ell }$ are the closest next generation intervals $I_{j}^{\ell +1}$ on
each side of $\dot{z}_{r}^{\ell }$. Now we define the relative weighting of
the measure $\omega _{\ell +1}$ on these two intervals $I_{r,\limfunc{left}%
}^{\ell }$ and $I_{r,\limfunc{right}}^{\ell }$ by taking $\omega _{\ell +1}$
to satisfy both%
\begin{eqnarray}
&&\left\vert I_{r}^{\ell }\right\vert _{\omega _{\ell +1}}=\left\vert I_{r,%
\limfunc{left}}^{\ell }\right\vert _{\omega _{\ell +1}}+\left\vert I_{r,%
\limfunc{right}}^{\ell }\right\vert _{\omega _{\ell +1}}=\left\vert I_{r,%
\limfunc{left}}^{\ell }\right\vert _{\omega _{\ell }}+\left\vert I_{r,%
\limfunc{right}}^{\ell }\right\vert _{\omega _{\ell }}=\left\vert
I_{r}^{\ell }\right\vert _{\omega _{\ell }},  \label{satisfy} \\
&&K_{\flat }\left( x_{r,\limfunc{left}}^{\ell }-\dot{z}_{r}^{\ell }\right)
\left( \left\vert I_{r,\limfunc{left}}^{\ell }\right\vert _{\omega _{\ell
+1}}-\left\vert I_{r,\limfunc{right}}^{\ell }\right\vert _{\omega _{\ell
+1}}\right) =-\sum_{i:I_{i}^{\ell +1}\notin \left\{ I_{r,\limfunc{left}%
}^{\ell },I_{r,\limfunc{right}}^{\ell }\right\} }\int_{I_{i}^{\ell
+1}}K_{\flat }\left( x-\dot{z}_{r}^{\ell }\right) d\omega _{\ell }\left(
x\right) .  \notag
\end{eqnarray}%
Thus we have%
\begin{eqnarray}
\left\vert I_{r,\limfunc{left}}^{\ell }\right\vert _{\omega _{\ell +1}} &=&%
\frac{1}{2}\left\{ \left\vert I_{r}^{\ell }\right\vert _{\omega _{\ell
}}-\sum_{i:I_{i}^{\ell +1}\notin \left\{ I_{r,\limfunc{left}}^{\ell },I_{r,%
\limfunc{right}}^{\ell }\right\} }\int_{I_{i}^{\ell +1}}\frac{K_{\flat
}\left( x-\dot{z}_{r}^{\ell }\right) }{K_{\flat }\left( x_{r,\limfunc{left}%
}^{\ell }-\dot{z}_{r}^{\ell }\right) }d\omega _{\ell }\left( x\right)
\right\} ,  \label{redistribution} \\
\left\vert I_{r,\limfunc{right}}^{\ell }\right\vert _{\omega _{\ell +1}} &=&%
\frac{1}{2}\left\{ \left\vert I_{r}^{\ell }\right\vert _{\omega _{\ell
}}+\sum_{i:I_{i}^{\ell +1}\notin \left\{ I_{r,\limfunc{left}}^{\ell },I_{r,%
\limfunc{right}}^{\ell }\right\} }\int_{I_{i}^{\ell +1}}\frac{K_{\flat
}\left( x-\dot{z}_{r}^{\ell }\right) }{K_{\flat }\left( x_{r,\limfunc{left}%
}^{\ell }-\dot{z}_{r}^{\ell }\right) }d\omega _{\ell }\left( x\right)
\right\} .  \notag
\end{eqnarray}

At this point we note that in analogy with (\ref{H flat}) we have%
\begin{equation*}
H_{\flat }\omega _{\ell +1}\left( \dot{z}_{r}^{\ell }\right) =K_{\flat
}\left( x_{r,\limfunc{left}}^{\ell }-\dot{z}_{r}^{\ell }\right) \left(
\left\vert I_{r,\limfunc{left}}^{\ell }\right\vert _{\omega _{\ell
+1}}-\left\vert I_{r,\limfunc{right}}^{\ell }\right\vert _{\omega _{\ell
+1}}\right) +\sum_{i:I_{i}^{\ell +1}\notin \left\{ I_{r,\limfunc{left}%
}^{\ell },I_{r,\limfunc{right}}^{\ell }\right\} }\int_{I_{i}^{\ell
+1}}K_{\flat }\left( x-\dot{z}_{r}^{\ell }\right) d\omega _{\ell +1}\left(
x\right) ,
\end{equation*}%
and then the second line in (\ref{satisfy}) gives%
\begin{equation*}
H_{\flat }\omega _{\ell +1}\left( \dot{z}_{r}^{\ell }\right)
=-\sum_{i:I_{i}^{\ell +1}\notin \left\{ I_{r,\limfunc{left}}^{\ell },I_{r,%
\limfunc{right}}^{\ell }\right\} }\int_{I_{i}^{\ell +1}}K_{\flat }\left( x-%
\dot{z}_{r}^{\ell }\right) d\omega _{\ell }\left( x\right)
+\sum_{i:I_{i}^{\ell +1}\notin \left\{ I_{r,\limfunc{left}}^{\ell },I_{r,%
\limfunc{right}}^{\ell }\right\} }\int_{I_{i}^{\ell +1}}K_{\flat }\left( x-%
\dot{z}_{r}^{\ell }\right) d\omega _{\ell +1}\left( x\right) .
\end{equation*}%
However, the intervals $I_{i}^{\ell +1}\notin \left\{ I_{r,\limfunc{left}%
}^{\ell },I_{r,\limfunc{right}}^{\ell }\right\} $ can be grouped into pairs $%
I_{k}^{\ell +1},I_{k^{\prime }}^{\ell +1}$ that are the children of the same
interval $I_{s}^{\ell }$ with $s\neq r$, and by (\ref{contain}) we have that 
$K_{\flat }\left( x-\dot{z}_{r}^{\ell }\right) $ is \emph{constant} on the
parent $I_{s}^{\ell }$ of $I_{k}^{\ell +1}$ and $I_{k^{\prime }}^{\ell +1}$.
Thus by the first line in (\ref{satisfy}), we conclude that%
\begin{eqnarray*}
&&-\sum_{i=k,k^{\prime }}\int_{I_{i}^{\ell +1}}K_{\flat }\left( x-\dot{z}%
_{r}^{\ell }\right) d\omega _{\ell }\left( x\right) +\sum_{i=k,k^{\prime
}}\int_{I_{i}^{\ell +1}}K_{\flat }\left( x-\dot{z}_{r}^{\ell }\right)
d\omega _{\ell +1}\left( x\right) \\
&=&K_{\flat }\left( \dot{z}_{s}^{\ell }-\dot{z}_{r}^{\ell }\right) \left(
-\left\vert I_{k}^{\ell +1}\right\vert _{\omega _{\ell }}-\left\vert
I_{k^{\prime }}^{\ell +1}\right\vert _{\omega _{\ell }}+\left\vert
I_{k}^{\ell +1}\right\vert _{\omega _{\ell +1}}+\left\vert I_{k^{\prime
}}^{\ell +1}\right\vert _{\omega _{\ell }+1}\right) =0.
\end{eqnarray*}%
Then summing over all $s\neq r$ we obtain the key consequence%
\begin{equation}
H_{\flat }\omega _{\ell +1}\left( \dot{z}_{r}^{\ell }\right) =0.  \label{key}
\end{equation}%
We will also see as an immediate consequence of Lemma \ref{eta is 1/N} below
that the ratios $\frac{\left\vert I_{r,\limfunc{left}}^{\ell }\right\vert
_{\omega _{\ell +1}}}{\left\vert I_{r,\limfunc{left}}^{\ell }\right\vert
_{\omega _{\ell }}}$ and $\frac{\left\vert I_{r,\limfunc{right}}^{\ell
}\right\vert _{\omega _{\ell +1}}}{\left\vert I_{r,\limfunc{right}}^{\ell
}\right\vert _{\omega _{\ell }}}$ are one of the two values $1\pm \frac{1}{N}
$.

Now we take the limit as $\ell \rightarrow \infty $ of the measures $\omega
_{\ell }$ to get the measure 
\begin{equation*}
\widehat{\omega }\equiv \lim_{\ell \rightarrow \infty }\omega _{\ell }
\end{equation*}%
that is supported on the Cantor set $\mathsf{E}^{\left( N\right) }$.
Finally, we define the interval $L_{j}^{k}$ by%
\begin{equation}
L_{j}^{k}\equiv \left[ {\dot{z}}_{j}^{k}-N^{-k-1},{\dot{z}}_{j}^{k}+N^{-k-1}%
\right] ,  \label{def L}
\end{equation}%
and for $N\geq 16$ and $\rho \geq \frac{2}{3}$ as above, we obtain from (\ref%
{key}) and (\ref{contain}) the following crucial estimate for $y\in
L_{j}^{k} $:%
\begin{equation}
H_{\flat }\widehat{\omega }\left( y\right) =0,\ \ \ \ \ \text{for }y\in
L_{j}^{k}.  \label{flatness}
\end{equation}

From (\ref{redistribution}) we have that $\left\vert I_{r,\limfunc{left}%
}^{\ell }\right\vert _{\widehat{\omega }}=\frac{1\pm \eta }{2}\left\vert
I_{r}^{\ell }\right\vert _{\widehat{\omega }}$ and $\left\vert I_{r,\limfunc{%
right}}^{\ell }\right\vert _{\widehat{\omega }}=\frac{1\mp \eta }{2}%
\left\vert I_{r}^{\ell }\right\vert _{\widehat{\omega }}$ for some positive
number $\eta =\eta \left( I_{r}^{\ell }\right) $ depending on the pair $%
\left( \ell ,r\right) $. We end this subsection by showing that $\eta \left(
I_{r}^{\ell }\right) $ is the constant $\frac{1}{N}$. For this it will be
convenient to introduce some tree notation.

\subsubsection{The dyadic tree}

Let $\left( \mathcal{T},\succcurlyeq ,\mathfrak{r},\pi ,\mathfrak{C}\right) $
denote the dyadic tree with relation $\succcurlyeq $, root $\mathfrak{r}$,
parent $\pi :\mathcal{T\setminus }\left\{ \mathfrak{r}\right\} \rightarrow 
\mathcal{T}$ and children $\mathfrak{C}:\mathcal{T}\rightarrow \mathcal{T}%
\times \mathcal{T}$ satisfying the following properties:

\begin{enumerate}
\item $\mathcal{T}$ is countable,

\item $\succcurlyeq $ is a partial order on $\mathcal{T}$, i.e. $%
\succcurlyeq $ is reflexive ($\alpha \succcurlyeq \alpha $), antisymmetric ($%
\alpha \succ \beta $ implies $\beta \not\succ \alpha $), and transitive ($%
\alpha \succcurlyeq \beta $ and $\beta \succcurlyeq \gamma $ implies $\alpha
\succcurlyeq \gamma $),

\item $\mathfrak{r}\succcurlyeq \alpha $ for all $\alpha \in \mathcal{T}$,

\item $\pi \alpha $ is the unique minimal element of $\left\{ \beta \in 
\mathcal{T}:\beta \succ \alpha \right\} $ for all $\alpha \in \mathcal{T}$,

\item $\mathfrak{C}\left( \alpha \right) =\left( \alpha _{-},\alpha
_{+}\right) $ where $\alpha _{-}$ and $\alpha _{+}$ are the unique pair of
maximal elements of $\left\{ \beta \in \mathcal{T}:\beta \prec \alpha
\right\} $.
\end{enumerate}

Define the \emph{distance} $d\left( \alpha ,\beta \right) $ between two
points $\alpha ,\beta $ in $\mathcal{T}$\ to be the number of steps taken to
reach $\beta $ from $\alpha $ by travelling along the unique geodesic $%
\overrightarrow{\alpha \beta }$ connecting $\alpha $ to $\beta $. Then
define the \emph{depth} or \emph{level} of $\alpha $ to be $d\left( \alpha
\right) =d\left( \mathfrak{r},\alpha \right) $. We refer to $\pi \alpha $ as
the \emph{parent} of $\alpha $, to $S_{\alpha }\equiv \left\{ \beta \in 
\mathcal{T}:\beta \prec \alpha \right\} $ as the \emph{successor set} of $%
\alpha $, and to $\alpha _{-}$ and $\alpha _{+}$ as the \emph{left }and\emph{%
\ right child} of $\alpha $ respectively. For convenience we often write $%
\pi ^{1}\alpha =\pi \alpha $, $\pi ^{2}\alpha =\pi \pi \alpha $, etc. and $%
\left( \alpha _{\pm }\right) _{\pm }=\alpha _{\pm \pm }$, $\left( \alpha
_{\pm \pm }\right) _{\pm }=\alpha _{\pm \pm \pm }$ etc. More generally, for $%
\mathbf{\varepsilon }\in \left\{ +,-\right\} ^{m}$, we denote by $\alpha _{%
\mathbf{\varepsilon }}$ the point (called an order $m$-grandchild) $\alpha
_{\pm \pm ...\pm }$ at depth $m$ below $\alpha $ that is given by the choice
of $\pm $ as determined by the sequence $\mathbf{\varepsilon }$. Finally,
for any $\alpha \in \mathcal{T\setminus }\left\{ \mathfrak{r}\right\} $, we
define the sibling $\theta \alpha $ of $\alpha $ to be the other child of $%
\pi \alpha $. The relevant example of a dyadic tree here is the following
dyadic tree defined in terms of our construction of intervals $I_{r}^{\ell }$%
\ above, and with partial order defined in terms of set inclusion $I\prec J$
when $I\subset J$: 
\begin{equation*}
\mathcal{D}=\left\{ I_{r}^{\ell }:\left( \ell ,r\right) \in \mathbb{Z}%
_{+}\times \mathbb{N},1\leq r\leq 2^{\ell }\right\} \sim \left\{ \left( \ell
,r\right) \in \mathbb{Z}_{+}\times \mathbb{N}:1\leq r\leq 2^{\ell }\right\} .
\end{equation*}%
This tree has an additional structure that is derived from its
representation as a collection of intervals on the line - namely the two
children of $I$ occur as a left child $I_{-}$ and a right child $I_{+}$ on
the line. Then $I_{\pm \pm }$ represent the four grandchildren, etc. We
refer to trees with this additional left and right child structure as \emph{%
embedded trees} since they can be drawn in the obvious way in the plane.

\begin{lemma}
\label{eta is 1/N}We have $\eta \left( I\right) =\frac{1}{N}$ for all $I\in 
\mathcal{D}$, and in addition,%
\begin{eqnarray}
\left\vert I_{--}\right\vert _{\widehat{\omega }} &=&\frac{1+\eta }{2}%
\left\vert I_{-}\right\vert _{\widehat{\omega }}\text{ and }\left\vert
I_{-+}\right\vert _{\widehat{\omega }}=\frac{1-\eta }{2}\left\vert
I_{-}\right\vert _{\widehat{\omega }}\ ,  \label{precisely} \\
\left\vert I_{+-}\right\vert _{\widehat{\omega }} &=&\frac{1-\eta }{2}%
\left\vert I_{+}\right\vert _{\widehat{\omega }}\text{ and }\left\vert
I_{++}\right\vert _{\widehat{\omega }}=\frac{1+\eta }{2}\left\vert
I_{+}\right\vert _{\widehat{\omega }}\ .  \notag
\end{eqnarray}
\end{lemma}

\begin{proof}
Let $c_{I}$ be the center of an interval $I$, and let $\mathcal{D}$ be the
tree of all the intervals $I_{r}^{\ell }$ under containment. For $I\in 
\mathcal{D}$ we denote by $I_{-}$ and $I_{+}$ the left and right children of 
$I$ in $D$, and thus $I_{--}$, $I_{-+}$, $I_{+-}$ and $I_{++}$ are the four
grandchildren of $I$ in $\mathcal{D}$ from left to right. We prove the
following statement by induction on $\ell $:%
\begin{equation*}
\eta \left( K\right) =\frac{1}{N}\text{ for }0<d\left( K\right) \leq \ell 
\text{ and }H_{\flat }\widehat{\omega }\left( c_{I}\right) =0\text{ for }%
d\left( I\right) =\ell ,\ \ \ \ \ K,I\in \mathcal{D}.
\end{equation*}%
The case $\ell =0$ holds trivially. Assume now that the induction statement
holds for a fixed $\ell \geq 0$, and fix an interval $I=I_{r}^{\ell }\in 
\mathcal{T}$ at level $\ell $ in the tree, so that $I_{-}=I_{r,\limfunc{left}%
}^{\ell }$ and $I_{+}=I_{r,\limfunc{right}}^{\ell }$. Then we have%
\begin{equation}
0=H_{\flat }\widehat{\omega }\left( c_{I}\right) =H_{\flat }\left( \mathbf{1}%
_{I^{c}}\widehat{\omega }\right) \left( c_{I}\right) +H_{\flat }\left( 
\mathbf{1}_{I_{-}}\widehat{\omega }\right) \left( c_{I}\right) +H_{\flat
}\left( \mathbf{1}_{I_{+}}\widehat{\omega }\right) \left( c_{I}\right) .
\label{then have}
\end{equation}%
Now, using the `flat spots' in the construction of $K_{\flat }$, we obtain
both $H_{\flat }\left( \mathbf{1}_{I^{c}}\widehat{\omega }\right) \left(
c_{I_{+}}\right) =H_{\flat }\left( \mathbf{1}_{I^{c}}\widehat{\omega }%
\right) \left( c_{I}\right) $ and $H_{\flat }\left( \mathbf{1}_{I_{-}}%
\widehat{\omega }\right) \left( c_{I_{+}}\right) =H_{\flat }\left( \mathbf{1}%
_{I_{-}}\widehat{\omega }\right) \left( c_{I}\right) $. Using these two
equalities in the second line below, and then using (\ref{then have}) in the
third line below, we obtain%
\begin{eqnarray*}
H_{\flat }\widehat{\omega }\left( c_{I_{+}}\right) &=&H_{\flat }\left( 
\mathbf{1}_{I^{c}}\widehat{\omega }\right) \left( c_{I_{+}}\right) +H_{\flat
}\left( \mathbf{1}_{I_{-}}\widehat{\omega }\right) \left( c_{I_{+}}\right)
+H_{\flat }\left( \mathbf{1}_{I_{+}}\widehat{\omega }\right) \left(
c_{I_{+}}\right) \\
&=&H_{\flat }\left( \mathbf{1}_{I^{c}}\widehat{\omega }\right) \left(
c_{I}\right) +H_{\flat }\left( \mathbf{1}_{I_{-}}\widehat{\omega }\right)
\left( c_{I}\right) +H_{\flat }\left( \mathbf{1}_{I_{+}}\widehat{\omega }%
\right) \left( c_{I_{+}}\right) \\
&=&-H_{\flat }\left( \mathbf{1}_{I_{+}}\widehat{\omega }\right) \left(
c_{I}\right) +H_{\flat }\left( \mathbf{1}_{I_{+}}\widehat{\omega }\right)
\left( c_{I_{+}}\right) \\
&=&-H_{\flat }\left( \mathbf{1}_{I_{+}}\widehat{\omega }\right) \left(
c_{I}\right) +H_{\flat }\left( \mathbf{1}_{I_{+-}}\widehat{\omega }\right)
\left( c_{I_{+}}\right) +H_{\flat }\left( \mathbf{1}_{I_{++}}\widehat{\omega 
}\right) \left( c_{I_{+}}\right) \\
&=&-N^{\ell }\left\vert I_{+}\right\vert _{\widehat{\omega }}-N^{\ell
+1}\left\vert I_{+-}\right\vert _{\widehat{\omega }}+N^{\ell +1}\left\vert
I_{++}\right\vert _{\widehat{\omega }}\ ,
\end{eqnarray*}%
and where in the last line above we have used the `flat spots' in the
construction of $K_{\flat }$. Thus in order to achieve $H_{\flat }\widehat{%
\omega }\left( c_{I_{+}}\right) =0$ we need%
\begin{equation*}
\left\vert I_{++}\right\vert _{\widehat{\omega }}-\left\vert
I_{+-}\right\vert _{\widehat{\omega }}=\frac{1}{N}\left\vert
I_{+}\right\vert _{\widehat{\omega }},
\end{equation*}%
and combined with the requirement that $\left\vert I_{++}\right\vert _{%
\widehat{\omega }}+\left\vert I_{+-}\right\vert _{\widehat{\omega }%
}=\left\vert I_{+}\right\vert _{\widehat{\omega }}$, we obtain%
\begin{equation*}
\left\vert I_{++}\right\vert _{\widehat{\omega }}=\frac{1+\frac{1}{N}}{2}%
\left\vert I_{+}\right\vert _{\widehat{\omega }}\text{ and }\left\vert
I_{+-}\right\vert _{\widehat{\omega }}=\frac{1-\frac{1}{N}}{2}\left\vert
I_{+}\right\vert _{\widehat{\omega }}\ ,
\end{equation*}%
which establishes the second line in (\ref{precisely}) for $I$, and also
that $\eta \left( I_{r,\limfunc{right}}^{\ell }\right) =\eta \left(
I_{+}\right) =\frac{1}{N}$ and $H_{\flat }\widehat{\omega }\left(
c_{I_{+}}\right) =0$. Similarly, we have the first line in (\ref{precisely})
for $I$ and $\eta \left( I_{r,\limfunc{left}}^{\ell }\right) =\frac{1}{N}$
and $H_{\flat }\widehat{\omega }\left( c_{I_{-}}\right) =0$, which completes
the proof of the inductive step.
\end{proof}

\subsection{The reweighted measures $\protect\widehat{\dot{\protect\sigma}}$
and $\protect\widehat{\protect\sigma }$}

Recall that we view the collection of intervals $\left\{ I_{j}^{k}\right\} $
as an embedded tree and use the usual parent/child terminology for this
tree. In Lemma \ref{eta is 1/N} above we showed that the pattern of
redistribution is given by 
\begin{eqnarray}
\left\vert I_{r,\limfunc{left}}^{\ell }\right\vert _{\widehat{\omega }} &=&%
\frac{1+\eta }{2}\left\vert I_{r}^{\ell }\right\vert _{\widehat{\omega }}%
\text{ and }\left\vert I_{r,\limfunc{right}}^{\ell }\right\vert _{\widehat{%
\omega }}=\frac{1-\eta }{2}\left\vert I_{r}^{\ell }\right\vert _{\widehat{%
\omega }}\text{ if }I_{r}^{\ell }=I_{s,\limfunc{left}}^{\ell -1}\ ,
\label{pattern} \\
\left\vert I_{r,\limfunc{left}}^{\ell }\right\vert _{\widehat{\omega }} &=&%
\frac{1-\eta }{2}\left\vert I_{r}^{\ell }\right\vert _{\widehat{\omega }}%
\text{ and }\left\vert I_{r,\limfunc{right}}^{\ell }\right\vert _{\widehat{%
\omega }}=\frac{1+\eta }{2}\left\vert I_{r}^{\ell }\right\vert _{\widehat{%
\omega }}\text{ if }I_{r}^{\ell }=I_{s,\limfunc{right}}^{\ell -1}\ ,  \notag
\end{eqnarray}%
where $I_{s}^{\ell -1}=\pi I_{r}^{\ell }$ is the parent of $I_{r}^{\ell }$
and where $\eta =\frac{1}{N}$. Thus we see that on the children of $%
I_{r}^{\ell }$ the measure $\widehat{\omega }$ is redistributed \textbf{away
from the center of the parent of }$I_{r}^{\ell }$, where the factors $\frac{%
1\pm \eta }{2}$ can be remembered via the mnemonic 4-tuple%
\begin{equation}
{\LARGE +,-,-,+.}  \label{mnemonic}
\end{equation}%
As a consequence of this redistribution, we see that the measures $%
\left\vert I_{r}^{\ell }\right\vert _{\widehat{\omega }}$ of the intervals $%
I_{r}^{\ell }$ at level $\ell $ are equal when $\ell =1$ and thereafter take
on the values 
\begin{equation}
\left\vert I_{r}^{\ell }\right\vert _{\widehat{\omega }}=\frac{1}{2}\left( 
\frac{1+\eta }{2}\right) ^{H}\left( \frac{1-\eta }{2}\right) ^{T}=\left(
1+\eta \right) ^{H}\left( 1-\eta \right) ^{T}\left\vert I_{r}^{\ell
}\right\vert _{\omega },\ \ \ \ \ H+T=\ell -1,  \label{heads and tails}
\end{equation}%
where $H=H\left( r\right) $ and $T=T\left( r\right) $ depend on $r$, and can
be thought of as the number of heads and tails respectively in $\ell -1$
tosses of a fair coin.

We now define the weights $\widehat{s}_{j}^{k}$ so that with 
\begin{equation*}
\widehat{\dot{\sigma}}\equiv \sum_{k,j}\widehat{s}_{j}^{k}\delta _{{\dot{z}}%
_{j}^{k}}
\end{equation*}%
we have%
\begin{equation}
\frac{\left\vert I_{j}^{k}\right\vert _{\widehat{\omega }}\widehat{s}_{j}^{k}%
}{N^{-2k}}=\frac{\left\vert I_{j}^{k}\right\vert _{\widehat{\omega }%
}\left\vert G_{j}^{k}\right\vert _{\widehat{\dot{\sigma}}}}{\left\vert
I_{j}^{k}\right\vert ^{2}}=1.  \label{def s hat}
\end{equation}%
Then we replace the point mass $\delta _{{\dot{z}}_{j}^{k}}$ at ${\dot{z}}%
_{j}^{k}$ in the definition of $\widehat{\dot{\sigma}}$ with the
approximation $\frac{1}{\left\vert L_{j}^{k}\right\vert }\mathbf{1}%
_{L_{j}^{k}}$ and define the resulting reweighted measure $\widehat{{\sigma }%
}$ by%
\begin{equation*}
\widehat{{\sigma }}\equiv \sum_{k,j}\widehat{s}_{j}^{k}\frac{1}{\left\vert
L_{j}^{k}\right\vert }\mathbf{1}_{L_{j}^{k}}.
\end{equation*}%
We now investigate properties of the measure pairs $\left( \widehat{\dot{%
\sigma}},\widehat{\omega }\right) $ and $\left( \widehat{\sigma },\widehat{%
\omega }\right) $ relative to the flattened Hilbert transform $H_{\flat }$.

\section{Testing and side conditions}

In this section we establish the Muckenhoupt/NTV $\mathcal{A}_{2}$
conditions for $\left( \widehat{\dot{\sigma}},\widehat{\omega }\right) $, as
well as the testing conditions for $H_{\flat }$ relative to the weight pair $%
\left( \widehat{\dot{\sigma}},\widehat{\omega }\right) $. Then we establish
both the forward and backward energy conditions for the weight pair $\left( 
\widehat{\dot{\sigma}},\widehat{\omega }\right) $, and use our $T1$ theorem
in \cite{SaShUr7} to conclude that the two weight norm inequality holds for
the flattened Hilbert transform $H_{\flat }$ relative to the weight pair $%
\left( \widehat{\dot{\sigma}},\widehat{\omega }\right) $. For the
convenience of the reader we recall the relevant $1$-dimensional version of
Theorem 2.6 from \cite{SaShUr7}.

\begin{theorem}
\label{T1 theorem}Suppose that $T$ is a standard Calder\'{o}n-Zygmund
operator on the real line $\mathbb{R}$, and that $\omega $ and $\sigma $ are
positive Borel measures on $\mathbb{R}$ without common point masses. Set $%
T_{\sigma }f=T\left( f\sigma \right) $ for any smooth truncation of $%
T_{\sigma }$.

\begin{enumerate}
\item The operator $T_{\sigma }$ is bounded from $L^{2}\left( \sigma \right) 
$ to $L^{2}\left( \omega \right) $, i.e. 
\begin{equation}
\left\Vert T_{\sigma }f\right\Vert _{L^{2}\left( \omega \right) }\leq 
\mathfrak{N}_{T_{\sigma }}\left\Vert f\right\Vert _{L^{2}\left( \sigma
\right) },  \label{two weight}
\end{equation}%
uniformly in smooth truncations of $T$, and moreover%
\begin{equation*}
\mathfrak{N}_{T_{\sigma }}\leq C_{\alpha }\left( \sqrt{\mathcal{A}_{2}+%
\mathcal{A}_{2}^{\ast }}+\mathfrak{T}_{T}+\mathfrak{T}_{T}^{\ast }+\mathcal{E%
}+\mathcal{E}^{\ast }+\mathcal{WBP}_{T}\right) ,
\end{equation*}%
provided that the two dual $\mathcal{A}_{2}$ conditions (\ref{A2}) hold; and
the two dual testing conditions for $T$ and $T^{\ast }$ hold,%
\begin{eqnarray}
\int_{I}\lvert T(\mathbf{1}_{I}\sigma )\rvert ^{2}\;\omega (dx) &\leq &%
\mathfrak{T}^{2}\left\vert I\right\vert _{\sigma }\,,\ \ \ \ \ \ \ \ \ \ \ \
\ \ \ \ \ \ \ \ \left( \mathfrak{T}\right) ,  \label{T testing} \\
\int_{I}\lvert T(\mathbf{1}_{I}\omega )\rvert ^{2}\;\sigma (dx) &\leq &(%
\mathfrak{T}^{\ast })^{2}\left\vert I\right\vert _{\omega }\,,\ \ \ \ \ \ \
\ \ \ \ \ \ \ \ \ \ \ \ \ \left( \mathfrak{T}^{\ast }\right) ,  \notag
\end{eqnarray}%
for all intervals $I$; the weak boundedness property for $T$ holds,%
\begin{equation*}
\left\vert \int_{J}T\left( \mathbf{1}_{I}{\sigma }\right) d\omega
\right\vert \leq \mathcal{WBP}_{T}\left( {\sigma },\omega \right) \sqrt{%
\left\vert J\right\vert _{\omega }\left\vert I\right\vert _{{\sigma }}},
\end{equation*}%
for all intervals $I,J\,$with $J\subset 3I$ and $I\subset 3J$; and provided
that the two dual energy conditions (\ref{energy condition}) and (\ref{dual
energy condition}) hold.

\item Conversely, suppose that $T$ is a Calder\'{o}n-Zygmund operator with
standard kernel $K$ and that in addition, there is $c>0$ such that 
\begin{equation}
\left\vert K\left( x,x+t\right) \right\vert \geq c\left\vert t\right\vert
^{-n},\ \ \ \ \ t\in \mathbb{R}.  \label{Ktalpha}
\end{equation}%
Furthermore, assume that $T$ is bounded from $L^{2}\left( \sigma \right) $
to $L^{2}\left( \omega \right) $, 
\begin{equation*}
\left\Vert T_{\sigma }f\right\Vert _{L^{2}\left( \omega \right) }\leq 
\mathfrak{N}_{T}\left\Vert f\right\Vert _{L^{2}\left( \sigma \right) }.
\end{equation*}%
Then the $\mathcal{A}_{2}$ condition holds, and moreover,%
\begin{equation*}
\sqrt{\mathcal{A}_{2}+\mathcal{A}_{2}^{\ast }}\leq C\mathfrak{N}_{T}.
\end{equation*}
\end{enumerate}
\end{theorem}

Next we must extend the boundedness of the the flattened Hilbert transform $%
H_{\flat }$ to the weight pair $\left( \widehat{\sigma },\widehat{\omega }%
\right) $ as well. For this we invoke the `flatness' of the kernel of $%
H_{\flat }$. In fact, using that $H_{\flat }\left( g\widehat{\omega }\right) 
$ is constant on the support of $\widehat{{\sigma }}$ for all $g\in
L^{2}\left( \widehat{\omega }\right) $, and that the support of $\widehat{{%
\sigma }}$ contains the support of $\widehat{{\dot{\sigma}}}$, we will
conclude that 
\begin{equation*}
\int \left\vert H_{\flat }\left( g\widehat{\omega }\right) \right\vert ^{2}d%
\widehat{{\sigma }}=\int \left\vert H_{\flat }\left( g\widehat{\omega }%
\right) \right\vert ^{2}d\widehat{{\dot{\sigma}}},
\end{equation*}%
thus obtaining the norm inequality for $H_{\flat }$ relative to the weight
pair $\left( \widehat{{\sigma }},\widehat{\omega }\right) $.

Finally we will show that the backward energy condition fails for the weight
pair $\left( \widehat{{\sigma }},\widehat{\omega }\right) $ since the
support of $\widehat{{\sigma }}$ consists of a countable union of \emph{%
intervals} rather than \emph{point masses} as is the case for the measure $%
\widehat{{\dot{\sigma}}}$. It is the fact that the energy of a measure on an
interval is positive provided the measure has positive density on the
interval, that accounts for the failure of the backward energy condition.
And it is the fact that the flattened Hilbert transform of $\widehat{\omega }
$ \emph{vanishes} on the support of $\widehat{{\sigma }}$ that permits the
backward testing condition to hold \textbf{despite} the failure of the
backward energy condition. \emph{This is the key reason for the failure of
the energy condition to be necessary for the norm inequality of an elliptic
singular integral.}

We can of course view this one-dimensional measure pair as living in the
plane, and then the flattened Hilbert transform is the restriction to the $x$%
-axis of a suitable flattening $R_{1\flat }$ of the Riesz transform $R_{1}$.
Thus the vector $\mathbf{R}=\left( R_{1\flat },R_{2}\right) $ is a strongly
elliptic singular integral that satisfies the norm inequality relative to $%
\left( \widehat{{\sigma }},\widehat{\omega }\right) $ but fails the backward
energy condition, thus extending the failure of necessity of energy
conditions to higher dimensions. One can easily extend this failure to hold
for flattened fractional Riesz transforms $\mathbf{R}^{\alpha ,n}$ in the
Euclidean space $\mathbb{R}^{n}$ as well.

\subsection{The Muckenhoupt/NTV condition}

We now verify the two-tailed Muckenhoupt/NTV condition $\mathcal{A}_{2}$ for
the weight pair $\left( \widehat{\dot{\sigma}},\widehat{\omega }\right) $.
We fix an interval $I_{r}^{\ell }$ at level $\ell $. Recall that $\eta =%
\frac{1}{N}$ by Lemma \ref{eta is 1/N}. Consider the two-tailed Muckenhoupt
condition $\mathcal{A}_{2}$. We write%
\begin{equation*}
\mathrm{P}\left( Q,\mu \right) =\int_{\mathbb{R}}\frac{\left\vert
Q\right\vert }{\left( \left\vert Q\right\vert +\left\vert y-c_{Q}\right\vert
\right) ^{2}}d\mu \left( y\right) =\int_{\mathbb{R}}\frac{1}{\left\vert
Q\right\vert }\frac{1}{\left( 1+\frac{\left\vert y-c_{Q}\right\vert }{%
\left\vert Q\right\vert }\right) ^{2}}d\mu \left( y\right) ,
\end{equation*}%
where $c_{Q}$ is the center of the interval $Q$, and we will use the
pointwise estimates%
\begin{eqnarray}
\frac{1}{\left\vert Q\right\vert }\frac{1}{\left( 1+\frac{\left\vert
y-c_{Q}\right\vert }{\left\vert Q\right\vert }\right) ^{2}} &\approx
&\sum_{k=0}^{\infty }\frac{1}{N^{k}}\frac{1}{\left\vert N^{k}Q\right\vert }%
\mathbf{1}_{N^{k}Q}\left( y\right)  \label{pointwise} \\
&\approx &\frac{1}{\left\vert Q\right\vert }\mathbf{1}_{Q}\left( y\right)
+\sum_{k=1}^{\infty }\frac{1}{N^{k}}\frac{1}{\left\vert N^{k}Q\right\vert }%
\mathbf{1}_{N^{k}Q\setminus N^{k-1}Q}\left( y\right) ,  \notag
\end{eqnarray}%
where the implied constants in $\approx $ depend only on $N$. We compute
both $\mathrm{P}\left( I_{r}^{\ell },\widehat{\omega }\right) $ and $\mathrm{%
P}\left( I_{r}^{\ell },\widehat{\dot{\sigma}}\right) $, beginning with the
simpler factor $\mathrm{P}\left( I_{r}^{\ell },\widehat{\omega }\right) $.

Using the second line in (\ref{pointwise}) we have%
\begin{equation*}
\mathrm{P}\left( I_{r}^{\ell },\widehat{\omega }\right) \approx \frac{%
\left\vert I_{r}^{\ell }\right\vert _{\widehat{\omega }}}{\left\vert
I_{r}^{\ell }\right\vert }+\frac{1}{N^{2}}\frac{\left\vert \theta
I_{r}^{\ell }\right\vert _{\widehat{\omega }}}{\left\vert I_{r}^{\ell
}\right\vert }+\frac{1}{N^{4}}\frac{\left\vert \theta \pi I_{r}^{\ell
}\right\vert _{\widehat{\omega }}}{\left\vert I_{r}^{\ell }\right\vert }+%
\frac{1}{N^{6}}\frac{\left\vert \theta \pi ^{\left( 2\right) }I_{r}^{\ell
}\right\vert _{\widehat{\omega }}}{\left\vert I_{r}^{\ell }\right\vert }+...
\end{equation*}%
Our construction above gives the inequalities 
\begin{equation}
\left\vert \theta \pi ^{\left( k\right) }I_{r}^{\ell }\right\vert _{\widehat{%
\omega }}\leq \frac{1+\eta }{1-\eta }\left\vert \pi ^{\left( k\right)
}I_{r}^{\ell }\right\vert _{\widehat{\omega }}\text{ and }\left\vert \pi
^{\left( k+1\right) }I_{r}^{\ell }\right\vert _{\widehat{\omega }}\leq \frac{%
2}{1-\eta }\left\vert \pi ^{\left( k\right) }I_{r}^{\ell }\right\vert _{%
\widehat{\omega }}\ ,\ \ \ \ \ k\geq 0,  \label{inequalities parents}
\end{equation}%
and these inequalities show that for $N\geq 4$ we have%
\begin{eqnarray*}
\mathrm{P}\left( I_{r}^{\ell },\widehat{\omega }\right) &\lesssim &\frac{%
\left\vert I_{r}^{\ell }\right\vert _{\widehat{\omega }}}{\left\vert
I_{r}^{\ell }\right\vert }\left\{ 1+\frac{1}{N^{2}}\left( \frac{1+\eta }{%
1-\eta }\right) +\frac{1}{N^{4}}\left( \frac{1+\eta }{1-\eta }\right) \left( 
\frac{2}{1-\eta }\right) +\frac{1}{N^{6}}\left( \frac{1+\eta }{1-\eta }%
\right) \left( \frac{2}{1-\eta }\right) ^{2}+...\right\} \\
&=&\frac{\left\vert I_{r}^{\ell }\right\vert _{\widehat{\omega }}}{%
\left\vert I_{r}^{\ell }\right\vert }\left\{ 1+\frac{1}{N^{2}}\left( \frac{%
1+\eta }{1-\eta }\right) \left[ 1+\frac{1}{N^{2}}\left( \frac{2}{1-\eta }%
\right) +\frac{1}{N^{4}}\left( \frac{2}{1-\eta }\right) ^{2}+...\right]
\right\} \\
&=&\frac{\left\vert I_{r}^{\ell }\right\vert _{\widehat{\omega }}}{%
\left\vert I_{r}^{\ell }\right\vert }\left\{ 1+\frac{1}{N^{2}}\left( \frac{%
1+\eta }{1-\eta }\right) \frac{1}{1-\frac{1}{N^{2}}\left( \frac{2}{1-\eta }%
\right) }\right\} =\frac{\left\vert I_{r}^{\ell }\right\vert _{\widehat{%
\omega }}}{\left\vert I_{r}^{\ell }\right\vert }\left\{ 1+\frac{1}{N^{2}}%
\frac{N+1}{N-1-\frac{2}{N}}\right\} ,
\end{eqnarray*}%
and hence that%
\begin{equation*}
\mathrm{P}\left( I_{r}^{\ell },\widehat{\omega }\right) \approx \frac{%
\left\vert I_{r}^{\ell }\right\vert _{\widehat{\omega }}}{\left\vert
I_{r}^{\ell }\right\vert }.
\end{equation*}

Similarly, using the first line in (\ref{pointwise}), we have%
\begin{equation}
\mathrm{P}\left( I_{r}^{\ell },\widehat{\dot{\sigma}}\right) \approx \frac{%
\left\vert I_{r}^{\ell }\right\vert _{\widehat{\dot{\sigma}}}}{\left\vert
I_{r}^{\ell }\right\vert }+\frac{1}{N^{2}}\frac{\left\vert \pi I_{r}^{\ell
}\right\vert _{\widehat{\dot{\sigma}}}}{\left\vert I_{r}^{\ell }\right\vert }%
+\frac{1}{N^{4}}\frac{\left\vert \pi ^{\left( 2\right) }I_{r}^{\ell
}\right\vert _{\widehat{\dot{\sigma}}}}{\left\vert I_{r}^{\ell }\right\vert }%
+\frac{1}{N^{6}}\frac{\left\vert \pi ^{\left( 3\right) }I_{r}^{\ell
}\right\vert _{\widehat{\dot{\sigma}}}}{\left\vert I_{r}^{\ell }\right\vert }%
+...  \label{sim}
\end{equation}%
Now we show that $\left\vert I_{r}^{\ell }\right\vert _{\widehat{\dot{\sigma}%
}}$ is comparable to $\left\vert G_{r}^{\ell }\right\vert _{\widehat{\dot{%
\sigma}}}=\widehat{s}_{r}^{\ell }$. For this it will be convenient to denote 
$G_{r}^{\ell }$ by $G\left( I_{r}^{\ell }\right) $ and $\widehat{s}%
_{r}^{\ell }$ by $\widehat{s}\left( I_{r}^{\ell }\right) $ in order to use
the notations $G\left( \pi ^{\left( k\right) }I_{r}^{\ell }\right) $ and $%
\widehat{s}\left( \pi ^{\left( k\right) }I_{r}^{\ell }\right) $.

\begin{lemma}
\label{sigma dot}For $N\geq 16$ we have%
\begin{equation*}
\left\vert I_{r}^{\ell }\right\vert _{\widehat{\dot{\sigma}}}=\frac{N^{2}-1}{%
N^{2}-5}\left\vert G_{r}^{\ell }\right\vert _{\widehat{\dot{\sigma}}}\text{
for all }\ell \geq 1\text{ and }1\leq r\leq 2^{\ell }.
\end{equation*}
\end{lemma}

\begin{proof}
Let $I=I_{r}^{\ell }\in \mathcal{D}$ with $\ell \geq 1$ and write $\widehat{s%
}\left( I\right) =\widehat{s}\left( I_{r}^{\ell }\right) \equiv \widehat{s}%
_{r}^{\ell }$. Recall the definition of $\widehat{s}_{r}^{\ell }=\left\vert
G_{r}^{\ell }\right\vert _{\widehat{\dot{\sigma}}}$ is given by $\left\vert
I_{r}^{\ell }\right\vert _{\widehat{\omega }}\left\vert G_{r}^{\ell
}\right\vert _{\widehat{\dot{\sigma}}}=\left\vert I_{r}^{\ell }\right\vert
^{2}$, and so we have%
\begin{eqnarray*}
\left\vert I\right\vert _{\widehat{\dot{\sigma}}} &=&\widehat{s}\left(
I\right) +\left\{ \widehat{s}\left( I_{-}\right) +\widehat{s}\left(
I_{+}\right) \right\} +\left\{ \widehat{s}\left( I_{--}\right) +\widehat{s}%
\left( I_{-+}\right) +\widehat{s}\left( I_{+-}\right) +\widehat{s}\left(
I_{++}\right) \right\} +... \\
&=&\frac{\left\vert I\right\vert ^{2}}{\left\vert I\right\vert _{\widehat{%
\omega }}}+\left\{ \frac{\left\vert I_{-}\right\vert ^{2}}{\left\vert
I_{-}\right\vert _{\widehat{\omega }}}+\frac{\left\vert I_{+}\right\vert ^{2}%
}{\left\vert I_{+}\right\vert _{\widehat{\omega }}}\right\} +\left\{ \frac{%
\left\vert I_{--}\right\vert ^{2}}{\left\vert I_{--}\right\vert _{\widehat{%
\omega }}}+\frac{\left\vert I_{-+}\right\vert ^{2}}{\left\vert
I_{-+}\right\vert _{\widehat{\omega }}}+\frac{\left\vert I_{+-}\right\vert
^{2}}{\left\vert I_{+-}\right\vert _{\widehat{\omega }}}+\frac{\left\vert
I_{++}\right\vert ^{2}}{\left\vert I_{++}\right\vert _{\widehat{\omega }}}%
\right\} +...
\end{eqnarray*}%
and if we combine $\left\vert J_{+}\right\vert _{\widehat{\omega }}=\frac{%
1\pm \eta }{2}\left\vert J\right\vert _{\widehat{\omega }}$ and $\left\vert
J_{-}\right\vert _{\widehat{\omega }}=\frac{1\mp \eta }{2}\left\vert
J\right\vert _{\widehat{\omega }}$ with 
\begin{equation*}
\frac{\left\vert J\right\vert _{\widehat{\omega }}}{\left\vert
J_{-}\right\vert _{\widehat{\omega }}}+\frac{\left\vert J\right\vert _{%
\widehat{\omega }}}{\left\vert J_{+}\right\vert _{\widehat{\omega }}}=\frac{2%
}{1+\eta }+\frac{2}{1-\eta }=\frac{4}{1-\eta ^{2}}\text{ and }\left\vert
I_{\pm }\right\vert ^{2}=\frac{\left\vert I\right\vert ^{2}}{N^{2}},
\end{equation*}%
we obtain%
\begin{eqnarray*}
\left\vert I_{r}^{\ell }\right\vert _{\widehat{\dot{\sigma}}} &=&\left\vert
I\right\vert _{\widehat{\dot{\sigma}}}=\frac{\left\vert I\right\vert ^{2}}{%
\left\vert I\right\vert _{\widehat{\omega }}}+\frac{\left\vert I\right\vert
^{2}}{\left\vert I\right\vert _{\widehat{\omega }}}\left\{ \frac{4}{%
N^{2}\left( 1-\eta ^{2}\right) }\right\} +\left\{ \frac{\left\vert
I_{-}\right\vert ^{2}}{\left\vert I_{-}\right\vert _{\widehat{\omega }}}%
\left( \frac{4}{N^{2}\left( 1-\eta ^{2}\right) }\right) +\frac{\left\vert
I_{+}\right\vert ^{2}}{\left\vert I_{+}\right\vert _{\widehat{\omega }}}%
\left( \frac{4}{N^{2}\left( 1-\eta ^{2}\right) }\right) \right\} +... \\
&=&\frac{\left\vert I\right\vert ^{2}}{\left\vert I\right\vert _{\widehat{%
\omega }}}+\frac{\left\vert I\right\vert ^{2}}{\left\vert I\right\vert _{%
\widehat{\omega }}}\left\{ \frac{4}{N^{2}\left( 1-\eta ^{2}\right) }\right\}
+\left\{ \left( \frac{\left\vert I_{-}\right\vert ^{2}}{\left\vert
I_{-}\right\vert _{\widehat{\omega }}}+\frac{\left\vert I_{+}\right\vert ^{2}%
}{\left\vert I_{+}\right\vert _{\widehat{\omega }}}\right) \left( \frac{4}{%
N^{2}\left( 1-\eta ^{2}\right) }\right) \right\} +... \\
&=&\frac{\left\vert I\right\vert ^{2}}{\left\vert I\right\vert _{\widehat{%
\omega }}}\left[ 1+\left\{ \frac{4}{N^{2}\left( 1-\eta ^{2}\right) }\right\}
+\left\{ \left( \frac{4}{N^{2}\left( 1-\eta ^{2}\right) }\right)
^{2}\right\} +...\right] \\
&=&\frac{\left\vert I\right\vert ^{2}}{\left\vert I\right\vert _{\widehat{%
\omega }}}\left( \frac{1}{1-\frac{4}{N^{2}\left( 1-\frac{1}{N^{2}}\right) }}%
\right) =\widehat{s}\left( I\right) \left( \frac{N^{2}-1}{N^{2}-5}\right) =%
\frac{N^{2}-1}{N^{2}-5}\left\vert G_{r}^{\ell }\right\vert _{\widehat{\dot{%
\sigma}}}.
\end{eqnarray*}
\end{proof}

From the definition of $\widehat{s}_{r}^{\ell }=\left\vert G_{r}^{\ell
}\right\vert _{\widehat{\dot{\sigma}}}$ given by $\left\vert I_{r}^{\ell
}\right\vert _{\widehat{\omega }}\left\vert G_{r}^{\ell }\right\vert _{%
\widehat{\dot{\sigma}}}=\left\vert I_{r}^{\ell }\right\vert ^{2}$ once more,
and the inequality $\left\vert \pi ^{\left( j\right) }I_{r}^{\ell
}\right\vert _{\widehat{\omega }}\leq \frac{1+\eta }{2}\left\vert \pi
^{\left( j+1\right) }I_{r}^{\ell }\right\vert _{\widehat{\omega }}$, we
obtain:%
\begin{eqnarray}
\left\vert G\left( \pi ^{\left( k\right) }I_{r}^{\ell }\right) \right\vert _{%
\widehat{\dot{\sigma}}} &=&\frac{\left\vert \pi ^{\left( k\right)
}I_{r}^{\ell }\right\vert ^{2}}{\left\vert \pi ^{\left( k\right)
}I_{r}^{\ell }\right\vert _{\widehat{\omega }}}\leq \frac{1+\eta }{2}N^{2}%
\frac{\left\vert \pi ^{\left( k-1\right) }I_{r}^{\ell }\right\vert ^{2}}{%
\left\vert \pi ^{\left( k-1\right) }I_{r}^{\ell }\right\vert _{\widehat{%
\omega }}}  \label{inequalities parents'} \\
&&\vdots  \notag \\
&\leq &\left( \frac{1+\eta }{2}N^{2}\right) ^{k}\frac{\left\vert I_{r}^{\ell
}\right\vert ^{2}}{\left\vert I_{r}^{\ell }\right\vert _{\widehat{\omega }}}%
=\left( \frac{1+\eta }{2}N^{2}\right) ^{k}\left\vert G\left( I_{r}^{\ell
}\right) \right\vert _{\widehat{\dot{\sigma}}}\ .  \notag
\end{eqnarray}%
Thus from (\ref{sim}), Lemma \ref{sigma dot}\ and (\ref{inequalities
parents'}) we have%
\begin{eqnarray*}
\mathrm{P}\left( I_{r}^{\ell },\widehat{\dot{\sigma}}\right) &\lesssim &%
\frac{\left\vert G\left( I_{r}^{\ell }\right) \right\vert _{\widehat{\dot{%
\sigma}}}}{\left\vert I_{r}^{\ell }\right\vert }+\frac{1}{N^{2}}\frac{%
\left\vert G\left( \pi I_{r}^{\ell }\right) \right\vert _{\widehat{\dot{%
\sigma}}}}{\left\vert I_{r}^{\ell }\right\vert }+\frac{1}{N^{4}}\frac{%
\left\vert G\left( \pi ^{\left( 2\right) }I_{r}^{\ell }\right) \right\vert _{%
\widehat{\dot{\sigma}}}}{\left\vert I_{r}^{\ell }\right\vert }+...+\frac{1}{%
N^{2\ell }}\frac{\left\vert G\left( \pi ^{\left( \ell \right) }I_{r}^{\ell
}\right) \right\vert _{\widehat{\dot{\sigma}}}}{\left\vert I_{r}^{\ell
}\right\vert } \\
&\lesssim &\frac{\left\vert G\left( I_{r}^{\ell }\right) \right\vert _{%
\widehat{\dot{\sigma}}}}{\left\vert I_{r}^{\ell }\right\vert }+\frac{1}{N^{2}%
}\left( \frac{1+\eta }{2}N^{2}\right) \frac{\left\vert G\left( I_{r}^{\ell
}\right) \right\vert _{\widehat{\dot{\sigma}}}}{\left\vert I_{r}^{\ell
}\right\vert }+\frac{1}{N^{4}}\left( \frac{1+\eta }{2}N^{2}\right) ^{2}\frac{%
\left\vert G\left( I_{r}^{\ell }\right) \right\vert _{\widehat{\dot{\sigma}}}%
}{\left\vert I_{r}^{\ell }\right\vert } \\
&&+...+\frac{1}{N^{2\ell }}\left( \frac{1+\eta }{2}N^{2}\right) ^{\ell }%
\frac{\left\vert G\left( I_{r}^{\ell }\right) \right\vert _{\widehat{\dot{%
\sigma}}}}{\left\vert I_{r}^{\ell }\right\vert } \\
&\lesssim &\frac{\left\vert G\left( I_{r}^{\ell }\right) \right\vert _{%
\widehat{\dot{\sigma}}}}{\left\vert I_{r}^{\ell }\right\vert }\left\{ 1+%
\frac{1}{N^{2}}\left( \frac{1+\eta }{2}N^{2}\right) +\frac{1}{N^{4}}\left( 
\frac{1+\eta }{2}N^{2}\right) ^{2}+...\right\} \\
&=&\frac{\left\vert G\left( I_{r}^{\ell }\right) \right\vert _{\widehat{\dot{%
\sigma}}}}{\left\vert I_{r}^{\ell }\right\vert }\left\{ 1+\left( \frac{%
1+\eta }{2}\right) +\left( \frac{1+\eta }{2}\right) ^{2}+...\right\} =\frac{%
\left\vert G\left( I_{r}^{\ell }\right) \right\vert _{\widehat{\dot{\sigma}}}%
}{\left\vert I_{r}^{\ell }\right\vert }\left\{ 1+\frac{1}{1-\frac{1+\eta }{2}%
}\right\} ,
\end{eqnarray*}%
and so $\mathrm{P}\left( I_{r}^{\ell },\widehat{\dot{\sigma}}\right) \approx 
\frac{\left\vert G\left( I_{r}^{\ell }\right) \right\vert _{\widehat{\dot{%
\sigma}}}}{\left\vert I_{r}^{\ell }\right\vert }$. Thus 
\begin{equation*}
\mathrm{P}\left( I_{r}^{\ell },\widehat{\omega }\right) \mathrm{P}\left(
I_{r}^{\ell },\widehat{\dot{\sigma}}\right) \approx \frac{\left\vert
I_{r}^{\ell }\right\vert _{\widehat{\omega }}\left\vert G\left( I_{r}^{\ell
}\right) \right\vert _{\widehat{\dot{\sigma}}}}{\left\vert I_{r}^{\ell
}\right\vert ^{2}}=1.
\end{equation*}%
This verifies the $\mathcal{A}_{2}$ condition for the weight pair $\left( 
\widehat{\dot{\sigma}},\widehat{\omega }\right) $ when tested over the
intervals $I_{r}^{\ell }$. This is easily seen to be enough since $\mathrm{P}
$ is a positive operator and Lebesgue measure is doubling. Indeed, given an
arbitrary interval $I\subset \left[ 0,1\right] $ that meets the support of
both measures $\widehat{\omega }$ and $\widehat{\dot{\sigma}}$, choose $\ell 
$ so that $\frac{1}{N^{\ell +1}}<\left\vert I\right\vert \leq \frac{1}{%
N^{\ell }}$ and then choose $I_{r}^{\ell }$ so that $I\cap I_{r}^{\ell }\neq
\emptyset $. Since $I\subset 3I_{r}^{\ell }$ and $\left\vert I\right\vert
\approx \left\vert 3I_{r}^{\ell }\right\vert $ (with implied constants
depending only on $N$), it is now an easy matter to show that%
\begin{eqnarray*}
\mathrm{P}\left( I,\widehat{\omega }\right) \mathrm{P}\left( I,\widehat{\dot{%
\sigma}}\right) &=&\left( \int_{\mathbb{R}}\frac{\left\vert I\right\vert }{%
\left( \left\vert I\right\vert +\left\vert x-c_{I}\right\vert \right) ^{2}}d%
\widehat{\omega }\left( x\right) \right) \left( \int_{\mathbb{R}}\frac{%
\left\vert I\right\vert }{\left( \left\vert I\right\vert +\left\vert
y-c_{I}\right\vert \right) ^{2}}d\widehat{\dot{\sigma}}\left( y\right)
\right) \\
&\lesssim &\left( \int_{\mathbb{R}}\frac{\left\vert I_{r}^{\ell }\right\vert 
}{\left( \left\vert I_{r}^{\ell }\right\vert +\left\vert x-c_{I_{r}^{\ell
}}\right\vert \right) ^{2}}d\widehat{\omega }\left( x\right) \right) \left(
\int_{\mathbb{R}}\frac{\left\vert I_{r}^{\ell }\right\vert }{\left(
\left\vert I_{r}^{\ell }\right\vert +\left\vert y-c_{I_{r}^{\ell
}}\right\vert \right) ^{2}}d\widehat{\dot{\sigma}}\left( y\right) \right) \\
&=&\mathrm{P}\left( I_{r}^{\ell },\widehat{\omega }\right) \mathrm{P}\left(
I_{r}^{\ell },\widehat{\dot{\sigma}}\right) \approx 1,
\end{eqnarray*}%
since%
\begin{equation*}
\frac{\left\vert I\right\vert }{\left( \left\vert I\right\vert +\left\vert
x-c_{I}\right\vert \right) ^{2}}\lesssim \frac{\left\vert I_{r}^{\ell
}\right\vert }{\left( \left\vert I_{r}^{\ell }\right\vert +\left\vert
x-c_{I_{r}^{\ell }}\right\vert \right) ^{2}},\ \ \ \ \ x\in \mathbb{R}.
\end{equation*}

We record the following facts proved implicitly above for future use. Note
that (\ref{heads and tails}) is used in the first line below.%
\begin{eqnarray}
\mathrm{P}\left( I_{r}^{\ell },\widehat{\omega }\right) &\approx &\frac{%
\left\vert I_{r}^{\ell }\right\vert _{\widehat{\omega }}}{\left\vert
I_{r}^{\ell }\right\vert }=\frac{1}{2}\frac{\left( \frac{1+\eta }{2}\right)
^{H\left( I_{r}^{\ell }\right) }\left( \frac{1-\eta }{2}\right) ^{T\left(
I_{r}^{\ell }\right) }}{\left\vert I_{r}^{\ell }\right\vert }=\kappa
_{r}^{\ell }\left( \frac{N}{2}\right) ^{\ell },  \label{facts for future} \\
\mathrm{P}\left( I_{r}^{\ell },\widehat{\dot{\sigma}}\right) &\approx &\frac{%
\left\vert I_{r}^{\ell }\right\vert _{\widehat{\dot{\sigma}}}}{\left\vert
I_{r}^{\ell }\right\vert }\approx \frac{1}{\kappa _{r}^{\ell }}\left( \frac{2%
}{N}\right) ^{\ell },  \notag
\end{eqnarray}%
with%
\begin{equation}
\kappa _{r}^{\ell }\equiv \left( 1+\frac{1}{N}\right) ^{H\left( I_{r}^{\ell
}\right) }\left( 1-\frac{1}{N}\right) ^{T\left( I_{r}^{\ell }\right) },
\label{def kappa}
\end{equation}%
and where $I=\left[ 0,1\right] _{\mathbf{\varepsilon }}$ with $\mathbf{%
\varepsilon }=\left( \varepsilon _{1},\varepsilon _{2},...,\varepsilon
_{d\left( I\right) }\right) \in \left\{ +,-\right\} ^{d\left( I\right) }$
and $H\left( I\right) $ (for heads) is the number of $+$ signs in $\mathbf{%
\varepsilon }^{\prime }\equiv \left( \varepsilon _{2},...,\varepsilon
_{d\left( I\right) }\right) $ and $T\left( I\right) $ (for tails) is the
number of $-$ signs in $\mathbf{\varepsilon }^{\prime }$. See immediately
below for a further discussion of the $\left[ 0,1\right] _{\mathbf{%
\varepsilon }}$-notation.

\subsection{The forward testing condition}

We will see that the forward testing condition for $H_{\flat }$ holds with
respect to the weight pair $\left( \widehat{\dot{\sigma}},\widehat{\omega }%
\right) $ because the measures $\widehat{\omega }$ and $\widehat{\dot{\sigma}%
}$ are self-similar, and the operator $H_{\flat }$ is invariant under
dilations of scale $N$. The argument here is considerably more delicate than
the corresponding argument in \cite{LaSaUr2}, involving some surprising
symmetries buried in the redistribution of the measure $\widehat{\omega }$.
In order to obtain a replicating formula for the measures $\widehat{\omega }$
and $\widehat{\dot{\sigma}}$, it is convenient to introduce the notion of a 
\emph{noncommutative addition} on an embedded dyadic tree.

\subsubsection{A noncommutative addition on the dyadic tree}

Let $\left( \mathcal{T},\succcurlyeq ,\mathfrak{r},\pi ,\mathfrak{C}\right) $
denote the dyadic tree with relation $\succcurlyeq $, root $\mathfrak{r}$,
parent $\pi :\mathcal{T\setminus }\left\{ \mathfrak{r}\right\} \rightarrow 
\mathcal{T}$ and children $\mathfrak{C}:\mathcal{T}\rightarrow \mathcal{T}%
\times \mathcal{T}$ as described earlier. Consider the following three
examples of trees defined in terms of our construction of intervals $%
I_{r}^{\ell }$\ above, and with partial order defined in terms of set
inclusion: 
\begin{eqnarray*}
\mathcal{D} &=&\left\{ I_{r}^{\ell }:\left( \ell ,r\right) \in \mathbb{Z}%
_{+}\times \mathbb{N},1\leq r\leq 2^{\ell }\right\} \sim \left\{ \left( \ell
,r\right) \in \mathbb{Z}_{+}\times \mathbb{N}:1\leq r\leq 2^{\ell }\right\} ,
\\
\mathcal{D}_{-} &=&\left\{ I_{r}^{\ell }:\left( \ell ,r\right) \in \mathbb{N}%
\times \mathbb{N},1\leq r\leq 2^{\ell -1}\right\} \sim \left\{ \left( \ell
,r\right) \in \mathbb{N}\times \mathbb{N}:1\leq r\leq 2^{\ell -1}\right\} ,
\\
\mathcal{D}_{+} &=&\left\{ I_{r}^{\ell }:\left( \ell ,r\right) \in \mathbb{N}%
\times \mathbb{N},2^{\ell -1}<r\leq 2^{\ell }\right\} \sim \left\{ \left(
\ell ,r\right) \in \mathbb{N}\times \mathbb{N}:2^{\ell -1}<r\leq 2^{\ell
}\right\} .
\end{eqnarray*}

In order to model dilations and translations of our measures on these dyadic
trees, we define a noncommutative addition $\oplus :\mathcal{T}\times 
\mathcal{T}\rightarrow \mathcal{T}$ in an embedded tree $\mathcal{T}$ as
follows. Fix $\alpha ,\beta \in \mathcal{T}$. Then if $\mathbf{\varepsilon }%
\in \left\{ +,-\right\} ^{d\left( \beta \right) }$ is the unique string of $%
\pm $ such that $\beta =\left( \mathfrak{r}\right) _{\varepsilon }$, we
define $\alpha \oplus \beta =\alpha _{\varepsilon }$. We can describe this
addition more informally when we view $\mathcal{T}$ as embedded in the
plane. Indeed, let $\mathcal{S}$ be a copy of $\mathcal{T}$ in the plane and
translate (in the plane) the root $\mathfrak{r}_{\mathcal{S}}$ of $\mathcal{S%
}$ to lie on top of the point $\alpha $ in the tree $\mathcal{T}$, and let
the remaining points of $\mathcal{S}$ fall on\ the corresponding points of $%
\mathcal{T}$ lying below $\alpha $ (for this we `dilate' the translate of
the copy $\mathcal{S}$ so as to fit overtop $\mathcal{T}$). Then the point $%
\alpha \oplus \beta $\ is the point in the tree $\mathcal{T}$ that lies
underneath the point $\beta $ in the tree $\mathcal{S}$. Finally we define
the translation $\beta \oplus \mathcal{T}$ of the tree $\mathcal{T}$ by a
point $\beta \in \mathcal{T}$ on the left to be the tree 
\begin{equation*}
\beta \oplus \mathcal{T}\equiv \left\{ \beta \oplus \alpha :\alpha \in 
\mathcal{T}\right\}
\end{equation*}%
equipped with the inherited structure from $\left( \mathcal{T},\succcurlyeq ,%
\mathfrak{r},\pi ,\mathfrak{C}\right) $. For example, the root of the tree $%
\beta \oplus \mathcal{T}$ is $\mathfrak{r}\left( \beta \oplus \mathcal{T}%
\right) =\beta $ and the $\beta \oplus \mathcal{T}$-children of $\beta
\oplus \alpha $ are $\beta \oplus \alpha _{\pm }$.

\subsubsection{Self-similarity}

We first focus on the left hand child $I_{1}^{1}$ of $I_{1}^{0}=\left[ 0,1%
\right] $ and recall that $\left\vert I_{1}^{1}\right\vert _{\widehat{\omega 
}}=\left\vert I_{2}^{1}\right\vert _{\widehat{\omega }}$. To each interval $%
I_{r}^{\ell }$ we attach the increment $\eta =\eta \left( I_{r}^{\ell
}\right) =\frac{1}{N}$ in the formula (\ref{redistribution}) for the
distribution of $\widehat{\omega }$ to its two children $I_{r,\limfunc{left}%
}^{\ell }=\left( I_{r}^{\ell }\right) _{-}$ and $I_{r,\limfunc{right}}^{\ell
}=\left( I_{r}^{\ell }\right) _{+}$, i.e. by multiplying with the factors $%
\frac{1\pm \eta \left( I_{r}^{\ell }\right) }{2}$ on the left and right
appropriately. Recall from (\ref{precisely}) that for any interval $I\in 
\mathcal{D}\setminus \left\{ \left[ 0,1\right] \right\} $,%
\begin{eqnarray*}
\left\vert I_{--}\right\vert _{\widehat{\omega }} &=&\frac{1+\eta }{2}%
\left\vert I_{-}\right\vert _{\widehat{\omega }}\text{ and }\left\vert
I_{-+}\right\vert _{\widehat{\omega }}=\frac{1-\eta }{2}\left\vert
I_{-}\right\vert _{\widehat{\omega }}\ , \\
\left\vert I_{+-}\right\vert _{\widehat{\omega }} &=&\frac{1-\eta }{2}%
\left\vert I_{+}\right\vert _{\widehat{\omega }}\text{ and }\left\vert
I_{++}\right\vert _{\widehat{\omega }}=\frac{1+\eta }{2}\left\vert
I_{+}\right\vert _{\widehat{\omega }}\ .
\end{eqnarray*}

We now show that, as a consequence of $\eta =\eta \left( I_{r}^{\ell
}\right) =\frac{1}{N}$ in Lemma \ref{eta is 1/N}, there is a `homogeneous'
replicating formula for the self-similar measures $\widehat{\omega }_{-}=%
\mathbf{1}_{\left[ 0,\frac{1}{N}\right] }\widehat{\omega }$ and $\widehat{%
\omega }_{+}=\mathbf{1}_{\left[ 1-\frac{1}{N},1\right] }\widehat{\omega }$,
and hence also an `inhomogeneous' replicating formula for both $\widehat{%
\dot{\sigma}}_{-}=\mathbf{1}_{\left[ 0,\frac{1}{N}\right] }\widehat{\dot{%
\sigma}}$ and $\widehat{\dot{\sigma}}_{+}=\mathbf{1}_{\left[ 1-\frac{1}{N},1%
\right] }\widehat{\dot{\sigma}}$ with a point mass as the inhomogeneous
term. In order to state these replication formulas precisely we need the
measures%
\begin{equation*}
\widehat{\omega }_{--}\equiv \mathbf{1}_{\left[ 0,\frac{1}{N^{2}}\right] }%
\widehat{\omega }\text{ and }\widehat{\omega }_{-+}\equiv \mathbf{1}_{\left[ 
\frac{1}{N}-\frac{1}{N^{2}},\frac{1}{N}\right] }\widehat{\omega }\text{ and }%
\widehat{\omega }_{+-}\equiv \mathbf{1}_{\left[ 1-\frac{1}{N},1-\frac{1}{N}+%
\frac{1}{N^{2}}\right] }\widehat{\omega }\text{ and }\widehat{\omega }%
_{++}\equiv \mathbf{1}_{\left[ 1-\frac{1}{N^{2}},1\right] }\widehat{\omega },
\end{equation*}%
and we define $\limfunc{Ref}$ to be reflection about the origin in the real
line. For convenience in viewing formulas, we set $\eta =\frac{1}{N}$ in the
factors, but retain the notation $\frac{1}{N}$ in dilations and translations.

Let $\func{Dil}_{\gamma }x\equiv \gamma x$, and $\func{Dil}_{\gamma }\mu
\left( x\right) \equiv \frac{1}{\gamma }\mu \left( \func{Dil}_{\frac{1}{%
\gamma }}x\right) =\frac{1}{\gamma }\mu \left( \frac{x}{\gamma }\right) $ be
the measure with the same mass as $\mu $ but with support dilated by the
positive factor $\gamma $. Similarly let $\func{Trans}_{\gamma }x=x+\gamma $
and $\func{Trans}_{\gamma }\mu \left( x\right) \equiv \mu \left( \func{Trans}%
_{-\gamma }x\right) =\mu \left( x-\gamma \right) $. Finally, define $%
\limfunc{Ref}\left( x\right) =-x$ to be reflection about the origin $0$ in
the real line and $\limfunc{Ref}\mu \left( x\right) =\mu \left( \limfunc{Ref}%
x\right) =\mu \left( -x\right) $. The following diagram pictures the $%
\widehat{\omega }$ weight of closed intervals $I_{r}^{\ell }$ in italics,
and the $\widehat{\dot{\sigma}}$ measure of the open intervals $G_{r}^{\ell
} $ in bold: 
\begin{equation*}
\begin{Bmatrix}
\lbrack &  &  &  &  &  &  &  &  &  &  &  &  & \mathit{1} &  &  &  &  &  &  & 
&  &  &  &  &  & ] \\ 
&  &  &  &  &  &  &  &  &  &  &  &  &  &  &  &  &  &  &  &  &  &  &  &  &  & 
\\ 
\lbrack &  &  &  & \frac{\mathit{1}}{\mathit{2}} &  &  &  & ] & ( &  &  &  & 
\mathbf{1} &  &  &  & ) & [ &  &  &  & \frac{\mathit{1}}{\mathit{2}} &  &  & 
& ] \\ 
&  &  &  &  &  &  &  &  &  &  &  &  &  &  &  &  &  &  &  &  &  &  &  &  &  & 
\\ 
\lbrack & \frac{\mathit{1+\eta }}{\mathit{4}} & ] & ( & \frac{\mathbf{2}}{%
\mathbf{N}^{\mathbf{2}}} & ) & [ & \frac{\mathit{1-\eta }}{\mathit{4}} & ] & 
&  &  &  &  &  &  &  &  & [ & \frac{\mathit{1-\eta }}{\mathit{4}} & ] & ( & 
\frac{\mathbf{2}}{\mathbf{N}^{\mathbf{2}}} & ) & [ & \frac{\mathit{1+\eta }}{%
\mathit{4}} & ]%
\end{Bmatrix}%
.
\end{equation*}

\begin{lemma}
\label{replicate}We have the replicating formulas for $\widehat{\omega }_{-}$
and $\widehat{\dot{\sigma}}_{-}$;%
\begin{eqnarray}
\widehat{\omega }_{-} &=&\widehat{\omega }_{--}+\widehat{\omega }_{-+};
\label{rep} \\
\widehat{\omega }_{--} &=&\frac{1+\eta }{2}\func{Dil}_{\frac{1}{N}}\widehat{%
\omega }_{-}\text{ and }\widehat{\omega }_{-+}=\frac{1-\eta }{2}\func{Trans}%
_{\frac{1}{N}}\limfunc{Ref}\func{Dil}_{\frac{1}{N}}\widehat{\omega }_{-}\ , 
\notag \\
\widehat{\dot{\sigma}}_{-} &=&\widehat{\dot{\sigma}}_{--}+\frac{2}{N^{2}}%
\delta _{\frac{1}{2N}}+\widehat{\dot{\sigma}}_{-+};  \notag \\
\widehat{\dot{\sigma}}_{--} &=&\frac{1}{N^{2}}\frac{2}{1+\eta }\func{Dil}_{%
\frac{1}{N}}\widehat{\dot{\sigma}}_{-}\text{ and }\widehat{\dot{\sigma}}%
_{-+}=\frac{1}{N^{2}}\frac{2}{1-\eta }\func{Trans}_{\frac{1}{N}}\limfunc{Ref}%
\func{Dil}_{\frac{1}{N}}\widehat{\dot{\sigma}}_{-}\ ,  \notag
\end{eqnarray}%
and the replicating formulas for $\widehat{\omega }_{+}$ and $\widehat{\dot{%
\sigma}}_{+}$;%
\begin{eqnarray*}
\widehat{\omega }_{+} &=&\widehat{\omega }_{+-}+\widehat{\omega }_{++}; \\
\widehat{\omega }_{+-} &=&\frac{1-\eta }{2}\func{Trans}_{1}\limfunc{Ref}%
\func{Trans}_{\frac{1}{N}}\limfunc{Ref}\func{Dil}_{\frac{1}{N}}\limfunc{Ref}%
\func{Trans}_{-1}\widehat{\omega }_{+}, \\
\widehat{\omega }_{++} &=&\frac{1+\eta }{2}\func{Trans}_{1}\limfunc{Ref}%
\func{Dil}_{\frac{1}{N}}\limfunc{Ref}\func{Trans}_{-1}\widehat{\omega }_{+}\
, \\
\widehat{\dot{\sigma}}_{+} &=&\widehat{\dot{\sigma}}_{+-}+\frac{2}{N^{2}}%
\delta _{1-\frac{1}{2N}}+\widehat{\dot{\sigma}}_{++}; \\
\widehat{\dot{\sigma}}_{+-} &=&\frac{1}{N^{2}}\frac{2}{1-\eta }\func{Trans}%
_{1}\limfunc{Ref}\func{Trans}_{\frac{1}{N}}\limfunc{Ref}\func{Dil}_{\frac{1}{%
N}}\limfunc{Ref}\func{Trans}_{-1}\widehat{\dot{\sigma}}_{+}, \\
\widehat{\dot{\sigma}}_{++} &=&\frac{1}{N^{2}}\frac{2}{1+\eta }\func{Trans}%
_{1}\limfunc{Ref}\func{Dil}_{\frac{1}{N}}\limfunc{Ref}\func{Trans}_{-1}%
\widehat{\dot{\sigma}}_{+}\ ,
\end{eqnarray*}%
as well as the dilation invariance of $H_{\flat }$:%
\begin{equation*}
K_{\flat }\left( x\right) =\func{Dil}_{\frac{1}{N}}K_{\flat }\left( \func{Dil%
}_{\frac{1}{N}}x\right) .
\end{equation*}
\end{lemma}

\begin{proof}
From (\ref{heads and tails}), the masses $\lambda \left( I\right) $ of the
measure $\widehat{\omega }$ on the intervals $I$ in the dyadic tree $%
\mathcal{D}$ with root $\mathfrak{r}=\mathfrak{r}\left( \mathcal{D}\right) =%
\left[ 0,1\right] $\ are given by%
\begin{equation}
\lambda \left( I\right) \equiv \left\vert I\right\vert _{\widehat{\omega _{-}%
}}=\frac{1}{2}\left( \frac{1+\eta }{2}\right) ^{H\left( I\right) }\left( 
\frac{1-\eta }{2}\right) ^{T\left( I\right) }  \label{def lambda}
\end{equation}%
where $I=\left[ 0,1\right] _{\mathbf{\varepsilon }}$ with $\mathbf{%
\varepsilon }=\left( \varepsilon _{1},\varepsilon _{2},...,\varepsilon
_{d\left( I\right) }\right) \in \left\{ +,-\right\} ^{d\left( I\right) }$
and $H\left( I\right) $ (for heads) is the number of $+$ signs in $\mathbf{%
\varepsilon }^{\prime }\equiv \left( \varepsilon _{2},...,\varepsilon
_{d\left( I\right) }\right) $ and $T\left( I\right) $ (for tails) is the
number of $-$ signs in $\mathbf{\varepsilon }^{\prime }$. We can picture the
restricted sequence $\lambda :\mathcal{D}_{-}\rightarrow \left( 0,\infty
\right) $ on the left half of the tree $\mathcal{D}$ in the usual way by
displaying the values schematically in an array as follows:%
\begin{eqnarray}
&&  \label{first} \\
&&  \notag \\
&&\left[ 
\begin{array}{ccccccc}
&  &  & \lambda \left( \mathfrak{r}_{-}\right) &  &  &  \\ 
&  &  &  &  &  &  \\ 
& \lambda \left( \mathfrak{r}_{--}\right) &  &  &  & \lambda \left( 
\mathfrak{r}_{-+}\right) &  \\ 
&  &  &  &  &  &  \\ 
\lambda \left( \mathfrak{r}_{---}\right) &  & \lambda \left( \mathfrak{r}%
_{--+}\right) &  & \lambda \left( \mathfrak{r}_{-+-}\right) &  & \lambda
\left( \mathfrak{r}_{-++}\right) \\ 
\vdots &  & \vdots &  & \vdots &  & \vdots%
\end{array}%
\right]  \notag
\end{eqnarray}%
which can be written in shorthand as%
\begin{eqnarray}
&&  \label{second} \\
&&  \notag \\
&&\left[ 
\begin{array}{ccccccc}
&  &  & \left[ 0\right] &  &  &  \\ 
&  &  &  &  &  &  \\ 
& \left[ +\right] &  &  &  & \left[ -\right] &  \\ 
&  &  &  &  &  &  \\ 
\left[ ++\right] &  & \left[ +-\right] &  & -- &  & \left[ -+\right] \\ 
\vdots &  & \vdots &  & \vdots &  & \vdots%
\end{array}%
\right]  \notag
\end{eqnarray}%
where $\left[ \varepsilon _{1}\varepsilon _{2}...\varepsilon _{k}\right]
=\left( \frac{1+\varepsilon _{2}\eta }{2}\right) \left( \frac{1+\varepsilon
_{3}\eta }{2}\right) ...\left( \frac{1+\varepsilon _{k}\eta }{2}\right) $,
i.e. the notation $\left[ -+-\right] $ stands for the product $\left( \frac{%
1-\eta }{2}\right) \left( \frac{1+\eta }{2}\right) \left( \frac{1-\eta }{2}%
\right) $ in (\ref{def lambda}), etc. Note that the $+^{\prime }s$ and $%
-^{\prime }s$ in square brackets in the second array above refer to the 
\emph{signs} $\varepsilon _{2},...\varepsilon _{k}$ in front of $\eta =\frac{%
1}{N}$, while the $+^{\prime }s$ and $-^{\prime }s$ ocurring as subscripts
of the root $\mathfrak{r}$ in the first array refer to the \emph{locations}
left and right. As a consequence the $+^{\prime }s$ and $-^{\prime }s$ in
the two arrays are quite different.

Recall that $\mathcal{D}_{-}=\left[ 0,\frac{1}{N}\right] \oplus \mathcal{D}$
and define in analogy $\mathcal{D}_{--}=\left[ 0,\frac{1}{N^{2}}\right]
\oplus \mathcal{D}$, etc. The pattern of $\pm $ signs is determined by the
formulas in (\ref{precisely}), and from this we see that the sequences $%
\left\{ \lambda \left( I\right) \right\} _{I\in \mathcal{D}_{-}}$ and $%
\left\{ \lambda \left( I\right) \right\} _{I\in \mathcal{D}_{--}}$ satisfy%
\begin{eqnarray*}
\left\{ \lambda \left( I\right) \right\} _{I\in \mathcal{D}_{--}} &=&\fbox{$%
\begin{array}{ccccccc}
&  &  & \left[ +\right] &  &  &  \\ 
&  & \swarrow &  & \searrow &  &  \\ 
& \left[ ++\right] &  &  &  & \left[ +-\right] &  \\ 
\left[ +++\right] &  & \left[ ++-\right] &  & \left[ +--\right] &  & \left[
+-+\right] \\ 
\vdots &  & \vdots &  & \vdots &  & \vdots%
\end{array}%
$} \\
&=&\left( \frac{1+\eta }{2}\right) \fbox{$%
\begin{array}{ccccccccccccccccc}
&  &  &  &  &  &  &  & \left[ 0\right] &  &  &  &  &  &  &  &  \\ 
&  &  &  &  &  & \swarrow &  &  &  & \searrow &  &  &  &  &  &  \\ 
&  &  &  & \swarrow &  &  &  &  &  &  &  & \searrow &  &  &  &  \\ 
&  & \left[ +\right] &  &  &  &  &  &  &  &  &  &  &  & \left[ -\right] &  & 
\\ 
& \swarrow &  & \searrow &  &  &  &  &  &  &  &  &  & \swarrow &  & \searrow
&  \\ 
\left[ ++\right] &  &  &  & \left[ +-\right] &  &  &  &  &  &  &  & \left[ --%
\right] &  &  &  & \left[ -+\right] \\ 
\vdots &  &  &  & \vdots &  &  &  &  &  &  &  & \vdots &  &  &  & \vdots%
\end{array}%
$} \\
&=&\left( \frac{1+\eta }{2}\right) \left\{ \lambda \left( I\right) \right\}
_{I\in \mathcal{D}_{-}}\ .
\end{eqnarray*}%
In fact, we need only check the top two lines in the tree because (\ref%
{precisely}) shows that the pattern of signs along \emph{any} row of
grandchildren is always $+,-,-,+$, as given in (\ref{mnemonic}). The
identity $\left\{ \lambda \left( I\right) \right\} _{I\in \mathcal{D}%
_{--}}=\left( \frac{1+\eta }{2}\right) \left\{ \lambda \left( I\right)
\right\} _{I\in \mathcal{D}_{-}}$ just established between the restrictions
of the sequence $\lambda $ to the trees $\mathcal{D}_{--}$ and $\mathcal{D}%
_{-}$ translates precisely into the first formula in the second line of (\ref%
{rep}): $\widehat{\omega }_{--}=\frac{1+\eta }{2}\func{Dil}_{\frac{1}{N}}%
\widehat{\omega }_{-}$.

In order to obtain a similar result for the tree $\mathcal{D}_{-+}$, we
define the operation of reflection $\limfunc{Reftree}$ on a sequence $%
\lambda =\left\{ \lambda \left( \alpha \right) \right\} _{\alpha \in 
\mathcal{T}}$ by $\limfunc{Reftree}\lambda =\left\{ \lambda \left( 
\widetilde{\alpha }\right) \right\} _{\alpha \in \mathcal{T}}$ where if $%
\alpha =\mathfrak{r}_{\mathbf{\varepsilon }}$ then $\widetilde{\alpha }=%
\mathfrak{r}_{-\mathbf{\varepsilon }}$ where $-\mathbf{\varepsilon }$ is the
sequence of location signs (as in (\ref{first})) with $+$ and $-$
interchanged - in other words $\limfunc{Reftree}\lambda $ is the mirror
image of the sequence $\lambda $ when presented as an array as in (\ref%
{second}). Then using that $\mathcal{D}_{-}$ and $\mathcal{D}_{-+}$ can both
be identified with $\mathcal{T}$ as embedded trees, we have that the
sequences $\left\{ \lambda \left( I\right) \right\} _{I\in \mathcal{D}_{-}}$
and $\left\{ \lambda \left( I\right) \right\} _{I\in \mathcal{D}_{-+}}$
satisfy%
\begin{eqnarray}
&&  \label{lambda} \\
&&  \notag \\
\left\{ \lambda \left( I\right) \right\} _{I\in \mathcal{D}_{-+}} &=&\fbox{$%
\begin{array}{ccccccc}
&  &  & \left[ -\right] &  &  &  \\ 
&  & \swarrow &  & \searrow &  &  \\ 
& \left[ --\right] &  &  &  & \left[ -+\right] &  \\ 
\left[ --+\right] &  & \left[ ---\right] &  & \left[ -+-\right] &  & \left[
-++\right] \\ 
\vdots &  & \vdots &  & \vdots &  & \vdots%
\end{array}%
$}  \notag \\
&=&\left( \frac{1-\eta }{2}\right) \fbox{$%
\begin{array}{ccccccccccccccccc}
&  &  &  &  &  &  &  & \left[ 0\right] &  &  &  &  &  &  &  &  \\ 
&  &  &  &  &  & \swarrow &  &  &  & \searrow &  &  &  &  &  &  \\ 
&  &  &  & \swarrow &  &  &  &  &  &  &  & \searrow &  &  &  &  \\ 
&  & \left[ -\right] &  &  &  &  &  &  &  &  &  &  &  & \left[ +\right] &  & 
\\ 
& \swarrow &  & \searrow &  &  &  &  &  &  &  &  &  & \swarrow &  & \searrow
&  \\ 
\left[ -+\right] &  &  &  & \left[ --\right] &  &  &  &  &  &  &  & \left[ +-%
\right] &  &  &  & \left[ ++\right] \\ 
\vdots &  &  &  & \vdots &  &  &  &  &  &  &  & \vdots &  &  &  & \vdots%
\end{array}%
$}  \notag
\end{eqnarray}%
which equals%
\begin{eqnarray}
&&\left( \frac{1-\eta }{2}\right) \limfunc{Reftree}\fbox{$%
\begin{array}{ccccccccccccccccc}
&  &  &  &  &  &  &  & \left[ 0\right] &  &  &  &  &  &  &  &  \\ 
&  &  &  &  &  & \swarrow &  &  &  & \searrow &  &  &  &  &  &  \\ 
&  &  &  & \swarrow &  &  &  &  &  &  &  & \searrow &  &  &  &  \\ 
&  & \left[ +\right] &  &  &  &  &  &  &  &  &  &  &  & \left[ -\right] &  & 
\\ 
& \swarrow &  & \searrow &  &  &  &  &  &  &  &  &  & \swarrow &  & \searrow
&  \\ 
\left[ ++\right] &  &  &  & \left[ +-\right] &  &  &  &  &  &  &  & \left[ --%
\right] &  &  &  & \left[ -+\right] \\ 
\vdots &  &  &  & \vdots &  &  &  &  &  &  &  & \vdots &  &  &  & \vdots%
\end{array}%
$}  \label{lambda2} \\
&=&\left( \frac{1-\eta }{2}\right) \limfunc{Reftree}\left\{ \lambda \left(
I\right) \right\} _{I\in \mathcal{D}_{-}}\ .  \notag
\end{eqnarray}%
Now for $K\in \mathcal{D}$, and a measure $\mu $ supported on the Cantor set 
$E^{\left( N\right) }$ and having $K$ equal to the minimal interval
containing the support of $\mu $, we associate to $\mu $ the sequence $%
\left\{ \mu \left( I\right) \right\} _{I\in \mathcal{D}_{K}}$ of its
measures on elements of the tree $\mathcal{D}_{K}\equiv \left\{ I\in 
\mathcal{D}:I\subset K\right\} $; and similarly if $\mu $ is supported on $%
E^{\left( N\right) }-1$, the translate of $E^{\left( N\right) }$ by one unit
to the left. There is of course a one-to-one correspondence between such
measures $\mu $ and their associated sequences on the tree $\mathcal{D}_{K}$%
. Then using the identity 
\begin{equation*}
\left\{ \limfunc{Ref}\mu \left( I\right) \right\} _{I\in \mathcal{D}_{K}-1}=%
\limfunc{Reftree}\left\{ \mu \left( I\right) \right\} _{I\in \mathcal{D}%
_{K}}\ ,
\end{equation*}%
where $\mathcal{D}_{K}-1$ denotes the translation of the tree $\mathcal{D}%
_{K}$ by one unit to the left, we see that (\ref{lambda}-\ref{lambda2})
translates precisely into the second formula in the second line of (\ref{rep}%
): $\widehat{\omega }_{-+}=\frac{1-\eta }{2}\func{Trans}_{\frac{1}{N}}%
\limfunc{Ref}\func{Dil}_{\frac{1}{N}}\widehat{\omega }_{-}$.

To obtain the analogous formulas for $\widehat{\omega }_{+}$ we note that%
\begin{eqnarray*}
\widehat{\omega }_{+} &=&\func{Trans}_{1}\limfunc{Ref}\widehat{\omega }_{-}\
, \\
\widehat{\omega }_{+-} &=&\func{Trans}_{1}\limfunc{Ref}\widehat{\omega }%
_{-+}\ , \\
\widehat{\omega }_{++} &=&\func{Trans}_{1}\limfunc{Ref}\widehat{\omega }%
_{--}\ .
\end{eqnarray*}%
Thus we have%
\begin{eqnarray*}
\widehat{\omega }_{+-} &=&\func{Trans}_{1}\limfunc{Ref}\left\{ \widehat{%
\omega }_{-+}\right\} =\func{Trans}_{1}\limfunc{Ref}\left\{ \frac{1-\eta }{2}%
\func{Trans}_{\frac{1}{N}}\limfunc{Ref}\func{Dil}_{\frac{1}{N}}\left[ 
\widehat{\omega }_{-}\right] \right\} \\
&=&\func{Trans}_{1}\limfunc{Ref}\left\{ \frac{1-\eta }{2}\func{Trans}_{\frac{%
1}{N}}\limfunc{Ref}\func{Dil}_{\frac{1}{N}}\left[ \limfunc{Ref}\func{Trans}%
_{-1}\widehat{\omega }_{+}\right] \right\} \\
&=&\frac{1-\eta }{2}\func{Trans}_{1}\limfunc{Ref}\func{Trans}_{\frac{1}{N}}%
\limfunc{Ref}\func{Dil}_{\frac{1}{N}}\limfunc{Ref}\func{Trans}_{-1}\widehat{%
\omega }_{+}
\end{eqnarray*}%
and%
\begin{eqnarray*}
\widehat{\omega }_{++} &=&\func{Trans}_{1}\limfunc{Ref}\left\{ \widehat{%
\omega }_{--}\right\} =\func{Trans}_{1}\limfunc{Ref}\left\{ \frac{1+\eta }{2}%
\func{Dil}_{\frac{1}{N}}\left[ \widehat{\omega }_{-}\right] \right\} \\
&=&\func{Trans}_{1}\limfunc{Ref}\left\{ \frac{1+\eta }{2}\func{Dil}_{\frac{1%
}{N}}\left[ \limfunc{Ref}\func{Trans}_{-1}\widehat{\omega }_{+}\right]
\right\} \\
&=&\frac{1+\eta }{2}\func{Trans}_{1}\limfunc{Ref}\func{Dil}_{\frac{1}{N}}%
\limfunc{Ref}\func{Trans}_{-1}\widehat{\omega }_{+}\ .
\end{eqnarray*}%
The formulas for $\widehat{\dot{\sigma}}_{\pm }$ are proved in the same way
as for $\widehat{\omega }_{\pm }$ using Lemma \ref{sigma dot} and the
definition (\ref{def s hat}) of the weights $\widehat{s}_{r}^{\ell
}=\left\vert G_{r}^{\ell }\right\vert _{\widehat{\dot{\sigma}}}$ in $%
\widehat{\dot{\sigma}}$ in terms of $\widehat{\omega }$:%
\begin{equation*}
\left\vert I_{r}^{\ell }\right\vert _{\widehat{\dot{\sigma}}}=\frac{N^{2}-1}{%
N^{2}-5}\left\vert G_{r}^{\ell }\right\vert _{\widehat{\dot{\sigma}}}=\frac{%
N^{2}-1}{N^{2}-5}\frac{\left\vert I_{r}^{\ell }\right\vert ^{2}}{\left\vert
I_{r}^{\ell }\right\vert _{\widehat{\omega }}}.
\end{equation*}%
This completes the proof of Lemma \ref{replicate}.
\end{proof}

\subsubsection{Completion of the proof for intervals $I_{r}^{\ell }$\label%
{Subsub comp}}

Armed with Lemma \ref{replicate} we now prove the forward testing condition
for the interval $I_{0}^{1}=\left[ 0,1\right] $:%
\begin{eqnarray}
\int \left\vert H_{\flat }\widehat{\dot{\sigma}}\right\vert ^{2}\widehat{%
\omega } &=&\int \left\vert H_{\flat }\left( \widehat{\dot{\sigma}}%
_{-}+\delta _{\frac{1}{2}}+\widehat{\dot{\sigma}}_{+}\right) \right\vert ^{2}%
\widehat{\omega }_{-}+\int \left\vert H_{\flat }\left( \widehat{\dot{\sigma}}%
_{-}+\delta _{\frac{1}{2}}+\widehat{\dot{\sigma}}_{+}\right) \right\vert ^{2}%
\widehat{\omega }_{+}  \label{we compute} \\
&\leq &\left( 1+\varepsilon \right) \left\{ \int \left\vert H_{\flat }%
\widehat{\dot{\sigma}}_{-}\right\vert ^{2}\widehat{\omega }_{-}+\int
\left\vert H_{\flat }\widehat{\dot{\sigma}}_{+}\right\vert ^{2}\widehat{%
\omega }_{+}\right\} +\mathcal{R}_{\varepsilon },  \notag
\end{eqnarray}%
where the remainder term $\mathcal{R}_{\varepsilon }$ is easily seen to
satisfy 
\begin{equation*}
\mathcal{R}_{\varepsilon }\lesssim _{\varepsilon }\mathcal{A}_{2}\left( \int 
\dot{\sigma}\right) ,
\end{equation*}%
since the supports of $\delta _{\frac{1}{2}}+\widehat{\dot{\sigma}}_{+}$ and 
$\widehat{\omega }_{-}$ are well separated, as are those of $\widehat{\dot{%
\sigma}}_{-}+\delta _{\frac{1}{2}}$ and $\widehat{\omega }_{+}$. Indeed, we
first use $\left( a+b\right) ^{2}\leq \left( 1+\varepsilon \right)
a^{2}+\left( 1+\frac{1}{\varepsilon }\right) b^{2}$ to obtain%
\begin{eqnarray*}
&&\int \left\vert H_{\flat }\left( \widehat{\dot{\sigma}}_{-}+\delta _{\frac{%
1}{2}}+\widehat{\dot{\sigma}}_{+}\right) \right\vert ^{2}\widehat{\omega _{-}%
} \\
&\leq &\int \left( \left\vert H_{\flat }\left( \widehat{\dot{\sigma}}%
_{-}\right) \right\vert +\left\vert H_{\flat }\left( \delta _{\frac{1}{2}}+%
\widehat{\dot{\sigma}}_{+}\right) \right\vert \right) ^{2}\widehat{\omega }%
_{-} \\
&\leq &\int \left\{ \left( 1+\varepsilon \right) \left\vert H\left( \widehat{%
\dot{\sigma}}_{-}\right) \right\vert ^{2}+\left( 1+\frac{1}{\varepsilon }%
\right) \left\vert H\left( \delta _{\frac{1}{2}}+\widehat{\dot{\sigma}}%
_{+}\right) \right\vert ^{2}\right\} \widehat{\omega }_{-},
\end{eqnarray*}%
and then for example, since $K_{\flat }\left( y-x\right) =1$ for $x\in \left[
0,\frac{1}{N}\right] $ and $y\in \left[ \frac{N-1}{N},1\right] $, 
\begin{eqnarray*}
\int \left\vert H_{\flat }\left( \widehat{\dot{\sigma}}_{+}\right)
\right\vert ^{2}\widehat{\omega }_{-} &=&\int_{\left[ 0,\frac{1}{N}\right]
}\left\vert \int_{\left[ \frac{N-1}{N},1\right] }d\widehat{\dot{\sigma}}%
_{+}\left( y\right) \right\vert ^{2}d\widehat{\omega }_{-} \\
&=&\left\vert \left[ 1-\frac{1}{N},1\right] \right\vert _{\widehat{\dot{%
\sigma}}_{+}}^{2}{}\left\vert \left[ 0,\tfrac{1}{M}\right] \right\vert _{%
\widehat{\omega }_{-}} \\
&\leq &\frac{\left\vert \left[ 0,1\right] \right\vert _{\widehat{\dot{\sigma}%
}}\left\vert \left[ 0,1\right] \right\vert _{\widehat{\omega }}}{|[0,1]|^{2}}%
\int \widehat{\dot{\sigma}}_{+}\leq \mathcal{A}_{2}\int \widehat{\dot{\sigma}%
}_{+}\ .
\end{eqnarray*}

Now we continue to the grandchildren so as to exploit the replication
formulas. We write%
\begin{eqnarray*}
\int \left\vert H_{\flat }\widehat{\dot{\sigma}}_{-}\right\vert ^{2}\widehat{%
\omega }_{-} &=&\int \left\vert H_{\flat }\left( \widehat{\dot{\sigma}}_{--}+%
\frac{2}{N^{2}}\delta _{\frac{1}{2N}}+\widehat{\dot{\sigma}}_{-+}\right)
\right\vert ^{2}\widehat{\omega }_{--}+\int \left\vert H_{\flat }\left( 
\widehat{\dot{\sigma}}_{--}+\frac{2}{N^{2}}\delta _{\frac{1}{2N}}+\widehat{%
\dot{\sigma}}_{-+}\right) \right\vert ^{2}\widehat{\omega }_{-+} \\
&\leq &\left( 1+\varepsilon \right) \left\{ \int \left\vert H_{\flat }%
\widehat{\dot{\sigma}}_{--}\right\vert ^{2}\widehat{\omega }_{--}+\int
\left\vert H_{\flat }\widehat{\dot{\sigma}}_{-+}\right\vert ^{2}\widehat{%
\omega }_{-+}\right\} +\mathcal{R}_{\varepsilon },
\end{eqnarray*}%
where again the remainder term $\mathcal{R}_{\varepsilon }$ is easily seen
just as above to satisfy 
\begin{equation*}
\mathcal{R}_{\varepsilon }\lesssim _{\varepsilon }\mathcal{A}_{2}\left( \int 
\widehat{\dot{\sigma}}_{-}\right) ,
\end{equation*}%
since the supports of $\frac{2}{N^{2}}\delta _{\frac{1}{2N}}+\widehat{\dot{%
\sigma}}_{-+}$ and $\widehat{\omega }_{--}$ are well separated, as are those
of $\widehat{\dot{\sigma}}_{--}+\frac{2}{N^{2}}\delta _{\frac{1}{2N}}$ and $%
\widehat{\omega }_{-+}$. Similarly we have%
\begin{equation*}
\int \left\vert H_{\flat }\widehat{\dot{\sigma}}_{+}\right\vert ^{2}\widehat{%
\omega }_{+}\leq \left( 1+\varepsilon \right) \left\{ \int \left\vert
H_{\flat }\widehat{\dot{\sigma}}_{+-}\right\vert ^{2}\widehat{\omega }%
_{+-}+\int \left\vert H_{\flat }\widehat{\dot{\sigma}}_{++}\right\vert ^{2}%
\widehat{\omega }_{++}\right\} +C\mathcal{A}_{2}\left( \int \widehat{\dot{%
\sigma}}_{+}\right) .
\end{equation*}

But now we note that%
\begin{eqnarray*}
\int \left\vert H_{\flat }\widehat{\dot{\sigma}}_{--}\right\vert ^{2}%
\widehat{\omega }_{--} &=&\int \left\vert H_{\flat }\widehat{\dot{\sigma}}%
_{--}\left( x\right) \right\vert ^{2}\frac{1+\eta }{2}\func{Dil}_{\frac{1}{N}%
}\widehat{\omega }_{--}\left( x\right) =\frac{1+\eta }{2}\int \left\vert
H_{\flat }\widehat{\dot{\sigma}}_{--}\left( x\right) \right\vert ^{2}N%
\widehat{\omega }_{-}\left( Nx\right) \\
&=&\frac{1+\eta }{2}\int \left\vert H_{\flat }\widehat{\dot{\sigma}}%
_{--}\left( \frac{y}{N}\right) \right\vert ^{2}N\widehat{\omega }_{-}\left(
y\right) \frac{dx}{dy}\ \ \ \ \ \ \ \ \ \ \ \ \ \ \ \ \ \ \ \ (y=Nx) \\
&=&\frac{1+\eta }{2}\int \left\vert \int K_{\flat }\left( \frac{y}{N}%
-z\right) \widehat{\dot{\sigma}}_{--}\left( z\right) \right\vert ^{2}%
\widehat{\omega }_{-}\left( y\right) \\
&=&\frac{1+\eta }{2}\int \left\vert \int K_{\flat }\left( \frac{y}{N}%
-z\right) \frac{2}{1+\eta }\frac{1}{N^{2}}\func{Dil}_{\frac{1}{N}}\widehat{%
\dot{\sigma}}_{-}\left( z\right) \right\vert ^{2}\widehat{\omega }_{-}\left(
y\right) \\
&=&\frac{2}{1+\eta }\frac{1}{N^{4}}\int \left\vert \int K_{\flat }\left( 
\frac{y}{N}-z\right) N\widehat{\dot{\sigma}}_{-}\left( Nz\right) \right\vert
^{2}\widehat{\omega }_{-}\left( y\right) \\
&=&\frac{2}{1+\eta }\frac{1}{N^{4}}\int \left\vert \int K_{\flat }\left( 
\frac{y}{N}-\frac{u}{N}\right) N\widehat{\dot{\sigma}}_{-}\left( u\right) 
\frac{dz}{du}\right\vert ^{2}\widehat{\omega }_{-}\left( y\right) \ \ \ \ \
\ \ \ \ \ \ \ \ \ \ \ \ \ \ \ (u=Nz) \\
&=&\frac{2}{1+\eta }\frac{1}{N^{4}}\int \left\vert \int NK_{\flat }\left(
y-u\right) \widehat{\dot{\sigma}}_{-}\left( u\right) \right\vert ^{2}%
\widehat{\omega }_{-}\left( y\right) =\frac{2}{1+\eta }\frac{1}{N^{2}}\int
\left\vert H_{\flat }\widehat{\dot{\sigma}}_{-}\right\vert ^{2}\widehat{%
\omega }_{-}\ ,
\end{eqnarray*}%
and similarly 
\begin{equation*}
\left\vert H_{\flat }\widehat{\dot{\sigma}}_{-+}\right\vert ^{2}\widehat{%
\omega }_{-+}=\frac{2}{1-\eta }\frac{1}{N^{2}}\int \left\vert H_{\flat }%
\widehat{\dot{\sigma}}_{-}\right\vert ^{2}\widehat{\omega }_{-}\ .
\end{equation*}%
Moreover, we also have from trivial modifications of the same arguments that%
\begin{equation*}
\int \left\vert H_{\flat }\widehat{\dot{\sigma}}_{+-}\right\vert ^{2}%
\widehat{\omega }_{+-}=\frac{2}{1-\eta }\frac{1}{N^{2}}\int \left\vert
H_{\flat }\widehat{\dot{\sigma}}_{+}\right\vert ^{2}\widehat{\omega }_{+}%
\text{ and }\int \left\vert H_{\flat }\widehat{\dot{\sigma}}_{++}\right\vert
^{2}\widehat{\omega }_{++}=\frac{2}{1+\eta }\frac{1}{N^{2}}\int \left\vert
H_{\flat }\widehat{\dot{\sigma}}_{+}\right\vert ^{2}\widehat{\omega }_{+}\ .
\end{equation*}%
Altogether we have that%
\begin{eqnarray*}
&&\int \left\vert H_{\flat }\widehat{\dot{\sigma}}_{-}\right\vert ^{2}%
\widehat{\omega }_{-}+\int \left\vert H_{\flat }\widehat{\dot{\sigma}}%
_{+}\right\vert ^{2}\widehat{\omega }_{+} \\
&\leq &\left( 1+\varepsilon \right) \left\{ \int \left\vert H_{\flat }%
\widehat{\dot{\sigma}}_{--}\right\vert ^{2}\widehat{\omega }_{--}+\int
\left\vert H_{\flat }\widehat{\dot{\sigma}}_{-+}\right\vert ^{2}\widehat{%
\omega }_{-+}+\int \left\vert H_{\flat }\widehat{\dot{\sigma}}%
_{+-}\right\vert ^{2}\widehat{\omega }_{+-}+\int \left\vert H_{\flat }%
\widehat{\dot{\sigma}_{++}}\right\vert ^{2}\widehat{\omega }_{++}\right\} +C%
\mathcal{A}_{2}^{2}\left( \int \widehat{\dot{\sigma}}\right) \\
&=&\left( 1+\varepsilon \right) \left( \frac{4}{1-\eta }+\frac{4}{1+\eta }%
\right) \frac{1}{N^{2}}\left\{ \int \left\vert H_{\flat }\widehat{\dot{\sigma%
}}_{-}\right\vert ^{2}\widehat{\omega }_{-}+\int \left\vert H_{\flat }%
\widehat{\dot{\sigma}}_{+}\right\vert ^{2}\widehat{\omega }_{+}\right\} +C%
\mathcal{A}_{2}\left( \int \widehat{\dot{\sigma}}\right) ,
\end{eqnarray*}%
and provided both $\int \left\vert H_{\flat }\widehat{\dot{\sigma}}%
_{-}\right\vert ^{2}\widehat{\omega }_{-}$ and $\int \left\vert H_{\flat }%
\widehat{\dot{\sigma}}_{+}\right\vert ^{2}\widehat{\omega }_{+}$ are finite
we conclude that%
\begin{equation*}
\int \left\vert H_{\flat }\widehat{\dot{\sigma}}_{-}\right\vert ^{2}\widehat{%
\omega }_{-}+\int \left\vert H_{\flat }\widehat{\dot{\sigma}}_{+}\right\vert
^{2}\widehat{\omega }_{+}\leq \frac{1}{1-\left( \frac{2}{1+\eta }+\frac{2}{%
1-\eta }\right) \frac{1+\varepsilon }{N^{2}}}C\mathcal{A}_{2}^{2}\left( \int 
\dot{\sigma}\right)
\end{equation*}%
for $1-\left( \frac{4}{1+\eta }+\frac{4}{1-\eta }\right) \frac{1+\varepsilon 
}{N^{2}}>0$, e.g. if $0<\varepsilon <1$ and $N\geq 5$. Finally then we have%
\begin{equation*}
\int \left\vert H_{\flat }\widehat{\dot{\sigma}}\right\vert ^{2}\widehat{%
\omega }\leq \left( 1+\varepsilon \right) \left\{ \int \left\vert H_{\flat }%
\widehat{\dot{\sigma}}_{-}\right\vert ^{2}\widehat{\omega }_{-}+\int
\left\vert H_{\flat }\widehat{\dot{\sigma}}_{+}\right\vert ^{2}\widehat{%
\omega }_{+}\right\} +C\mathcal{A}_{2}\left( \int \widehat{\dot{\sigma}}%
\right) \leq C\mathcal{A}_{2}\left( \int \widehat{\dot{\sigma}}\right) .
\end{equation*}

To avoid making the assumption that $\int \left\vert H\widehat{\dot{\sigma}}%
\right\vert ^{2}\widehat{\omega }$ is finite, we use the approximations%
\begin{eqnarray}
d\widehat{\omega }^{\left( m\right) }\left( x\right)
&=&\sum_{i=1}^{2^{m}}2^{-m}\kappa _{i}^{m}\frac{1}{\left\vert
I_{i}^{m}\right\vert }\mathbf{1}_{I_{i}^{m}}\left( x\right) dx,
\label{approximations} \\
\widehat{\dot{\sigma}}^{\left( n\right) } &=&\sum_{k<n}\sum_{j=1}^{2^{k}}%
\widehat{s}_{j}^{k}\delta _{\dot{z}_{j}^{k}},  \notag
\end{eqnarray}%
($\kappa _{i}^{m}$ is defined in (\ref{def kappa})) as in the argument for
the forward testing condition for the corresponding weight pair in \cite%
{LaSaUr2}. Note that $\widehat{\omega }^{\left( m\right) }=\omega _{m}$ is
the $m^{th}$ generation redistribution of $\omega $ as defined in Subsection %
\ref{SubSec redist}, and that $\widehat{\dot{\sigma}}^{\left( n\right) }$ is
a partial sum of the series defining $\widehat{\dot{\sigma}}$. Then for
fixed $n\leq m$ we obtain in analogy with \cite{LaSaUr2} that%
\begin{equation*}
\int \left\vert H_{\flat }\widehat{\dot{\sigma}}^{\left( n\right)
}\right\vert ^{2}\widehat{\omega }^{\left( m\right) }\leq C\mathcal{A}%
_{2}\left( \int \widehat{\dot{\sigma}}^{\left( n\right) }\right) +\left(
1-\delta \right) ^{m-n}\int \left\vert H_{\flat }\widehat{\dot{\sigma}}%
^{\left( 0\right) }\right\vert ^{2}\widehat{\omega }^{\left( m-n\right) },
\end{equation*}%
where $\int \left\vert H_{\flat }\widehat{\dot{\sigma}}^{\left( 0\right)
}\right\vert ^{2}\widehat{\omega }^{\left( m-n\right) }\leq C\mathcal{A}_{2}$
independent of $m$ and $n$. Now for fixed $n$ we let $m\rightarrow \infty $
which gives%
\begin{equation*}
\int \left\vert H_{\flat }\widehat{\dot{\sigma}}^{\left( n\right)
}\right\vert ^{2}\widehat{\omega }\leq C\mathcal{A}_{2}\left( \int \widehat{%
\dot{\sigma}}^{\left( n\right) }\right)
\end{equation*}%
since on the support of $\widehat{\omega }^{\left( n\right) }$ the function $%
H_{\flat }\widehat{\dot{\sigma}}^{\left( n\right) }$ is continuous. Then let 
$n\rightarrow \infty $ to obtain using Fatou's lemma,%
\begin{equation*}
\int \left\vert H_{\flat }\widehat{\dot{\sigma}}\right\vert ^{2}\widehat{%
\omega }=\int \liminf_{n\rightarrow \infty }\left\vert H_{\flat }\widehat{%
\dot{\sigma}}^{\left( n\right) }\right\vert ^{2}\widehat{\omega }\leq
\liminf_{n\rightarrow \infty }\int \left\vert H_{\flat }\widehat{\dot{\sigma}%
}^{\left( n\right) }\right\vert ^{2}\widehat{\omega }\leq C\mathcal{A}%
_{2}\liminf_{n\rightarrow \infty }\left( \int \widehat{\dot{\sigma}}^{\left(
n\right) }\right) =C\mathcal{A}_{2}\left( \int \widehat{\dot{\sigma}}\right)
,
\end{equation*}%
since for any $x$ in the Cantor set $E^{\left( N\right) }$, $%
\lim_{n\rightarrow \infty }H_{\flat }\widehat{\dot{\sigma}}^{\left( n\right)
}\left( x\right) =H_{\flat }\widehat{\dot{\sigma}}\left( x\right) $ because
the support of $\widehat{\dot{\sigma}}$ consists of points $\dot{z}%
_{r}^{\ell }$ whose distances from $x$ are at least $\frac{c}{N^{\ell }}$.

This completes the proof of the forward testing condition for the interval $%
I_{0}^{1}=\left[ 0,1\right] $. The arguments above are easily adapted to
prove the forward testing condition for any $I_{r}^{\ell }\in \mathcal{D}$
not equal to $I_{0}^{1}$. Indeed, the only essential difference, apart from
scaling factors related to $\ell $, is that, unlike the case $I=I_{0}^{1}$,
there is a redistribution of the $\widehat{\omega }$-measures of the
children of $I_{r}^{\ell }$. Consequently, in analogy with Lemma \ref%
{replicate}, there are replicating formulas for the children of $I_{r}^{\ell
}$ as opposed to the grandchildren as for the case $I_{0}^{1}=\left[ 0,1%
\right] $, resulting in a small simplification of the proof. For example, if
we fix $I_{1}^{\ell }=\left[ 0,\frac{1}{N^{\ell }}\right] $ to be the
leftmost interval at generation $\ell $, and if we denote the restricted
measures $\widehat{\omega }\mid _{I_{1}^{\ell }}$ and $\widehat{\dot{\sigma}}%
\mid _{I_{1}^{\ell }}$ by $\widehat{\omega }^{\ell }$ and $\widehat{\dot{%
\sigma}}^{\ell }$ respectively, then the replicating formulas for the
children of $\widehat{\omega }^{\ell }$ and $\widehat{\dot{\sigma}}^{\ell }$
are given by,%
\begin{eqnarray}
\widehat{\omega }^{\ell } &=&\widehat{\omega }_{-}^{\ell }+\widehat{\omega }%
_{+}^{\ell };  \label{new reps} \\
\widehat{\omega }_{-}^{\ell } &=&\frac{1+\eta }{2}\func{Dil}_{\frac{1}{N}}%
\widehat{\omega }^{\ell }\text{ and }\widehat{\omega }_{+}^{\ell }=\frac{%
1-\eta }{2}\func{Trans}_{\frac{1}{N^{\ell }}}\limfunc{Ref}\func{Dil}_{\frac{1%
}{N}}\widehat{\omega }^{\ell }\ ,  \notag \\
\widehat{\dot{\sigma}}^{\ell } &=&\widehat{\dot{\sigma}}_{-}^{\ell }+\frac{2%
}{N^{\ell +1}}\delta _{\frac{1}{2N^{\ell }}}+\widehat{\dot{\sigma}}%
_{+}^{\ell };  \notag \\
\widehat{\dot{\sigma}}_{-}^{\ell } &=&\frac{1}{N^{2}}\frac{2}{1+\eta }\func{%
Dil}_{\frac{1}{N}}\widehat{\dot{\sigma}}^{\ell }\text{ and }\widehat{\dot{%
\sigma}}_{+}^{\ell }=\frac{1}{N^{2}}\frac{2}{1-\eta }\func{Trans}_{\frac{1}{%
N^{\ell }}}\limfunc{Ref}\func{Dil}_{\frac{1}{N}}\widehat{\dot{\sigma}}^{\ell
}\ .  \notag
\end{eqnarray}%
With the formulas in ((\ref{new reps})) in hand, the forward testing
condition can now be obtained for $I_{1}^{\ell }$ by repeating verbatim the
above argument for $I_{0}^{1}$ using (\ref{new reps}). Finally, the case of
general $I_{r}^{\ell }$ is handled in exactly the same way but the
replicating formulas are more complicated to write out as the interval $%
I_{r}^{\ell }$ no longer has left endpoint at the origin.

\subsubsection{Completion of the proof for general intervals $I$\label%
{Subsub general I}}

It thus remains only to establish the forward testing condition for an
arbitrary interval $I$ contained in $\left[ 0,1\right] $. The analogous
estimate for the forward testing condition in \cite{LaSaUr2} is actually
more delicate than indicated there, where it was claimed that "the general
case now follows without much extra work", and consequently we will give a
detailed proof here.

So let $I=\left[ a,b\right] $ be an arbitrary \emph{closed} subinterval of $%
\left[ 0,1\right] $. Then there is a unique point $\dot{z}_{r}^{\ell }$ such
that $\widehat{s}_{r}^{\ell }=\sup_{\dot{z}_{j}^{k}\in I}\widehat{s}_{j}^{k}$%
. Then $\left\vert I\setminus I_{r}^{\ell }\right\vert _{\widehat{\dot{\sigma%
}}}=0$ follows from the choice of $\dot{z}_{r}^{\ell }$ and the fact that $I$
is closed. Thus from Lemma \ref{sigma dot} we now obtain 
\begin{equation*}
\left\vert I\right\vert _{\widehat{\dot{\sigma}}}\leq \left\vert I_{r}^{\ell
}\right\vert _{\widehat{\dot{\sigma}}}=\frac{N^{2}-1}{N^{2}-5}\left\vert
G_{r}^{\ell }\right\vert _{\widehat{\dot{\sigma}}}=\frac{N^{2}-1}{N^{2}-5}%
\left\vert \left\{ \dot{z}_{r}^{\ell }\right\} \right\vert _{\widehat{\dot{%
\sigma}}}\leq \frac{N^{2}-1}{N^{2}-5}\left\vert I\right\vert _{\widehat{\dot{%
\sigma}}}\ ,
\end{equation*}%
which shows that $\left\vert I\right\vert _{\widehat{\dot{\sigma}}}\approx
\left\vert \left\{ \dot{z}_{r}^{\ell }\right\} \right\vert _{\widehat{\dot{%
\sigma}}}=\widehat{s}_{r}^{\ell }$. We set 
\begin{equation*}
I_{\limfunc{left}}=\left[ a,\dot{z}_{r}^{\ell }\right) \text{ and }I_{%
\limfunc{right}}=\left( \dot{z}_{r}^{\ell },b\right] .
\end{equation*}%
Then we estimate%
\begin{eqnarray*}
\int_{I}\left\vert H_{\flat }\left( \mathbf{1}_{I}\widehat{\dot{\sigma}}%
\right) \right\vert ^{2}\widehat{\omega } &\leq &3\int_{I}\left\vert
H_{\flat }\left( \mathbf{1}_{I_{\limfunc{left}}}\widehat{\dot{\sigma}}%
\right) \right\vert ^{2}\widehat{\omega }+3\int_{I}\left\vert H_{\flat
}\left( \widehat{s}_{r}^{\ell }\delta _{\dot{z}_{r}^{\ell }}\right)
\right\vert ^{2}\widehat{\omega }+3\int_{I}\left\vert H_{\flat }\left( 
\mathbf{1}_{I_{\limfunc{right}}}\widehat{\dot{\sigma}}\right) \right\vert
^{2}\widehat{\omega } \\
&\leq &3\int_{I_{r}^{\ell }}\left\vert H_{\flat }\left( \mathbf{1}_{I_{%
\limfunc{left}}}\widehat{\dot{\sigma}}\right) \right\vert ^{2}\widehat{%
\omega }+3\int_{I_{r}^{\ell }}\left\vert H_{\flat }\left( \widehat{s}%
_{r}^{\ell }\delta _{\dot{z}_{r}^{\ell }}\right) \right\vert ^{2}\widehat{%
\omega }+3\int_{I_{r}^{\ell }}\left\vert H_{\flat }\left( \mathbf{1}_{I_{%
\limfunc{right}}}\widehat{\dot{\sigma}}\right) \right\vert ^{2}\widehat{%
\omega }
\end{eqnarray*}%
where the second line holds since $\left\vert I\setminus I_{r}^{\ell
}\right\vert _{\widehat{\omega }}=0$ follows from the choice of $\dot{z}%
_{r}^{\ell }$. Since the middle term is trivially estimated by $A_{2}%
\widehat{s}_{r}^{\ell }$, it suffices by symmetry to prove that%
\begin{equation}
\int_{I_{r}^{\ell }}\left\vert H_{\flat }\left( \mathbf{1}_{I_{\limfunc{left}%
}}\widehat{\dot{\sigma}}\right) \right\vert ^{2}\widehat{\omega }\lesssim 
\widehat{s}_{r}^{\ell }\ .  \label{to prove that}
\end{equation}%
Now let $\mathcal{M}_{I_{\limfunc{left}}}\equiv \left\{ J\in \mathcal{D}%
:J\subset I_{\limfunc{left}}\text{ is maximal}\right\} $ consist of the
maximal intervals from $\mathcal{D}$ that lie inside $I_{\limfunc{left}}$.
Then we can order the intervals in $\mathcal{M}_{I_{\limfunc{left}}}$ from
right to left as $\left\{ M_{i}\right\} _{i=1}^{\infty }$ where $%
M_{i}=I_{r_{i}}^{\ell _{i}}$ for a \emph{strictly} increasing sequence $%
\left\{ \ell _{i}\right\} _{i=1}^{\infty }$ with $\ell _{1}>\ell $ (of
course the sequence could be finite but then the proof is similar with
obvious modifications) and a choice of $r_{i}$ so that $M_{i+1}$ lies to the
left of $M_{i}$ for all $i\geq 1$. If $\left( I_{r}^{\ell }\right) _{-}$ is
contained in $I_{\limfunc{left}}$, then $\left( I_{r}^{\ell }\right) _{-}$
is the only maximal interval $M_{1}$ and the estimate (\ref{to prove that})
is then trivial. So we assume from now on that $\left( I_{r}^{\ell }\right)
_{-}\not\subset I_{\limfunc{left}}$. Then the sequence $\left\{
M_{i}\right\} _{i=1}^{\infty }$ is defined inductively as follows.

\begin{enumerate}
\item Define $J\equiv \left( I_{r}^{\ell }\right) _{-}$. If $J_{+}\subset I_{%
\limfunc{left}}$ then we set $M_{1}=J_{+}$. If $J_{+}\not\subset I_{\limfunc{%
left}}$ but $J_{++}\subset I_{\limfunc{left}}$, then we set $M_{1}=J_{++}$.
Otherwise, we continue in this way and let $M_{1}=J_{\overset{m_{1}}{%
\overbrace{++...+}}}$ where $J_{\overset{m_{1}}{\overbrace{++...+}}}$ with $%
m_{1}\geq 1$ is the first such interval that is contained in $I_{\limfunc{%
left}}$.

\item Define $K$ to be the sibling of $M_{1}=J_{\overset{m_{1}}{\overbrace{%
++...+}}}$ in the tree $\mathcal{D}$. Note that $K\not\subset I_{\limfunc{%
left}}$ by the choice of $M_{1}$. If $K_{+}\subset I_{\limfunc{left}}$ then
we set $M_{2}=K_{+}$. If $K_{+}\not\subset I_{\limfunc{left}}$ but $%
K_{++}\subset I_{\limfunc{left}}$, then we set $M_{2}=K_{++}$. Otherwise, we
continue in this way and let $M_{2}=K_{\overset{m_{2}}{\overbrace{++...+}}}$
where $K_{\overset{m_{2}}{\overbrace{++...+}}}$ with $m_{2}\geq 1$ is the
first such interval that is contained in $I_{\limfunc{left}}$.

\item Define $L$ to be the sibling of $K_{\overset{m_{2}}{\overbrace{++...+}}%
}$ in the tree $\mathcal{D}$. Note that $L\not\subset I_{\limfunc{left}}$ by
the choice of $M_{2}$. If $L_{+}\subset I_{\limfunc{left}}$ then we set $%
M_{3}=L_{+}$. If $L_{+}\not\subset I_{\limfunc{left}}$ but $L_{++}\subset I_{%
\limfunc{left}}$, then we set $M_{3}=L_{++}$. Otherwise, we continue in this
way and let $M_{3}=L_{\overset{m_{3}}{\overbrace{++...+}}}$ where $L_{%
\overset{m_{3}}{\overbrace{++...+}}}$ with $m_{3}\geq 1$ is the first such
interval that is contained in $I_{\limfunc{left}}$.

\item Continue to define $M_{4}$, $M_{5}$, $M_{6}$, ... in this way until
the procedure either terminates in a finite sequence $\left\{ M_{i}\right\}
_{i=1}^{N}$ or defines an infinite sequence $\left\{ M_{i}\right\}
_{i=1}^{\infty }$.
\end{enumerate}

Then we estimate%
\begin{eqnarray*}
\int_{I_{r}^{\ell }}\left\vert H_{\flat }\left( \mathbf{1}_{I_{\limfunc{left}%
}}\widehat{\dot{\sigma}}\right) \right\vert ^{2}\widehat{\omega }
&=&\int_{I_{r}^{\ell }}\left\vert H_{\flat }\left( \sum_{i=1}^{\infty }%
\mathbf{1}_{M_{i}}\widehat{\dot{\sigma}}\right) \right\vert ^{2}\widehat{%
\omega }=\sum_{i,j=1}^{\infty }\int_{I_{r}^{\ell }}H_{\flat }\left( \mathbf{1%
}M_{i}\widehat{\dot{\sigma}}\right) \overline{H_{\flat }\left( \mathbf{1}%
_{M_{j}}\widehat{\dot{\sigma}}\right) }\widehat{\omega } \\
&\leq &2\sum_{i=1}^{\infty }\sum_{j=i}^{\infty }\left\vert \int_{I_{r}^{\ell
}}H_{\flat }\left( \mathbf{1}_{M_{i}}\widehat{\dot{\sigma}}\right) \overline{%
H_{\flat }\left( \mathbf{1}_{M_{j}}\widehat{\dot{\sigma}}\right) }\widehat{%
\omega }\right\vert \\
&\leq &2\sum_{i=1}^{\infty }\sum_{j=i}^{\infty }\sqrt{\int_{I_{r}^{\ell
}}\left\vert H_{\flat }\left( \mathbf{1}_{M_{i}}\widehat{\dot{\sigma}}%
\right) \right\vert ^{2}\widehat{\omega }}\sqrt{\int_{I_{r}^{\ell
}}\left\vert H_{\flat }\left( \mathbf{1}M_{j}\widehat{\dot{\sigma}}\right)
\right\vert ^{2}\widehat{\omega }} \\
&\lesssim &\sum_{i=1}^{\infty }\sum_{j=i}^{\infty }\left\{ \frac{1}{\left(
j-i+1\right) ^{2}}\int_{I_{r}^{\ell }}\left\vert H_{\flat }\left( \mathbf{1}%
_{M_{i}}\widehat{\dot{\sigma}}\right) \right\vert ^{2}\widehat{\omega }%
+\left( j-i+1\right) ^{2}\int_{I_{r}^{\ell }}\left\vert H_{\flat }\left( 
\mathbf{1}_{M_{j}}\widehat{\dot{\sigma}}\right) \right\vert ^{2}\widehat{%
\omega }\right\} \\
&\lesssim &\sum_{i=1}^{\infty }\int_{I_{r}^{\ell }}\left\vert H_{\flat
}\left( \mathbf{1}_{M_{i}}\widehat{\dot{\sigma}}\right) \right\vert ^{2}%
\widehat{\omega }+\sum_{j=1}^{\infty }j^{3}\int_{I_{r}^{\ell }}\left\vert
H_{\flat }\left( \mathbf{1}_{M_{j}}\widehat{\dot{\sigma}}\right) \right\vert
^{2}\widehat{\omega } \\
&\lesssim &\sum_{k=1}^{\infty }k^{3}\int_{I_{r}^{\ell }}\left\vert H_{\flat
}\left( \mathbf{1}_{M_{k}}\widehat{\dot{\sigma}}\right) \right\vert ^{2}%
\widehat{\omega }=\sum_{k=1}^{\infty }k^{3}\int_{I_{r}^{\ell }}\left\vert
H_{\flat }\left( \mathbf{1}_{I_{r_{k}}^{\ell _{k}}}\widehat{\dot{\sigma}}%
\right) \right\vert ^{2}\widehat{\omega }.
\end{eqnarray*}

Now we note that if $I,J\in \mathcal{D}$ are disjoint, then $2I\cap
J=\emptyset $ as well. Thus from the testing condition for intervals $I\in 
\mathcal{D}$ that was proved above, together with the $\mathcal{A}_{2}$
condition, we easily obtain the \emph{full} testing condition for the
intervals $I\in \mathcal{D}$:%
\begin{eqnarray*}
\int_{0}^{1}\left\vert H_{\flat }\left( \mathbf{1}_{I}\widehat{\dot{\sigma}}%
\right) \right\vert ^{2}\widehat{\omega } &=&\int_{I}\left\vert H_{\flat
}\left( \mathbf{1}_{I}\widehat{\dot{\sigma}}\right) \right\vert ^{2}\widehat{%
\omega }+\int_{\left[ 0,1\right] \setminus 2I}\left\vert H_{\flat }\left( 
\mathbf{1}_{I}\widehat{\dot{\sigma}}\right) \right\vert ^{2}\widehat{\omega }
\\
&\leq &\mathfrak{T}_{H_{\flat }}\left\vert I\right\vert _{\widehat{\dot{%
\sigma}}}+C_{N}\mathcal{A}_{2}\left\vert I\right\vert _{\widehat{\dot{\sigma}%
}}\lesssim \left\vert I\right\vert _{\widehat{\dot{\sigma}}}\ .
\end{eqnarray*}%
Thus using $\int_{I_{r}^{\ell }}\left\vert H_{\flat }\left( \mathbf{1}%
_{I_{r_{k}}^{\ell _{k}}}\widehat{\dot{\sigma}}\right) \right\vert ^{2}%
\widehat{\omega }\lesssim \left\vert I_{r_{k}}^{\ell _{k}}\right\vert _{%
\widehat{\dot{\sigma}}}$ we obtain%
\begin{equation*}
\int_{I_{r}^{\ell }}\left\vert H_{\flat }\left( \mathbf{1}_{I_{\limfunc{left}%
}}\widehat{\dot{\sigma}}\right) \right\vert ^{2}\widehat{\omega }\lesssim
\sum_{k=1}^{\infty }k^{3}\int_{I_{r}^{\ell }}\left\vert H_{\flat }\left( 
\mathbf{1}_{I_{r_{k}}^{\ell _{k}}}\widehat{\dot{\sigma}}\right) \right\vert
^{2}\widehat{\omega }\lesssim \sum_{k=1}^{\infty }k^{3}\left\vert
I_{r_{k}}^{\ell _{k}}\right\vert _{\widehat{\dot{\sigma}}}\ ,
\end{equation*}%
and it remains to show that%
\begin{equation*}
\sum_{k=1}^{\infty }k^{3}\left\vert I_{r_{k}}^{\ell _{k}}\right\vert _{%
\widehat{\dot{\sigma}}}\lesssim \left\vert I_{r}^{\ell }\right\vert _{%
\widehat{\dot{\sigma}}}\ .
\end{equation*}

But this latter inequality is an easy consequence of the geometric decay of
the numbers $\left\vert I_{r_{k}}^{\ell _{k}}\right\vert _{\widehat{\dot{%
\sigma}}}$ as $k\rightarrow \infty $. Indeed, returning to the inductive
definition of $M_{k}=I_{r_{k}}^{\ell _{k}}$, we see from the inequality $%
\left\vert I_{\pm }\right\vert _{\widehat{\dot{\sigma}}}\leq \frac{2}{%
N^{2}\left( 1-\eta \right) }\left\vert I\right\vert _{\widehat{\dot{\sigma}}%
} $ that

\begin{enumerate}
\item $\left\vert M_{1}\right\vert _{\widehat{\dot{\sigma}}}=\left\vert J_{%
\overset{m_{1}}{\overbrace{++...+}}}\right\vert _{\widehat{\dot{\sigma}}%
}\leq \left( \frac{2}{N^{2}\left( 1-\eta \right) }\right) ^{m_{1}}\left\vert
J\right\vert _{\widehat{\dot{\sigma}}}$;

\item $\left\vert K\right\vert _{\widehat{\dot{\sigma}}}\leq \frac{1+\eta }{%
1-\eta }\left\vert M_{1}\right\vert _{\widehat{\dot{\sigma}}}$ and $%
\left\vert M_{2}\right\vert _{\widehat{\dot{\sigma}}}=\left\vert K_{\overset{%
m_{2}}{\overbrace{++...+}}}\right\vert _{\widehat{\dot{\sigma}}}\leq \left( 
\frac{2}{N^{2}\left( 1-\eta \right) }\right) ^{m_{2}}\left\vert K\right\vert
_{\widehat{\dot{\sigma}}}$;

\item $\left\vert L\right\vert _{\widehat{\dot{\sigma}}}\leq \frac{1+\eta }{%
1-\eta }\left\vert M_{2}\right\vert _{\widehat{\dot{\sigma}}}$ and $%
\left\vert M_{3}\right\vert _{\widehat{\dot{\sigma}}}=\left\vert L_{\overset{%
m_{3}}{\overbrace{++...+}}}\right\vert _{\widehat{\dot{\sigma}}}\leq \left( 
\frac{2}{N^{2}\left( 1-\eta \right) }\right) ^{m_{3}}\left\vert L\right\vert
_{\widehat{\dot{\sigma}}}$;

\item etc.
\end{enumerate}

Thus we see that%
\begin{eqnarray*}
\left\vert M_{1}\right\vert _{\widehat{\dot{\sigma}}} &\leq &\left( \frac{2}{%
N^{2}\left( 1-\eta \right) }\right) ^{m_{1}}\left\vert I_{r}^{\ell
}\right\vert _{\widehat{\dot{\sigma}}}\ , \\
\left\vert M_{2}\right\vert _{\widehat{\dot{\sigma}}} &\leq &\left( \frac{2}{%
N^{2}\left( 1-\eta \right) }\right) ^{m_{2}}\frac{1+\eta }{1-\eta }%
\left\vert M_{1}\right\vert _{\widehat{\dot{\sigma}}}\ , \\
\left\vert M_{3}\right\vert _{\widehat{\dot{\sigma}}} &\leq &\left( \frac{2}{%
N^{2}\left( 1-\eta \right) }\right) ^{m_{3}}\frac{1+\eta }{1-\eta }%
\left\vert M_{2}\right\vert _{\widehat{\dot{\sigma}}}\ , \\
&&etc.
\end{eqnarray*}%
and so%
\begin{eqnarray*}
\left\vert I_{r_{k}}^{\ell _{k}}\right\vert _{\widehat{\dot{\sigma}}}
&=&\left\vert M_{k}\right\vert _{\widehat{\dot{\sigma}}}\leq \left( \frac{2}{%
N^{2}\left( 1-\eta \right) }\right) ^{m_{1}+...+m_{k}}\left( \frac{1+\eta }{%
1-\eta }\right) ^{k-1}\left\vert I_{r}^{\ell }\right\vert _{\widehat{\dot{%
\sigma}}} \\
&\leq &\left( \frac{2}{N^{2}\left( 1-\eta \right) }\frac{1+\eta }{1-\eta }%
\right) ^{k}\left\vert I_{r}^{\ell }\right\vert _{\widehat{\dot{\sigma}}}\ .
\end{eqnarray*}%
Thus we have%
\begin{equation*}
\sum_{k=1}^{\infty }k^{3}\left\vert I_{r_{k}}^{\ell _{k}}\right\vert _{%
\widehat{\dot{\sigma}}}\leq \sum_{k=1}^{\infty }k^{3}\left( \frac{2}{%
N^{2}\left( 1-\eta \right) }\frac{1+\eta }{1-\eta }\right) ^{k}\left\vert
I_{r}^{\ell }\right\vert _{\widehat{\dot{\sigma}}}\leq C_{N}\left\vert
I_{r}^{\ell }\right\vert _{\widehat{\dot{\sigma}}}\ ,
\end{equation*}%
provided $\frac{2\left( 1+\frac{1}{N}\right) }{N^{2}\left( 1-\frac{1}{N}%
\right) ^{2}}<1$, i.e. $N\geq 3$, and this completes the proof of the
forward testing condition for the weight pair $\left( \widehat{\dot{\sigma}},%
\widehat{\omega }\right) $ relative to $H_{\flat }$.

\subsection{The backward testing condition}

The backward testing condition will hold for the weight pair $\left( 
\widehat{\dot{\sigma}},\widehat{\omega }\right) $ because we have arranged
that $H_{\flat }\widehat{\omega }\left( \dot{z}_{j}^{k}\right) =0$. Indeed,
for an interval $I_{r}^{\ell }$ with $\dot{z}_{j}^{k}\in I_{r}^{\ell }$, we
claim that 
\begin{equation}
\left\vert H_{\flat }\left( \mathbf{1}_{I_{r}^{\ell }}\widehat{\omega }%
\right) \left( \dot{z}_{j}^{k}\right) \right\vert \lesssim \mathrm{P}\left(
I_{r}^{\ell },\widehat{\omega }\right) .  \label{claim Poisson'}
\end{equation}%
To see (\ref{claim Poisson'}) let $I_{s}^{\ell -1}$ denote the parent of $%
I_{r}^{\ell }$ and let $I_{r+1}^{\ell }$ denote the other child of $%
I_{s}^{\ell -1}$. Then we have using $H_{\flat }\widehat{\omega }\left( \dot{%
z}_{j}^{k}\right) =0$, 
\begin{eqnarray*}
H_{\flat }\left( \mathbf{1}_{I_{r}^{\ell }}\widehat{\omega }\right) \left( 
\dot{z}_{j}^{k}\right) &=&-H_{\flat }\left( \mathbf{1}_{\left( I_{r}^{\ell
}\right) ^{c}}\widehat{\omega }\right) \left( \dot{z}_{j}^{k}\right) \\
&=&-H_{\flat }\left( \mathbf{1}_{\left( I_{s}^{\ell -1}\right) ^{c}}\widehat{%
\omega }\right) \left( \dot{z}_{j}^{k}\right) -H_{\flat }\left( \mathbf{1}%
_{I_{r+1}^{\ell }}\widehat{\omega }\right) \left( \dot{z}_{j}^{k}\right) .
\end{eqnarray*}%
Using $H_{\flat }\widehat{\omega }\left( {\dot{z}_{r}^{\ell }}\right) =0$ we
also have that%
\begin{eqnarray*}
H_{\flat }\left( \mathbf{1}_{\left( I_{s}^{\ell -1}\right) ^{c}}\widehat{%
\omega }\right) \left( \dot{z}_{j}^{k}\right) &=&H_{\flat }\left( \mathbf{1}%
_{\left( I_{s}^{\ell -1}\right) ^{c}}\widehat{\omega }\right) \left( {\dot{z}%
_{r}^{\ell }}\right) -\left\{ H_{\flat }\left( \mathbf{1}_{\left(
I_{s}^{\ell -1}\right) ^{c}}\widehat{\omega }\right) \left( {\dot{z}%
_{r}^{\ell }}\right) -H_{\flat }\left( \mathbf{1}_{\left( I_{s}^{\ell
-1}\right) ^{c}}\widehat{\omega }\right) \left( \dot{z}_{j}^{k}\right)
\right\} \\
&=&-H_{\flat }\left( \mathbf{1}_{I_{s}^{\ell -1}}\widehat{\omega }\right)
\left( {\dot{z}_{r}^{\ell }}\right) -A,
\end{eqnarray*}%
where 
\begin{equation*}
A\equiv H_{\flat }\left( \mathbf{1}_{\left( I_{s}^{\ell -1}\right) ^{c}}%
\widehat{\omega }\right) \left( {\dot{z}_{r}^{\ell }}\right) -H_{\flat
}\left( \mathbf{1}_{\left( I_{s}^{\ell -1}\right) ^{c}}\widehat{\omega }%
\right) \left( \dot{z}_{j}^{k}\right) .
\end{equation*}%
Combining equalities yields 
\begin{equation*}
H_{\flat }\left( \mathbf{1}_{I_{r}^{\ell }}\widehat{\omega }\right) \left( 
\dot{z}_{j}^{k}\right) =H_{\flat }\left( \mathbf{1}_{I_{s}^{\ell -1}}%
\widehat{\omega }\right) \left( \dot{z}_{r}^{\ell }\right) +A-H_{\flat
}\left( \mathbf{1}_{I_{r+1}^{\ell }}\widehat{\omega }\right) \left( \dot{z}%
_{j}^{k}\right) .
\end{equation*}%
We then have from (\ref{facts for future}) that for $\left( k,j\right) $
such that $\dot{z}_{j}^{k}\in I_{r}^{\ell }$,%
\begin{eqnarray*}
&&\left\vert H_{\flat }\left( \mathbf{1}_{I_{s}^{\ell -1}}\widehat{\omega }%
\right) \left( \dot{z}_{r}^{\ell }\right) \right\vert \lesssim \frac{%
\left\vert I_{s}^{\ell -1}\right\vert _{\widehat{\omega }}}{\left\vert {%
I_{s}^{\ell -1}}\right\vert }\approx \mathrm{P}\left( I_{s}^{\ell -1},%
\widehat{\omega }\right) , \\
&&\left\vert A\right\vert \lesssim \int_{\left( I_{s}^{\ell -1}\right)
^{c}}\left\vert \frac{1}{x-\dot{z}_{r}^{\ell }}-\frac{1}{x-\dot{z}_{j}^{k}}%
\right\vert d\widehat{\omega }\left( x\right) \lesssim \int_{\left(
I_{s}^{\ell -1}\right) ^{c}}\frac{\left\vert I_{r}^{\ell }\right\vert }{%
\left\vert x-\dot{z}_{s}^{\ell -1}\right\vert ^{2}}d\widehat{\omega }\left(
x\right) \lesssim \mathrm{P}\left( I_{r}^{\ell },\widehat{\omega }\right) ,
\\
&&\left\vert H_{\flat }\left( \mathbf{1}_{I_{r+1}^{\ell }}\widehat{\omega }%
\right) \left( \dot{z}_{j}^{k}\right) \right\vert \lesssim \frac{\left\vert
I_{s}^{\ell -1}\right\vert _{\widehat{\omega }}}{\left\vert {I_{s}^{\ell -1}}%
\right\vert }\approx \mathrm{P}\left( I_{s}^{\ell -1},\widehat{\omega }%
\right) ,
\end{eqnarray*}%
which proves (\ref{claim Poisson'}) since $\mathrm{P}\left( I_{s}^{\ell -1},%
\widehat{\omega }\right) \approx \mathrm{P}\left( I_{r}^{\ell },\widehat{%
\omega }\right) $.

Now, using (\ref{claim Poisson'}) and the estimate $\mathrm{P}\left(
I_{r}^{\ell },\widehat{\omega }\right) \approx \frac{\left\vert I_{r}^{\ell
}\right\vert _{\widehat{\omega }}}{\left\vert I_{r}^{\ell }\right\vert }$,
we compute that%
\begin{eqnarray}
\int_{I_{r}^{\ell }}\left\vert H_{\flat }\left( \mathbf{1}_{I_{r}^{\ell }}%
\widehat{\omega }\right) \right\vert ^{2}d\widehat{{\dot{\sigma}}}
&=&\sum_{\left( k,j\right) :\dot{z}_{j}^{k}\in I_{r}^{\ell }}\left\vert
H_{\flat }\left( \mathbf{1}_{I_{r}^{\ell }}\widehat{\omega }\right) \left( {%
\dot{z}}_{j}^{k}\right) \right\vert ^{2}\widehat{s}_{j}^{k}\leq
C\sum_{\left( k,j\right) :\dot{z}_{j}^{k}\in I_{r}^{\ell }}\left\vert 
\mathrm{P}\left( I_{r}^{\ell },\widehat{\omega }\right) \right\vert ^{2}%
\widehat{s}_{j}^{k}  \label{last'} \\
&\lesssim &\left\vert I_{r}^{\ell }\right\vert _{\widehat{{\dot{\sigma}}}%
}\left( \frac{\left\vert I_{r}^{\ell }\right\vert _{\widehat{\omega }}}{%
\left\vert I_{r}^{\ell }\right\vert }\right) ^{2}\lesssim \mathcal{A}%
_{2}\left\vert I_{r}^{\ell }\right\vert _{\widehat{\omega }}.  \notag
\end{eqnarray}%
This proves the case $I=I_{r}^{\ell }$ of the backward testing condition (%
\ref{e.H2}) for the weight pair $\left( \widehat{{\dot{\sigma}}},\widehat{%
\omega }\right) $. The general case of an arbitrary interval $I$ follows
from this case using the argument in Subsubsection \ref{Subsub general I},
where the analogous result for the \emph{forward} testing condition was
obtained - namely the special case $I=I_{r}^{\ell }$ of the \emph{forward}
testing condition was used to obtain the \emph{forward} testing condition
for general intervals $I$.

\subsection{The weak boundedness property}

Here we prove the weak boundedness property for the weight pair $\left( 
\widehat{{\dot{\sigma}}},\widehat{\omega }\right) $ relative to the
flattened Hibert transform $H_{\flat }$, i.e.%
\begin{equation*}
\left\vert \int_{J}H_{\flat }\left( \mathbf{1}_{I}\widehat{{\dot{\sigma}}}%
\right) d\widehat{\omega }\right\vert \leq \mathcal{WBP}_{H_{\flat }}\left( 
\widehat{{\dot{\sigma}}},\widehat{\omega }\right) \sqrt{\left\vert
J\right\vert _{\widehat{\omega }}\left\vert I\right\vert _{\widehat{{\dot{%
\sigma}}}}},
\end{equation*}%
with $\mathcal{WBP}_{H_{\flat }}\left( \widehat{{\dot{\sigma}}},\widehat{%
\omega }\right) <\infty $, for all $I,J\,$with $J\subset 3I$ and $I\subset
3J $. This is accomplished by showing that $H_{\flat }$ satisfies the
following inequality for any weight pair $\left( {\sigma },\omega \right) $:%
\begin{equation*}
\mathcal{WBP}_{H_{\flat }}\left( {\sigma },\omega \right) \leq \min \left\{ 
\mathfrak{T}_{H_{\flat }}\left( {\sigma },\omega \right) ,\mathfrak{T}%
_{H_{\flat }}^{\ast }\left( {\sigma },\omega \right) \right\} +C\mathcal{A}%
_{2}\left( {\sigma },\omega \right) .
\end{equation*}%
This last inequality is proved for the Hilbert transform $H$ in Proposition
2.9 of \cite{LaSaUr2}, and the same proof, which uses the two weight Hardy
inequalities of Muckenhoupt, applies here with little change upon noting
that $K_{\flat }\left( x\right) \approx \frac{1}{x}$ in the Hardy
inequalities.

\subsection{The energy conditions}

We show in the first two subsubsections that the backward energy condition
holds for the weight pair $\left( \widehat{{\dot{\sigma}}},\widehat{\omega }%
\right) $, but not for the weight pair $\left( \widehat{{\sigma }},\widehat{%
\omega }\right) $. Then we show in the third subsubsection that the forward
energy condition holds for the weight pair $\left( \widehat{{\dot{\sigma}}},%
\widehat{\omega }\right) $. These proofs for the energy conditions follow
the corresponding proofs for the energy conditions in \cite{LaSaUr2}, but
are complicated by the estimates for the redistributed measure $\widehat{%
\omega }$ and the reweighted measure $\widehat{{\dot{\sigma}}}$.

\subsubsection{The backward Energy Condition for $\left( \protect\widehat{{%
\dot{\protect\sigma}}},\protect\widehat{\protect\omega }\right) $}

First we show that the backward Energy Condition for the weight pair $\left( 
\widehat{{\dot{\sigma}}},\widehat{\omega }\right) $ holds: 
\begin{equation}
\sum_{r=1}^{\infty }\left\vert I_{r}\right\vert _{\widehat{{\dot{\sigma}}}}%
\mathsf{E}\left( I_{r},\widehat{{\dot{\sigma}}}\right) ^{2}\mathrm{P}\left(
I_{r},\mathbf{1}_{I_{0}}\widehat{\omega }\right) ^{2}\leq \left( \mathcal{E}%
^{\ast }\right) ^{2}\left\vert I_{0}\right\vert _{\widehat{\omega }}.
\label{dual energy}
\end{equation}%
We use the following estimate, which shows that with the energy factor
included, we obtain a strengthening of the $A_{2}$ condition.

\begin{proposition}
\label{p.gettingE}For any interval $I\subset \lbrack 0,1]$, we have the
inequality 
\begin{equation}
\left\vert I\right\vert _{\widehat{{\dot{\sigma}}}}\mathsf{E}\left( I,%
\widehat{{\dot{\sigma}}}\right) ^{2}\mathrm{P}\left( I,\widehat{\omega }%
\right) ^{2}\lesssim \left\vert I\right\vert _{\widehat{\omega }}.
\label{e.gettingE}
\end{equation}
\end{proposition}

\begin{proof}
We may assume that $\mathsf{E}(I;\sigma )\neq 0$. Let $k$ be the smallest
integer for which there is an $r$ with $\dot{z}_{r}^{k}\in I$. Let $n$ be
the smallest integer so that for some $s$ we have $\dot{z}_{s}^{k+n}\in I$
and $\dot{z}_{s}^{k+n}\neq \dot{z}_{r}^{k}$. We can estimate $\mathsf{E}(I;%
\widehat{{\dot{\sigma}}})$ in terms of $n$ by 
\begin{equation}
\mathsf{E}\left( I,\widehat{{\dot{\sigma}}}\right) ^{2}\lesssim \left( \frac{%
2}{N^{2}\left( 1-\eta \right) }\right) ^{n}\,.  \label{e.En}
\end{equation}%
Indeed, the worst case is when $s$ is not unique. Then there are two choices
of $s$ -- but not more. Let $\dot{z}_{s^{\prime }}^{k+n}\in I$, where $%
s^{\prime }\neq s$. Note that we then have 
\begin{equation*}
\frac{\lvert I-\left\{ \dot{z}_{r}^{k}\right\} \rvert _{\widehat{{\dot{\sigma%
}}}}}{\left\vert I\right\vert _{\widehat{{\dot{\sigma}}}}}\lesssim \frac{%
\max \left\{ \widehat{s}_{s}^{k+n},\widehat{s}_{s^{\prime }}^{k+n}\right\} }{%
\widehat{s}_{r}^{k}}\lesssim \left( \frac{2}{N^{2}\left( 1-\eta \right) }%
\right) ^{n},
\end{equation*}%
and this gives (\ref{e.En}) upon using the characterization of Energy as a
variance:%
\begin{equation*}
\mathsf{E}\left( I,\widehat{{\dot{\sigma}}}\right) ^{2}\leq \frac{1}{%
\left\vert I\right\vert _{\widehat{{\dot{\sigma}}}}}\int_{I}\left\vert \frac{%
x-\dot{z}_{r}^{k}}{\left\vert I\right\vert }\right\vert ^{2}\widehat{{\dot{%
\sigma}}}\left( x\right) \leq \frac{1}{\left\vert I\right\vert _{\widehat{{%
\dot{\sigma}}}}}\int_{I\setminus \left\{ \dot{z}_{r}^{k}\right\} }\widehat{{%
\dot{\sigma}}}\ .
\end{equation*}

Next we note from (\ref{facts for future}) that $\left\vert I\right\vert _{%
\widehat{{\dot{\sigma}}}}\approx \left\vert I_{r}^{k}\right\vert _{\widehat{{%
\dot{\sigma}}}}$ and $\left\vert I\right\vert \approx \left\vert
I_{r}^{k}\right\vert $ and $\frac{\left\vert I_{r}^{k}\right\vert _{\widehat{%
{\dot{\sigma}}}}}{\left\vert I_{r}^{k}\right\vert }\approx \frac{1}{\kappa
_{r}^{k}}\left( \frac{2}{N}\right) ^{k}$, $\frac{\left\vert I\right\vert _{%
\widehat{\omega }}}{\left\vert I_{m}^{k+n}\right\vert }\geq \frac{\left\vert
I_{m}^{k+n}\right\vert _{\widehat{\omega }}}{\left\vert
I_{m}^{k+n}\right\vert }\approx \kappa _{m}^{k+n}\left( \frac{N}{2}\right)
^{k+n}$ for some $m$ so that $I\supset I_{m}^{k+n}$, and $\mathrm{P}\left( I,%
\widehat{\omega }\right) \approx \mathrm{P}\left( I_{r}^{k},\widehat{\omega }%
\right) \simeq \kappa _{r}^{k}\left( \frac{N}{2}\right) ^{k}$ where $\kappa
_{r}^{\ell }\equiv \left( 1+\frac{1}{N}\right) ^{H\left( I_{r}^{\ell
}\right) }\left( 1-\frac{1}{N}\right) ^{T\left( I_{r}^{\ell }\right) }$.
Thus we have estimates for all of the factors on both sides of (\ref%
{e.gettingE}), and the following inequality is sufficient for (\ref%
{e.gettingE}):%
\begin{equation*}
\left[ \left\vert I_{r}^{k}\right\vert \frac{1}{\kappa _{r}^{k}}\left( \frac{%
2}{N}\right) ^{k}\right] \left[ \left( \frac{2}{N^{2}\left( 1-\eta \right) }%
\right) ^{n}\right] \left[ \left( \kappa _{r}^{k}\left( \frac{N}{2}\right)
^{k}\right) ^{2}\right] \lesssim \left\vert I_{m}^{k+n}\right\vert \kappa
_{m}^{k+n}\left( \frac{N}{2}\right) ^{k+n}.
\end{equation*}%
But this reduces to%
\begin{equation*}
\left( \frac{2}{N}\right) ^{2n}\left( \frac{1}{1-\eta }\right) ^{n}\lesssim 
\frac{\kappa _{m}^{k+n}}{\kappa _{r}^{k}}=\frac{\left( 1+\frac{1}{N}\right)
^{H\left( I_{m}^{k+n}\right) }\left( 1-\frac{1}{N}\right) ^{T\left(
I_{m}^{k+n}\right) }}{\left( 1+\frac{1}{N}\right) ^{H\left( I_{r}^{k}\right)
}\left( 1-\frac{1}{N}\right) ^{T\left( I_{r}^{k}\right) }},
\end{equation*}%
which is in fact\ true for all pairs of $n$ and $k$ since $I$ can only
intersect the interval $I_{r}^{k}$ among those in $\mathcal{D}$ at
generation $k$, and thus $I_{m}^{k+n}\subset I_{r}^{k}$, so that $T\left(
I_{m}^{k+n}\right) -T\left( I_{r}^{k}\right) \leq n$ and $H\left(
I_{m}^{k+n}\right) -H\left( I_{r}^{k}\right) \leq n$. Indeed, we then get 
\begin{equation*}
\frac{\left( 1+\frac{1}{N}\right) ^{H\left( I_{m}^{k+n}\right) }\left( 1-%
\frac{1}{N}\right) ^{T\left( I_{m}^{k+n}\right) }}{\left( 1+\frac{1}{N}%
\right) ^{H\left( I_{r}^{k}\right) }\left( 1-\frac{1}{N}\right) ^{T\left(
I_{r}^{k}\right) }}\geq \left( \frac{1-\frac{1}{N}}{1+\frac{1}{N}}\right)
^{n}=\left( \frac{N-1}{N+1}\right) ^{n}\gtrsim \left( \frac{4}{N^{2}\left(
1-\eta \right) }\right) ^{n},
\end{equation*}%
for all $n\geq 1$ provided $\frac{N-1}{N+1}\geq \frac{4}{N^{2}\left( 1-\eta
\right) }$, e.g. if $N\geq 4$.
\end{proof}

It is now clear that the pair of weights $\left( \widehat{{\dot{\sigma}}},%
\widehat{\omega }\right) $ satisfy the backward Energy Condition. Indeed,
let $I_{0}\subset \lbrack 0,1]$ and let $\{I_{r}\;:\;r\geq 1\}$ be any
partition of $I_{0}$. From (\ref{e.gettingE}) we have 
\begin{equation*}
\sum_{r\geq 1}\left\vert I_{r}\right\vert _{\widehat{{\dot{\sigma}}}}\mathsf{%
E}\left( I_{r},\widehat{{\dot{\sigma}}}\right) ^{2}\mathrm{P}\left( I_{r},%
\widehat{\omega }\right) ^{2}\lesssim \sum_{r\geq 1}\left\vert
I_{r}\right\vert _{\widehat{\omega }}=\left\vert I_{0}\right\vert _{\widehat{%
\omega }}\,.
\end{equation*}%
We will next see that the backward energy condition \textbf{fails} for the
weight pair $\left( \widehat{{\sigma }},\widehat{\omega }\right) $ in which $%
\widehat{{\sigma }}$ is no longer a sum of point masses.

\subsubsection{Failure of the backward Energy Condition for $\left( \protect%
\widehat{{\protect\sigma }},\protect\widehat{\protect\omega }\right) $\label%
{Subsubfails}}

We will choose a subdecomposition of $\left[ 0,1\right] $ in which the
backward Energy Condition for the weight pair $\left( \widehat{{\sigma }},%
\widehat{\omega }\right) $ is the same as the backward Pivotal Condition for
the weight pair $\left( \widehat{{\sigma }},\widehat{\omega }\right) $
because the measure $\widehat{\sigma }$ has energy essentially equal to the
`pivotal energy' of $\sigma $ on the intervals $G_{j}^{k}$, more precisely
because $\mathsf{E}\left( G_{j}^{k},\widehat{\sigma }\right) \approx \mathsf{%
E}\left( L_{j}^{k},\widehat{\sigma }\right) \approx 1$. Since the backward
Pivotal Condition fails for the weight pair $\left( \widehat{{\sigma }},%
\widehat{\omega }\right) $, it follows that the backward Energy Condition
fails as well for the weight pair $\left( \widehat{{\sigma }},\widehat{%
\omega }\right) $. Here are the details. We have from (\ref{facts for future}%
) that%
\begin{equation*}
\mathrm{P}\left( G_{r}^{\ell },\widehat{\omega }\right) \approx \mathrm{P}%
\left( I_{r}^{\ell },\widehat{\omega }\right) \approx \frac{\left\vert
I_{r}^{\ell }\right\vert _{\widehat{\omega }}}{\left\vert I_{r}^{\ell
}\right\vert },\ \ \ \ \ \text{for all }r.
\end{equation*}%
Thus for the decomposition $\overset{\cdot }{\bigcup }_{\ell ,r}G_{r}^{\ell
}\subset \left[ 0,1\right] $ we have%
\begin{eqnarray*}
&&\sum_{\ell ,r}\left\vert G_{r}^{\ell }\right\vert _{\widehat{{\sigma }}}%
\mathsf{E}\left( G_{r}^{\ell },\widehat{{\sigma }}\right) ^{2}\mathrm{P}%
\left( G_{r}^{\ell },\widehat{\omega }\right) ^{2}\approx \sum_{\ell
,r}\left\vert G_{r}^{\ell }\right\vert _{\widehat{\sigma }}\mathrm{P}\left(
G_{r}^{\ell },\widehat{\omega }\right) ^{2} \\
&\approx &\sum_{\ell =0}^{\infty }\sum_{r}\frac{\left\vert G_{r}^{\ell
}\right\vert _{\widehat{\sigma }}\left\vert I_{r}^{\ell }\right\vert _{%
\widehat{\omega }}}{\left\vert I_{r}^{\ell }\right\vert ^{2}}\left\vert
I_{r}^{\ell }\right\vert _{\widehat{\omega }}\mathbb{\approx }\sum_{\ell
=0}^{\infty }\sum_{r}\left\vert I_{r}^{\ell }\right\vert _{\widehat{\omega }%
}=\sum_{\ell =0}^{\infty }\left\vert \left[ 0,1\right] \right\vert _{%
\widehat{\omega }}=\infty .
\end{eqnarray*}

\subsubsection{The forward Energy Condition}

It remains to verify that the measure pair $\left( \widehat{\dot{\sigma}},%
\widehat{\omega }\right) $ satisfies the forward Energy Condition. We will
in fact establish the pivotal condition (\ref{pivotalcondition})%
\begin{equation*}
\sum_{r=1}^{\infty }\left\vert I_{r}\right\vert _{\widehat{\omega }}\mathrm{P%
}\left( I_{r},\mathbf{1}_{I_{0}}\widehat{\dot{\sigma}}\right) ^{2}\leq 
\mathcal{P}^{2}\left\vert I_{0}\right\vert _{\widehat{\dot{\sigma}}},
\end{equation*}%
which then implies that $\mathcal{E}\left( \widehat{\dot{\sigma}},\widehat{%
\omega }\right) <\infty $ since $\mathsf{E}\left( I_{r},\mathbf{1}_{I_{0}}%
\widehat{\dot{\sigma}}\right) $. For this it suffices to show that the
forward maximal inequality%
\begin{equation}
\int M\left( f\widehat{\dot{\sigma}}\right) ^{2}d\widehat{\omega }\leq C\int
\left\vert f\right\vert ^{2}d\widehat{\dot{\sigma}}  \label{M2weight}
\end{equation}%
holds for the pair $\left( \widehat{\dot{\sigma}},\widehat{\omega }\right) $%
. Indeed, as is well known, the Poisson integral is an $\ell ^{1}$ average
of a sequence of expanding $\widehat{\dot{\sigma}}$ averages, and thus it is
clearly dominated by the supremum of these averages, and in particular,%
\begin{equation*}
\mathrm{P}\left( I_{r},\mathbf{1}_{I_{0}}\widehat{\dot{\sigma}}\right)
\lesssim \inf_{x\in I_{r}}M\left( \mathbf{1}_{I_{0}}\widehat{\dot{\sigma}}%
\right) .
\end{equation*}%
Then we have%
\begin{equation*}
\sum_{r=1}^{\infty }\left\vert I_{r}\right\vert _{\widehat{\omega }}\mathrm{P%
}\left( I_{r},\mathbf{1}_{I_{0}}\widehat{\dot{\sigma}}\right) ^{2}\lesssim
\sum_{r=1}^{\infty }\left\vert I_{r}\right\vert _{\widehat{\omega }%
}\inf_{x\in I_{r}}M\left( \mathbf{1}_{I_{0}}\widehat{\dot{\sigma}}\right)
^{2}\leq \int M\left( \mathbf{1}_{I_{0}}\widehat{\dot{\sigma}}\right) ^{2}d%
\widehat{\omega }\leq C\int \mathbf{1}_{I_{0}}d\widehat{\dot{\sigma}}\ ,
\end{equation*}%
by (\ref{M2weight}) with $f=\mathbf{1}_{I_{0}}$.

Now (\ref{M2weight}) in turn follows from the testing condition%
\begin{equation}
\int_{Q}M\left( \mathbf{1}_{Q}\widehat{\dot{\sigma}}\right) ^{2}d\widehat{%
\omega }\leq C\int_{Q}d\widehat{\dot{\sigma}},  \label{testmax}
\end{equation}%
for all intervals $Q$ (see \cite{MR676801}). We will show (\ref{testmax})
when $Q=I_{r}^{\ell }$, the remaining cases being an easy consequence of
this one. Indeed, let $\dot{z}_{r}^{\ell }$ be the point in $Q$ with
smallest $\ell $. Inequality (\ref{testmax}) then remains unaffected by
replacing $Q$ with $Q\cap I_{r}^{\ell }$. Next we increase $Q\cap
I_{r}^{\ell }$ to $I_{r}^{\ell }$ and note the left hand side of (\ref%
{testmax}) increases while the right hand side remains comparable. Thus we
have reduced matters to checking (\ref{testmax}) when $Q=I_{r}^{\ell }$.

For the case $Q=I_{r}^{\ell }$ of (\ref{testmax}), we use the fact from (\ref%
{facts for future}) that%
\begin{equation}
\mathcal{M}\left( \mathbf{1}_{I_{r}^{\ell }}\widehat{\dot{\sigma}}\right)
\left( x\right) \leq C\frac{1}{\kappa _{r}^{\ell }}\left( \frac{2}{N}\right)
^{\ell },\ \ \ \ \ x\in \mathsf{E}^{\left( N\right) }\cap I_{r}^{\ell },
\label{Msigma bounded}
\end{equation}%
where $\kappa _{r}^{\ell }\equiv \left( 1+\frac{1}{N}\right) ^{H\left(
I_{r}^{\ell }\right) }\left( 1-\frac{1}{N}\right) ^{T\left( I_{r}^{\ell
}\right) }$ as in (\ref{def kappa}). To see (\ref{Msigma bounded}), note
that for each $x\in I_{r}^{\ell }$ that also lies in the Cantor set $\mathsf{%
E}^{\left( N\right) }$, we have%
\begin{equation*}
\mathcal{M}\left( \mathbf{1}_{I_{r}^{\ell }}\widehat{\dot{\sigma}}\right)
\left( x\right) \lesssim \sup_{\left( k,j\right) :x\in I_{j}^{k}}\frac{1}{%
\left\vert I_{j}^{k}\right\vert }\int_{I_{j}^{k}\cap I_{r}^{\ell }}d\widehat{%
\dot{\sigma}}\approx \sup_{\left( k,j\right) :x\in I_{j}^{k}}\frac{1}{\kappa
_{j}^{k}}\frac{\left( \frac{1}{N}\right) ^{k\vee \ell }\left( \frac{2}{N}%
\right) ^{k\vee \ell }}{\left( \frac{1}{N}\right) ^{k}}\approx \frac{1}{%
\kappa _{r}^{\ell }}\left( \frac{2}{N}\right) ^{\ell }.
\end{equation*}%
Now we consider for each fixed $m$, the approximations $\widehat{\omega }%
^{\left( m\right) }$ and $\widehat{\dot{\sigma}}^{\left( m\right) }$ to the
measures $\widehat{\omega }$ and $\widehat{\dot{\sigma}}$ given in (\ref%
{approximations}). For these approximations we have in the same way the
estimate%
\begin{equation*}
\mathcal{M}\left( \mathbf{1}_{I_{r}^{\ell }}\widehat{\dot{\sigma}}^{\left(
m\right) }\right) \left( x\right) \leq C\frac{1}{\kappa _{r}^{\ell }}\left( 
\frac{2}{N}\right) ^{\ell },\ \ \ \ \ x\in \bigcup_{i=1}^{2^{m}}I_{i}^{m}.
\end{equation*}%
Thus for each $m\geq 1$ we have%
\begin{eqnarray*}
\int_{I_{r}^{\ell }}\mathcal{M}\left( \mathbf{1}_{I_{r}^{\ell }}\widehat{%
\dot{\sigma}}^{\left( m\right) }\right) ^{2}d\widehat{\omega }^{\left(
m\right) } &\leq &C\sum_{i:I_{i}^{m}\subset I_{r}^{\ell }}\left( \frac{1}{%
\kappa _{r}^{\ell }}\right) ^{2}\left( \frac{2}{N}\right) ^{2\ell }\kappa
_{i}^{m}2^{-m} \\
&\leq &C2^{m-\ell }\frac{1}{\kappa _{r}^{\ell }}\left( \frac{2}{N}\right)
^{2\ell }2^{-m}=C\widehat{s}_{r}^{\ell }\approx C\int_{I_{r}^{\ell }}d%
\widehat{\dot{\sigma}}.
\end{eqnarray*}

Taking the limit as $m\rightarrow \infty $ yields the case $Q=I_{r}^{\ell }$
of (\ref{testmax}). Indeed, from the Monotone Convergence Theorem we have%
\begin{equation*}
\int_{I_{r}^{\ell }}\mathcal{M}\left( \mathbf{1}_{I_{r}^{\ell }}\widehat{%
\dot{\sigma}}\right) ^{2}d\widehat{\omega }=\lim_{n\rightarrow \infty
}\int_{I_{r}^{\ell }}\mathcal{M}\left( \mathbf{1}_{I_{r}^{\ell }}\widehat{%
\dot{\sigma}}^{\left( n\right) }\right) ^{2}d\widehat{\omega }.
\end{equation*}%
From the continuity of $\mathcal{M}\left( \mathbf{1}_{I_{r}^{\ell }}\widehat{%
\dot{\sigma}}^{\left( n\right) }\right) ^{2}$ on the support of $d\widehat{%
\omega }^{\left( 2n\right) }$, and the fact that the supports of $\widehat{%
\omega }^{\left( m\right) }$ are decreasing, we have%
\begin{equation*}
\int_{I_{r}^{\ell }}\mathcal{M}\left( \mathbf{1}_{I_{r}^{\ell }}\widehat{%
\dot{\sigma}}^{\left( n\right) }\right) ^{2}d\widehat{\omega }%
=\lim_{m\rightarrow \infty }\int_{I_{r}^{\ell }}\mathcal{M}\left( \mathbf{1}%
_{I_{r}^{\ell }}\widehat{\dot{\sigma}}^{\left( n\right) }\right) ^{2}d%
\widehat{\omega }^{\left( m\right) }.
\end{equation*}%
For $m\geq n$, we have 
\begin{equation*}
\int_{I_{r}^{\ell }}\mathcal{M}\left( \mathbf{1}_{I_{r}^{\ell }}\widehat{%
\dot{\sigma}}^{\left( n\right) }\right) ^{2}d\widehat{\omega }^{\left(
m\right) }\leq \int_{I_{r}^{\ell }}\mathcal{M}\left( \mathbf{1}_{I_{r}^{\ell
}}\widehat{\dot{\sigma}}^{\left( m\right) }\right) ^{2}d\widehat{\omega }%
^{\left( m\right) }\leq C\int_{I_{r}^{\ell }}d\widehat{\dot{\sigma}}
\end{equation*}%
by monotonicity, and so altogether%
\begin{equation*}
\int_{I_{r}^{\ell }}\mathcal{M}\left( \mathbf{1}_{I_{r}^{\ell }}\widehat{%
\dot{\sigma}}\right) ^{2}d\widehat{\omega }=\lim_{n\rightarrow \infty
}\lim_{m\rightarrow \infty }\int_{I_{r}^{\ell }}\mathcal{M}\left( \mathbf{1}%
_{I_{r}^{\ell }}\widehat{\dot{\sigma}}^{\left( n\right) }\right) ^{2}d%
\widehat{\omega }^{\left( m\right) }\leq C\int_{I_{r}^{\ell }}d\widehat{\dot{%
\sigma}}\ .
\end{equation*}%
This completes our proof of the pivotal condition, and hence also the
forward Energy Condition for the weight pair $\left( \widehat{\dot{\sigma}},%
\widehat{\omega }\right) $.

\subsection{The norm inequality}

Here we show that the norm inequality for $H_{\flat }$ holds with respect to
the weight pair $\left( \widehat{\sigma },\widehat{\omega }\right) $. We
first observe that we have already established above the following facts for
the other weight pair $\left( \widehat{\dot{\sigma}},\widehat{\omega }%
\right) $.

\begin{enumerate}
\item The Muckenhoupt/NTV condition $\mathcal{A}_{2}$ holds:%
\begin{equation*}
\sup_{I}\mathrm{P}(I,\widehat{\omega })\cdot \mathrm{P}(I,\widehat{\dot{%
\sigma}})=\mathcal{A}_{2}<\infty .
\end{equation*}

\item The forward testing condition holds:%
\begin{equation*}
\int_{I}\left\vert H_{\flat }\left( \mathbf{1}_{I}\widehat{\dot{\sigma}}%
\right) \right\vert ^{2}d\widehat{\omega }\lesssim \left\vert I\right\vert _{%
\widehat{\dot{\sigma}}}\ .
\end{equation*}

\item The backward testing condition holds:%
\begin{equation*}
\int_{I}\left\vert H_{\flat }\left( \mathbf{1}_{I}\widehat{\omega }\right)
\right\vert ^{2}d\widehat{\dot{\sigma}}\lesssim \left\vert I\right\vert _{%
\widehat{\omega }}\ .
\end{equation*}

\item The weak boundedness property holds:%
\begin{equation*}
\left\vert \int_{J}H_{\flat }\left( \mathbf{1}_{I}\widehat{{\dot{\sigma}}}%
\right) d\widehat{\omega }\right\vert \lesssim \sqrt{\left\vert J\right\vert
_{\widehat{\omega }}\left\vert I\right\vert _{\widehat{{\dot{\sigma}}}}}\ .
\end{equation*}

\item The forward energy condition holds:%
\begin{equation*}
\sum_{J}\left( \frac{\mathrm{P}\left( J,\mathbf{1}_{I}\widehat{\dot{\sigma}}%
\right) }{\left\vert J\right\vert }\right) ^{2}\left\Vert \mathsf{P}_{J}^{%
\widehat{\omega }}x\right\Vert _{L^{2}\left( \widehat{\omega }\right)
}^{2}\lesssim \left\vert I\right\vert _{\widehat{\dot{\sigma}}}\ .
\end{equation*}

\item The backward energy condition holds:%
\begin{equation*}
\sum_{J}\left( \frac{\mathrm{P}\left( J,\mathbf{1}_{I}\widehat{\omega }%
\right) }{\left\vert J\right\vert }\right) ^{2}\left\Vert \mathsf{P}_{J}^{%
\widehat{\dot{\sigma}}}x\right\Vert _{L^{2}\left( \widehat{\dot{\sigma}}%
\right) }^{2}\lesssim \left\vert I\right\vert _{\widehat{\omega }}\ .
\end{equation*}
\end{enumerate}

Now we can apply our $T1$ theorem with an energy side condition\ in \cite%
{SaShUr7} (or see \cite{SaShUr6} or \cite{SaShUr9}) to obtain the dual norm
inequality%
\begin{equation*}
\int \left\vert H_{\flat }\left( g\widehat{\omega }\right) \right\vert ^{2}d%
\widehat{\dot{\sigma}}\lesssim \int \left\vert g\right\vert ^{2}d\widehat{%
\omega }.
\end{equation*}%
Consider the weight pair $\left( \widehat{\sigma },\widehat{\omega }\right) $%
. Since $H_{\flat }\left( g\widehat{\omega }\right) $ is constant on each
interval $L_{j}^{k}$, we see that $\int \left\vert H_{\flat }\left( g%
\widehat{\omega }\right) \right\vert ^{2}d\widehat{\sigma }=\int \left\vert
H_{\flat }\left( g\widehat{\omega }\right) \right\vert ^{2}d\widehat{\dot{%
\sigma}}$, and hence the dual norm inequality for $H_{\flat }$ holds with
respect to the weight pair $\left( \widehat{\sigma },\widehat{\omega }%
\right) $.

Thus we have just shown that the norm inequality for the elliptic singular
integral $H_{\flat }$ holds with respect to the weight pair $\left( \widehat{%
\sigma },\widehat{\omega }\right) $, and in Subsubsection \ref{Subsubfails}
we showed that the backward energy condition fails for $\left( \widehat{%
\sigma },\widehat{\omega }\right) $. This completes the proof of Theorem \ref%
{energy condition fails}.

\section{Energy reversal and the $T1$ theorem}

Here we prove Theorem \ref{gradient elliptic} by first establishing reversal
of energy and the necessity of the energy conditions, and then applying the
main result from \cite{SaShUr9} or \cite{SaShUr10}. For the reader's
benefit, we recall the relevant $1$-dimensional version of Theorem 2 in \cite%
{SaShUr10} (note that our energy conditions $\mathcal{E}+\mathcal{E}^{\ast
}<\infty $ imply the strong energy conditions $\mathcal{E}^{\limfunc{strong}%
}+\mathcal{E}^{\ast ,\limfunc{strong}}<\infty $ assumed in Theorem 2 in \cite%
{SaShUr10}).

\begin{theorem}
\label{one dim theorem}Suppose that $T$ is a standard singular integral
operator on $\mathbb{R}$, and that $\omega $ and $\sigma $ are locally
finite positive Borel measures on $\mathbb{R}$. Set $T_{\sigma }f=T\left(
f\sigma \right) $ for any smooth truncation of $T_{\sigma }$. Then the
operator $T_{\sigma }$ is bounded from $L^{2}\left( \sigma \right) $ to $%
L^{2}\left( \omega \right) $, i.e. 
\begin{equation*}
\left\Vert T_{\sigma }f\right\Vert _{L^{2}\left( \omega \right) }\leq 
\mathfrak{N}_{T_{\sigma }}\left\Vert f\right\Vert _{L^{2}\left( \sigma
\right) },
\end{equation*}%
uniformly in smooth truncations of $T$, and moreover%
\begin{equation*}
\mathfrak{N}_{T_{\sigma }}\leq C\left( \sqrt{\mathfrak{A}_{2}}+\mathfrak{T}%
_{T}+\mathfrak{T}_{T}^{\ast }+\mathcal{E}+\mathcal{E}^{\ast }\right) ,
\end{equation*}%
provided that the four Muckenhoupt conditions hold, i.e. $\mathfrak{A}%
_{2}<\infty $, and the two dual testing conditions (\ref{T testing}) for $T$
hold, and provided that the two dual energy conditions (\ref{energy
condition}) and (\ref{dual energy condition}) hold.
\end{theorem}

So suppose our kernel $K\left( x,y\right) $ is gradient elliptic, i.e.
satisfies 
\begin{equation*}
\frac{\partial }{\partial x}K\left( x,y\right) ,-\frac{\partial }{\partial y}%
K\left( x,y\right) \geq \frac{c}{\left( x-y\right) ^{2}}
\end{equation*}%
Then for a positive measure $\mu $ supported outside an interval $J$, and
for $x,z\in J$ with $x>z$, we have%
\begin{eqnarray*}
T\mu \left( x\right) -T\mu \left( z\right) &=&\int_{\mathbb{R}}\left[
K\left( x,y\right) -K\left( z,y\right) \right] d\mu \left( y\right) =\int_{%
\mathbb{R}}\left[ \int_{z}^{x}\frac{\partial }{\partial x}K\left( t,y\right)
dt\right] d\mu \left( y\right) \\
&\geq &\int_{\mathbb{R}}\left[ \int_{z}^{x}\frac{c}{\left( t-y\right) ^{2}}dt%
\right] d\mu \left( y\right) =c\left( x-z\right) \int_{\mathbb{R}}\frac{d\mu
\left( y\right) }{\left( x-y\right) \left( z-y\right) } \\
&\geq &\frac{c}{4}\left( x-z\right) \int_{\mathbb{R}}\frac{d\mu \left(
y\right) }{\left( c_{J}-y\right) ^{2}}\approx \frac{c}{4}\left( x-z\right) 
\frac{\mathsf{P}\left( J,\mu \right) }{\left\vert J\right\vert },
\end{eqnarray*}%
and hence%
\begin{eqnarray*}
&&\int_{J}\int_{J}\left\vert T\mu \left( x\right) -T\mu \left( z\right)
\right\vert ^{2}d\omega \left( x\right) d\omega \left( z\right) \\
&\geq &\frac{c^{2}}{16}\left( \frac{\mathsf{P}\left( J,\mu \right) }{%
\left\vert J\right\vert }\right) ^{2}\int_{J}\int_{J}\left\vert
x-z\right\vert ^{2}d\omega \left( x\right) d\omega \left( z\right) \\
&=&\frac{c^{2}}{8}\left( \frac{\mathsf{P}\left( J,\mu \right) }{\left\vert
J\right\vert }\right) ^{2}\left\vert J\right\vert _{\omega
}\int_{J}\left\vert x-\mathbb{E}_{L}^{\omega }x\right\vert ^{2}d\omega
\left( x\right) \\
&=&\frac{c^{2}}{8}\left( \frac{\mathsf{P}\left( J,\mu \right) }{\left\vert
J\right\vert }\right) ^{2}\left\vert J\right\vert _{\omega }\left\Vert 
\mathsf{P}_{J}^{\omega }\mathbf{x}\right\Vert _{L^{2}\left( \omega \right)
}^{2}\ .
\end{eqnarray*}%
It now follows, using the definition of the punctured Muckenhoupt condition $%
A_{2}^{\limfunc{punct}}$ from \cite{SaShUr9}, that if $\overset{\cdot }{%
\dbigcup }J_{n}\subset I$, then%
\begin{eqnarray*}
&&\sum_{n}\left( \frac{\mathsf{P}\left( J_{n},\mathbf{1}_{I}\mu \right) }{%
\left\vert J_{n}\right\vert }\right) ^{2}\left\Vert \mathsf{P}%
_{J_{n}}^{\omega }\mathbf{x}\right\Vert _{L^{2}\left( \omega \right) }^{2} \\
&\lesssim &\sum_{n}\left( \frac{\mathsf{P}\left( J_{n},\mathbf{1}_{J_{n}}\mu
\right) }{\left\vert J_{n}\right\vert }\right) ^{2}\left\Vert \mathsf{P}%
_{J_{n}}^{\omega }\mathbf{x}\right\Vert _{L^{2}\left( \omega \right)
}^{2}+\sum_{n}\left( \frac{\mathsf{P}\left( J_{n},\mathbf{1}_{I\setminus
J_{n}}\mu \right) }{\left\vert J_{n}\right\vert }\right) ^{2}\left\Vert 
\mathsf{P}_{J_{n}}^{\omega }\mathbf{x}\right\Vert _{L^{2}\left( \omega
\right) }^{2} \\
&\lesssim &A_{2}^{\limfunc{punct}}\sum_{n}\left\vert J_{n}\right\vert _{\mu
}+\sum_{n}\frac{1}{\left\vert J_{n}\right\vert _{\omega }}%
\int_{J_{n}}\int_{J_{n}}\left\vert T\left( \mathbf{1}_{I\setminus J_{n}}\mu
\right) \left( x\right) -T\left( \mathbf{1}_{I\setminus J_{n}}\mu \right)
\left( z\right) \right\vert ^{2}d\omega \left( x\right) d\omega \left(
z\right) \\
&\lesssim &A_{2}^{\limfunc{punct}}\left\vert I\right\vert _{\mu
}+\sum_{n}\int_{J_{n}}\left\vert T\left( \mathbf{1}_{I\setminus J_{n}}\mu
\right) \left( x\right) \right\vert ^{2}d\omega \left( x\right) \\
&\lesssim &A_{2}^{\limfunc{punct}}\left\vert I\right\vert _{\mu
}+\sum_{n}\int_{J_{n}}\left\vert T\left( \mathbf{1}_{J_{n}}\mu \right)
\left( x\right) \right\vert ^{2}d\omega \left( x\right)
+\sum_{n}\int_{J_{n}}\left\vert T\left( \mathbf{1}_{I}\mu \right) \left(
x\right) \right\vert ^{2}d\omega \left( x\right) \\
&\lesssim &A_{2}^{\limfunc{punct}}\left\vert I\right\vert _{\mu }+\sum_{n}%
\mathfrak{T}\left\vert J_{n}\right\vert _{\mu }+\mathfrak{T}\left\vert
I\right\vert _{\mu }\lesssim \left( A_{2}^{\limfunc{punct}}+\mathfrak{T}%
\right) \left\vert I\right\vert _{\mu }\ .
\end{eqnarray*}%
This shows that the energy conditions are controlled by the punctured
Muckenhoupt conditions and the testing conditions, and now an application of
Theorem \ref{one dim theorem} completes the proof of Theorem \ref{gradient
elliptic}.

\end{document}